\documentclass[11pt]{article} 
\usepackage{hyperref}
\usepackage{amssymb}
\usepackage{amsmath}
 \usepackage{color}

\usepackage{amsmath}
\usepackage{latexsym}
\usepackage{amssymb}

\usepackage{amsthm}

\font\tengoth=eufm10 at 10pt
\font\sevengoth=eufm7 at 6pt
\newfam\gothfam
\textfont\gothfam=\tengoth
\scriptfont\gothfam=\sevengoth

\newcommand{\mlabel}[1]{\marginpar{#1}\label{#1}}

\newcommand{\g}{{\mathfrak g}}

\newcommand{\fg}{{\mathfrak g}}

\renewcommand{\:}{\colon}
\newcommand{\1}{\mathbf{1}}

\newcommand{\cD}{\mathcal{D}}
\newcommand{\cE}{\mathcal{E}}
\newcommand{\cF}{\mathcal{F}}

\newcommand{\cH}{\mathcal{H}}
\newcommand{\cK}{\mathcal{K}}
\newcommand{\cL}{\mathcal{L}}
\newcommand{\cM}{\mathcal{M}}
\newcommand{\cN}{\mathcal{N}}
\newcommand{\cO}{\mathcal{O}}

\newcommand{\cR}{\mathcal{R}}
\newcommand{\cS}{\mathcal{S}}
\newcommand{\cT}{\mathcal{T}}

\newcommand{\cV}{\mathcal{V}}

\newcommand\bx{{\bf{x}}}
\newcommand\by{{\bf{y}}}

\newcommand{\bO}{\mathbf{O}}

\newcommand{\eset}{\emptyset}

\newcommand{\derat}[1]{\frac{d}{dt}\big\vert_{t = #1}}

\newcommand{\subeq}{\subseteq}
\newcommand{\supeq}{\supseteq}

\newcommand{\into}{\hookrightarrow}
\newcommand{\eps}{\varepsilon}

\newcommand{\shalf}{{\textstyle{\frac{1}{2}}}}

\newcommand{\N}{{\mathbb N}}
\newcommand{\Z}{{\mathbb Z}}
\newcommand{\R}{{\mathbb R}}
\newcommand{\C}{{\mathbb C}}

\newcommand{\K}{{\mathbb K}}

\renewcommand{\H}{{\mathbb H}}

\newcommand{\bS}{{\mathbb S}}

\renewcommand{\hat}{\widehat}

\renewcommand{\tilde}{\widetilde}

\newcommand{\Aff}{\mathop{{\rm Aff}}\nolimits}

\newcommand{\GL}{\mathop{{\rm GL}}\nolimits}

\newcommand{\PGL}{\mathop{{\rm PGL}}\nolimits}

\newcommand{\AU}{\mathop{{\rm AU}}\nolimits}

\newcommand{\SO}{\mathop{{\rm SO}}\nolimits}

\newcommand{\OO}{\mathop{\rm O{}}\nolimits}

\newcommand{\U}{\mathop{\rm U{}}\nolimits}

\newcommand{\Sp}{\mathop{{\rm Sp}}\nolimits}
\newcommand{\Sym}{\mathop{{\rm Sym}}\nolimits}

\newcommand{\fsl} {\mathop{{\mathfrak{sl} }}\nolimits}

\newcommand{\so}  {\mathop{{\mathfrak{so} }}\nolimits}

\newcommand{\Exp}{\mathop{{\rm Exp}}\nolimits}
\newcommand{\Fix}{\mathop{{\rm Fix}}\nolimits}

\newcommand{\ad}{\mathop{{\rm ad}}\nolimits}

\renewcommand{\Re}{\mathop{{\rm Re}}\nolimits}
\renewcommand{\Im}{\mathop{{\rm Im}}\nolimits}

\newcommand{\tr}{\mathop{{\rm tr}}\nolimits}

\newcommand{\Hol}{\mathop{{\rm Hol}}\nolimits}

\newcommand{\Herm}{\mathop{{\rm Herm}}\nolimits}

\newcommand{\Mp}{\mathop{\rm Mp{}}\nolimits}

\newcommand{\Aut}{\mathop{{\rm Aut}}\nolimits}

\newcommand{\diag}{\mathop{{\rm diag}}\nolimits}
\newcommand{\End}{\mathop{{\rm End}}\nolimits}
\newcommand{\id}{\mathop{{\rm id}}\nolimits}
\newcommand{\rk}{\mathop{{\rm rank}}\nolimits}

\renewcommand{\dim}{\mathop{{\rm dim}}\nolimits}

\newcommand{\supp}{\mathop{{\rm supp}}\nolimits}

\newcommand{\sgn}{\mathop{{\rm sgn}}\nolimits}

\newcommand{\ev}{\mathop{{\rm ev}}\nolimits}

\newcommand{\Bil}{\mathop{{\rm Bil}}\nolimits}

\newcommand{\Str}{\mathop{{\rm Str}}\nolimits}

\newcommand{\PSO}{\mathop{{\rm PSO}}\nolimits}

\newcommand{\Rarrow}{\Rightarrow}
\newcommand{\nin}{\noindent} 
\newcommand{\oline}{\overline}

\newcommand{\la}{\langle}
\newcommand{\ra}{\rangle}

\newcommand{\res}{\vert}

\newcommand{\spann}{{\rm span}}

\newcommand{\Spec}{{\rm Spec}}

\newcommand{\ssssarr}{\hbox to 15pt{\rightarrowfill}}
\newcommand{\sssarr}{\hbox to 20pt{\rightarrowfill}}
\newcommand{\ssarr}{\hbox to 30pt{\rightarrowfill}}
\newcommand{\sarr}{\hbox to 40pt{\rightarrowfill}}
\newcommand{\arr}{\hbox to 60pt{\rightarrowfill}}
\newcommand{\sssslarr}{\hbox to 15pt{\leftarrowfill}}
\newcommand{\ssslarr}{\hbox to 20pt{\leftarrowfill}}
\newcommand{\sslarr}{\hbox to 30pt{\leftarrowfill}}
\newcommand{\slarr}{\hbox to 40pt{\leftarrowfill}}
\newcommand{\larr}{\hbox to 60pt{\leftarrowfill}}

\newcommand{\Arr}{\hbox to 80pt{\rightarrowfill}}

\def\theoremname{Theorem}
\def\propositionname{Proposition}
\def\corollaryname{Corollary}
\def\lemmaname{Lemma}
\def\remarkname{Remark}
\def\conjecturename{Conjecture} 

\def\definitionname{Definition}
\def\exercisename{Exercise}
\def\examplename{Example}
\def\examplesname{Examples}
\def\problemname{Problem}
\def\problemsname{Problems}

\def\satzname{Satz} 
\def\koroname{Korollar}
\def\folgname{Folgerung}
\def\bemerkname{Bemerkung}
\def\aufgname{Aufgabe}

\def\beisname{Beispiel}
\def\beissname{Beispiele}
\def\bewname{Beweis}

\def\@thmcounter#1{\noexpand\arabic{#1}}
\def\@thmcountersep{}
\def\@begintheorem#1#2{\it \trivlist \item[\hskip 
\labelsep{\bf #1\ #2.\quad}]}
\def\@opargbegintheorem#1#2#3{\it \trivlist
      \item[\hskip \labelsep{\bf #1\ #2.\quad{\rm #3}}]}
\makeatother
\newtheorem{theor}{\theoremname}[section]
\newtheorem{propo}[theor]{\propositionname}
\newtheorem{coro}[theor]{\corollaryname}
\newtheorem{lemm}[theor]{\lemmaname}

\newenvironment{thm}{\begin{theor}\it}{\end{theor}}
\newenvironment{theorem}{\begin{theor}\it}{\end{theor}}

\newenvironment{prop}{\begin{propo}\it}{\end{propo}}

\newenvironment{cor}{\begin{coro}\it}{\end{coro}}

\newenvironment{lem}{\begin{lemm}\it}{\end{lemm}}
\newenvironment{lemma}{\begin{lemm}\it}{\end{lemm}}

\newtheorem{rema}[theor]{\remarkname}

\newenvironment{rem}{\begin{rema}\rm}{\end{rema}}

\newtheorem{stepnow}[theor]{}

\newtheorem{defin}[theor]{\definitionname} 

\newenvironment{definition}{\begin{defin}\rm}{\end{defin}}
\newenvironment{defn}{\begin{defin}\rm}{\end{defin}}

\newtheorem{exerc}{\exercisename}[section]

\newtheorem{exa}[theor]{\examplename}

\newenvironment{ex}{\begin{exa}\rm}{\end{exa}}

\newtheorem{exas}[theor]{\examplesname}

\newenvironment{exs}{\begin{exas}\rm}{\end{exas}}

\newtheorem{conj}[theor]{\conjecturename}

\newtheorem{pro}[theor]{\problemname}

\newtheorem{prs}[theor]{\problemsname}

\newtheorem{aufg}{\aufgname}[section]

\newenvironment{prf}{\begin{proof}}{\end{proof}}
   
{\hfill\qed\end{trivlist}}

\newenvironment{beweis*}{\begin{trivlist}\item[\hskip%
\labelsep{\bf\bewname.\quad}]}%
{\end{trivlist}}

\newtheorem{satzn}[theor]{\satzname}

\newtheorem{koro}[theor]{\koroname}

\newtheorem{folg}[theor]{\folgname}

\newtheorem{bem}[theor]{\bemerkname}

\newtheorem{aufgn}[theor]{\aufgname}

\newtheorem{beis}[theor]{\beisname}

\newtheorem{beiss}[theor]{\beissname}

\newcommand{\Cay}{\mathop{{\rm Cay}}\nolimits}
\newcommand{\sH}{{\sf H}}
\newcommand{\sK}{{\sf K}}
\newcommand{\sV}{{\tt V}}
\newcommand{\sW}{{\tt W}}

\newcommand{\ind}{\mathop{{\rm ind}}\nolimits}

\newcommand{\PSU}{\mathop{{\rm PSU}}\nolimits}

\renewcommand{\Hol}{\mathrm{Hol}}

\newcommand\bk{{\bf{k}}}

\renewcommand{\bO}{\mathbb O}

\newcommand{\rank}{\mathop{{\rm rank}}\nolimits}

\renewcommand{\phi}{\varphi}

\newcommand{\dtZ}{\left.\dfrac{d}{dt}\right|_{t=0}}
 
\newcommand{\ip}[2]{\la #1,#2 \ra}

\newcommand{\btau}{\oline{\tau}}
\numberwithin{equation}{section}

\renewcommand{\Hol}{\mathrm{Hol}}
\renewcommand{\mlabel}{\label}

\begin{document}

\title{Standard subspaces of Hilbert spaces\\
 of holomorphic functions on tube domains }
  \author{Karl-Hermann Neeb, Bent {\O}rsted, Gestur \'Olafsson\thanks{The research of K.-H. Neeb was partially
supported by DFG-grant NE 413/9-1. The research of G. \'Olafsson was partially supported by Simons grant 586106.}} 
 
 \date{}
\maketitle

\begin{abstract} 
In this article we study standard subspaces of 
Hilbert spaces of vector-valued 
holomorphic functions on tube domains 
$E + i C^0$, where $C \subeq E$ is a pointed generating cone invariant 
under $e^{\R h}$ for some endomorphism $h \in \End(E)$, 
diagonalizable with the eigenvalues $1,0,-1$ (generalizing a Lorentz boost).
This data specifies a wedge domain 
$W(E,C,h) \subeq E$  and one of our main results exhibits 
corresponding standard subspaces as being generated using 
test functions on these domains.
We also investigate aspects of reflection 
positivity for the triple $(E,C,e^{\pi i h})$ 
and the support properties of distributions on $E$, arising as Fourier transforms of 
operator-valued measures defining the Hilbert spaces~$\cH$. 
For the imaginary part of these distributions, we find 
similarities to the well known Huygens' principle, 
relating to wedge duality in the Minkowski context.
Interesting examples 
are the Riesz distributions associated to euclidean Jordan algebras. 
\end{abstract}

\section{Introduction} 
\mlabel{sec:1}

In mathematical physics and the theory of algebras of local 
observables, three fundamental concepts play a key role, namely  
(1) modular groups and conjugations (based on the 
Tomita-Takesaki Theorem); 
(2) KMS states of operator algebras, signifying the holomorphic
nature of certain symmetries; and (3)  Reflection Positivity, 
which classically connects euclidean and relativistic quantum 
field theories.
In this paper, we shall revisit and elucidate some of the
interesting connections between these three topics; in particular we shall give in a general
framework of Hilbert spaces of holomorphic functions (already a ubiquitous category)
the crucial construction of the relevant spaces in (3), and also make explicit the connections to
(1) and (2). While the technical details are connected with tube domains, some of the methods
are rather robust, and could be of interest in other geometric situations as well.

We start by reviewing some ideas stemming from physics, in particular quantum field theory and
the Tomita-Takesaki theory. 
A closed real subspace $\sV$ of a complex Hilbert space $\cH$ 
is called {\it standard} 
if $\sV \cap i \sV = \{0\}$ and $\sV + i \sV$ is dense in $\cH$ 
(\cite{Lo08}). A central  goal of this paper is 
to provide explicit descriptions of standard subspaces in 
Hilbert spaces of holomorphic functions on tube domains. 

Standard subspaces arise naturally in the modular theory of 
von Neumann algebras. If $\cM \subeq B(\cH)$ is a von Neumann algebra 
and $\Omega \in \cH$ is a cyclic separating vector for $\cM$, 
i.e., $\cM\Omega$ is dense in $\cH$ and the map 
$\cM \to \cH, m\mapsto m\Omega$, is injective, then 
\[ \sV_\cM := \oline{ \{ m\Omega \: m^* = m, m \in \cM \}} \] 
is a standard subspace of $\cH$. Conversely, 
one can use the functorial process provided by Second Quantization 
(\cite{Si74}) to associate to each standard subspace $\sV \subeq \cH$   
a von Neumann algebra $\cR_\pm(\sV)$ 
on the bosonic/fermionic Fock space $\cF_\pm(\cH)$, for which the 
vacuum vector $\Omega$ is cyclic and separating 
(see \cite[\S\S 4,6]{NO17} and \cite{Lo08} for details).  
This method has been developed by Araki and Woods  in the context of 
free bosonic quantum fields (\cite{AW63}); 
some of the corresponding fermionic results are more recent 
(cf.\ \cite{EO73}, \cite{Fo83}, \cite{BJL02}). 
This establishes a direct 
connection between standard subspaces and pairs 
$(\cM,\Omega)$ of von Neumann algebras with cyclic separating vectors. 
As such pairs play an important role in 
Algebraic Quantum Field Theory (AQFT) 
in the context of Haag--Kastler nets of local observables 
(\cite{Ar99, Ha96, BW92}), we would like to 
understand standard subspaces and their geometric realizations. 
Here a crucial point is that $\sV_\cM$ reflects many important properties 
of the von Neumann algebras $\cM$, related to  inclusions and 
symmetry groups quite faithfully, but in a much simpler environment 
(\cite{Lo08}, \cite[\S 4.2]{NO17}). 
In AQFT, standard subspaces provide 
the basis for the technique of modular localization, developed 
by Brunetti, Guido and Longo in \cite{BGL02}, see \cite{LL15} 
for more recent applications and  \cite{BDFS00} for 
geometric aspects of modular automorphism groups. 

In this article we develop a method for the explicit identification 
of standard subspaces 
in Hilbert spaces $\cH$ of holomorphic functions on tube domains, i.e., 
domains of the form $\cT = E + i C^0$, where $E$ is a finite dimensional 
real vector space, $C \subeq E$ is a closed 
convex cone with non-empty interior $C^0$, and the additive group 
$(E,+)$ acts unitarily by translations. 
Motivated by applications in AQFT, we are interested 
in those standard subspaces for which the 
modular group and the modular conjugation 
act naturally on~$\cT$.  Our construction provides in particular 
a direct conceptual way to Hilbert spaces of distributions on 
Minkowski space by taking boundary values 
of Hilbert spaces of holomorphic functions on suitable complex tube domains. 
In the subsequent papers \cite{NO21a, NO21b, NO21c} and \cite{Oeh21},
some of our results will be used to construct nets of 
standard subspaces on Lie groups and 
causal homogeneous  spaces. 

To explain our key idea, let us describe a context which is 
much more general than needed here, but which may therefore 
exhibit the overall idea more clearly. Suppose that 
$\Xi$ is a complex manifold and that 
$\cH \subeq \Hol (\Xi)$ is a Hilbert space of 
holomorphic functions on~$\Xi$. To implement the action of a modular group, 
we assume a group action 
\[ \sigma \: \R^\times \times \Xi \to \Xi, \quad 
(r,z) \mapsto \sigma_r(z) = \sigma^z(r) = r.z \] 
such that the maps $\sigma_{e^t}$ are holomorphic 
and $\tau_\Xi := \sigma_{-1}$ 
is an antiholomorphic involution 
with non-trivial fixed point space~$\Xi^{\tau_\Xi}$. 
We assume that $\sigma$ corresponds to an antiunitary representation 
$(U,\cH)$ of $\R^\times$ on $\cH$, so that a standard subspace $\sV \subeq \cH$ 
is specified by 
\[ \Delta_\sV^{-it/2\pi} = U(e^t) \quad \mbox{ and } \quad 
J_\sV = U(-1)\] 
(cf.\ Section~\ref{sec:2}). 
Here $\sigma$ provides a geometric implementation 
of the modular group and the modular conjugation on~$\Xi$.
We further assume that $\Xi$ sits in a larger manifold 
containing an $\R^\times$-invariant submanifold 
$M$ in the boundary of $\Xi$, 
such that there exists an injective boundary value map 
\[ b \: \cH \to C^{-\infty}(M).\] 
Then $\Xi^{\tau_\Xi}$ is a totally real submanifold of $\Xi$, 
so that all elements of $\Hol(\Xi)$ 
are determined by their values on $\Xi^{\tau_\Xi}$. 
We shall require that: 
\begin{itemize}
\item[\rm(C)] For $m \in \Xi^{\tau_\Xi}$, the orbit map 
$\sigma^m \: \R \to \Xi, t \mapsto \exp(th).m,$ extends holomorphically 
to a map $\cS_{-\pi/2,\pi/2} \to \Xi$ which further extends continuously 
to a map $\oline{\cS_{-\pi/2,\pi/2}} \to \Xi \cup M$ satisfying 
$\sigma^m(\pm \pi i/2) \in M$. 
\end{itemize}
This should lead to a realization of the standard subspace $\sV$ corresponding 
to $(\Delta_\sV, J_\sV)$ as a space of distributions on the boundary subsets  
$W_\pm := \big\{ \sigma^m\big(\mp\frac{\pi i}{2}\big) 
\: m \in \Xi^{\tau_\Xi}\big\}$, so-called {\it wedge domains} in $M$. 

The concrete environment we study is specified in 
Section~\ref{sec:3} in terms of the following geometric data: 
\begin{itemize}
\item[\rm(A1)] $E$ is a finite dimensional real vector space.
\item[\rm(A2)] $h \in \End(E)$ is 
diagonalizable with eigenvalues $\{-1,0,1\}$ and 
$\tau := e^{\pi i h}$.
\item[\rm(A3)] $C \subeq E$ is a pointed, generating 
closed convex cone invariant under $e^{\R h}$ and $-\tau$.
\end{itemize}
This geometric context includes in particular 
the case where $E$ is Minkowski space, $C$ is the closed forward light 
cone and $h$ generates a Lorentz boost. It also contains a series  
of interesting generalizations in the context of euclidean Jordan algebras 
discussed in Section~\ref{sec:6} and Appendix~\ref{app:b}. 
Minkowski spaces are the simple euclidean 
Jordan algebras of rank~$2$. The other simple euclidean Jordan algebras 
of rank $r$ are 
\[ \Sym_r(\R), \quad \Herm_r(\C), \quad 
\Herm_r(\H) \quad \mbox{ for } \quad r \in \N 
\quad \mbox{ and } \quad \Herm_3(\bO),\] 
where $\bO$ is the alternative algebra of octonions 
(see  \cite{JvNW34}, \cite{FK94} for the classification) and for these 
the cone $C$ of squares 
consists of the positive hermitian matrices. These configurations 
are not Lorentzian for $r \not=2$, but they carry 
causal and conformal structures generalizing those of Minkowski space. 

Besides the Lorentz boost, 
other triples $(E,C,h)$ that appear naturally in AQFT are
$(\R,\R_+, \id_\R)$, which corresponds to positive energy 
representations of the $ax+b$-group and modular inclusions 
of standard subspaces, and the triple 
$(\R^{1,d-1},\oline{V_+}, \id_\R)$ appears in conformal 
field theory on Minkowski space, where the modular group 
acts by dilations and the corresponding wedge domain is the positive 
cone $V_+$ itself (cf.\ \cite{Bu78}). Note that 
the triple $(E,C,h)$ does not encode any Lorentzian structure, 
only the causal structure specified by the cone and the 
``modular'' structure specified by $h$. However, 
Lorentzian structures  
appear indirectly on low-dimensional subspaces 
(see Example~\ref{ex:1}(c)).  

More such triples appear, where $E= \g$ is a non-abelian Lie algebra with 
an invariant cone $C$ and $h \in \g$ is such that 
$\ad h$ defines a $3$-grading of $\g$ 
(in \cite{MN21} these elements are called {\it Euler elements}), 
but then the associated unitary representations are more complicated 
than the ones we consider here (see \cite{NO21a}). Here the Lie algebras arising 
in AQFT are $\fsl_2(\R)$ (from the M\"obius group in CFT), 
the Poincar\'e--Lie algebra and the conformal Lie algebra 
$\so_{2,d}(\R)$ of Minkowski space.

The above setup is supplemented by the following analytic data. 
We fix a complex Hilbert space $\cK$ and a $\Herm^+(\cK)$-valued 
measure $\mu$ on the dual cone $C^\star \subeq E^*$ which defines a 
$\cK$-valued $L^2$-space $\cH := 
L^2(C^\star, \mu;\cK)$ (cf.\ \cite{Ne98}). 
In the special case $(E,C) = (\R^{1,d-1},\oline{V_+})$ 
one identifies $E^*$ with $E$ and $C^\star$ with $\oline{V_+}$. Here 
the Lorentz invariant measures on the mass hyperboloids and the 
mantle of the light cone provide natural examples with $\cK = \C$. 
Higher spin corresponds to non-trivial spaces $\cK$ which carry a 
representation of the Lorentz group (cf.~\cite{SW64}). 

To specify a standard subspace of $\cH$, 
we fix a representation $\rho \: \R^\times \to \GL(\cK)$ by normal 
operators for which $J_\cK := \rho(-1)$ is a conjugation 
and $\rho(e^t) = e^{t\Lambda}$ for a normal operator $\Lambda$ with bounded 
symmetric part. 
The compatibility between $\mu$ and $\rho$ consists in the relations 
\[  (-\tau)_*\mu = J_\cK \mu J_\cK \quad \mbox{ and } \quad 
\sigma(r)_*\mu = \rho(r)^* \cdot \mu \cdot \rho(r)
\quad \mbox{ for } \quad r > 0, \]
where the $\R^\times_+$-action on $E^*$ is given by 
$\sigma(e^t)\lambda := \lambda \circ e^{-th}$. 
We specify a standard subspace $\sV \subeq \cH 
= L^2(E^*,\mu;\cK)$ by 
\[ (\Delta_\sV^{-it/2\pi} f)(\lambda) = \rho(e^t) f(\lambda \circ e^{th}) 
\quad \mbox{ and }\quad 
(J_\sV f)(\lambda) := J_\cK f(-\lambda \circ \tau), \qquad 
\lambda \in C^\star.\] 
Then we have a realization
\[ \Phi \: L^2(C^\star,\mu;\cK) \to \Hol(\cT, \cK), \quad 
\Phi(f)(z) = \int_{E^*} e^{i\lambda(z)} d\mu(\lambda) f(\lambda) 
\] 
of $\cH$ as a Hilbert space $\cH_K$ of $\cK$-valued holomorphic function on 
the tube $\cT = E + i C^0$ whose reproducing kernel is given by 
\[ K \: \cT \times \cT \to B(\cK), \quad 
K(z,w) := \tilde\mu(z - \oline w), \quad \mbox{ where } \quad 
\tilde\mu(z) := \int_{C^\star} e^{i\lambda(z)}\, d\mu(\lambda)\] 
(Lemma \ref{lem:UnitRep}).  Passing to boundary values, we  obtain a 
third realization of $\cH$  as a subspace 
of $\cS'(E,\cK)$, the space  of $\cK$-valued tempered distributions on~$E$. 
Our first main result is the Standard Subspace Theorem (Theorem~\ref{thm:2.1}). 
To formulate it, we introduce the wedge domain 
\[ W = W(E,C,h) := E_0(h) + (C^0 \cap E_1(h)) - (C^0 \cap E_{-1}(h)) \subeq E, \] 
specified by the data in (A1-3), where 
$E_j(h) = \ker(h - j \id_E)$ is the $j$-eigenspace of $h$ in $E$. 
Then the standard subspace $\sV \subeq \cH$ can be 
described in terms of $W$ by 
\begin{equation}
  \label{eq:intro1}
\sV = 
\oline{\spann_\R \{ \tilde\phi \eta \: \phi \in C^\infty_c(W,\R), 
\eta \in \sV_\cK^\sharp\}}, \quad \mbox{ where }\quad 
\sV_\cK^\sharp := \Fix(e^{-\pi i \Lambda} J_\cK) \subeq \cK
\end{equation}
is a standard subspace and $\tilde\phi(\lambda) 
= \int_E e^{i\lambda(x)} \phi(x) \,  d\mu_E(x)$ 
denotes the Fourier transform. 

This theorem has several interesting consequences. 
First it exhibits an interesting relation with reflection positivity 
(Osterwalder--Schrader positivity) and its 
representation theoretic aspects (\cite{NO14,NO18}). 
This starts with the observation that the triple 
$(\cH^\R, \sV,J_\sV)$, where $\cH^\R$ is the real Hilbert space 
underlying $\cH$, is {\it reflection positive}, i.e., 
$\la \xi, J \xi \ra \geq 0$ for $\xi \in \sV$ 
(\cite{Lo08}). 
Further, the multiplication representation $(U,\cH)$ of the 
additive group $(E,+)$ on $L^2(C^\star,\mu;\cK)$ is reflection 
positive for the triple $(E,W,\tau)$ in the sense of \cite{NO18} 
(see Section~\ref{sec:4} for details), so that the general 
theory of reflection positive distributions 
developed in \cite{NO14} suggests that the restriction 
of the positive definite distribution $\tilde\mu \in \cS'(E,B(\cK))$ 
to the open cone $W$ should be positive definite with respect to the 
involution $x^\sharp = - \tau(x)$. 
This is verified in Proposition~\ref{prop:analyticfunc}, which 
even shows that it is represented by an operator-valued function. 

Presently we do not see how this occurrence of reflection positivity 
is related to euclidean models, which was the original intention of 
this concept (\cite{OS73}). This would require a better understanding of the 
group theoretic side of Schlingemann's work \cite{Sc99a, Sc99b}.

From the description of $\sV$ in \eqref{eq:intro1}, 
we draw interesting conclusions on the 
support of the ``imaginary part'' of the operator-valued distribution 
$\tilde\mu$, provided that the infinitesimal generator $\Lambda$ of $\rho$ 
has integral spectrum (Proposition~\ref{prop:supp-control}). 
This aspect is explored in detail 
in Section~\ref{sec:6} for the Riesz measures 
$\mu_s$ of a simple euclidean Jordan algebra~$E$. This is an important 
one-parameter family of scalar-valued measures on the 
dual $C^\star$ of the positive cone $C$ of the Jordan algebra~$E$. 
For their Fourier transforms $\tilde\mu_s$, we analyze the support 
of the imaginary part in detail. Using some tools developed 
in Appendix~\ref{app:b}, it is 
straightforward to decide which connected components of the 
set $E^\times$ of invertible elements are contained in the support 
(Proposition~\ref{prop:opensupp}), but the description of the precise 
 support requires closer inspection 
(Proposition~\ref{prop:7.10}). 
In Theorem~\ref{thm:6.11} we characterize 
the situations when the support of $\Im(\tilde\mu_s)$ 
is contained in the complement of the wedge domain 
$W(E,C,h)$. In the AQFT context on Minkowski space, this 
corresponds to wedge duality.

For the one-dimensional Jordan algebra $E = \R$, 
the Riesz measures $\mu_s$ are multiples of the measures 
$x^{s-1}\, dx$ ($s > 0$) on the positive half line. 
These measures arise naturally in the context of CFT from 
derivatives of the $\U(1)$-current (see \cite[\S 5.2]{Lo08} for details). 
The spaces $L^2(\R_+, \mu_s)$ carry a natural positive energy 
representation of the $2$-dimensional group $\Aff(\R)$. 
These representations are mutually equivalent, but the realization 
by holomorphic functions on the half plane yields extensions 
to the simply connected covering $G$ of the M\"obius group $\PGL_2(\R)$ 
which depend on $s$ (corresponding to 
 the lowest weight). This leads to a one-parameter 
family of M\"obius covariant nets of standard subspaces 
with rather different properties. 

The contents of this paper are as follows: 
In Section~\ref{sec:2} we collect some generalities on 
standard subspaces in Hilbert spaces of holomorphic functions.
An important starting point is Proposition~\ref{prop:standchar}, 
characterizing elements of $\sV$ as those vectors $\xi \in \cH$ for which 
the orbit map 
\[ \alpha^\xi \: \R \to \cH, \quad t \mapsto \Delta_\sV^{-it/2\pi}\xi \] 
extends to a continuous map on the closed strip 
$\{ z \in \C \: 0 \leq \Im z \leq \pi\}$, holomorphic on the interior, and 
satisfying $\alpha^\xi(\pi i) = J\xi$. This condition 
is intimately related to (C) above. 
We also discuss the standard subspaces of $\cK$ naturally 
attached to the representation $\rho$ and provide a direct 
description of the real subspaces $\cH_K^{J_\sV}$ 
in terms of the reproducing kernel~$K$. 

In Section~\ref{sec:3} we turn to our subject proper, 
the spaces $\cH = L^2(C^\star, \mu;\cK)$ and the corresponding 
standard subspace. 
After discussing the geometric implications of the axioms (A1-3) 
(Subsection~\ref{subsec:3.1}), 
we introduce in Subsection~\ref{subsec:3.2} the Hilbert space 
$\cH$ and its realizations by holomorphic functions and tempered distributions. 
The Standard Subspace Theorem (Theorem~\ref{thm:2.1}) 
is proved in Subsection~\ref{subsec:3.3}. 
We end this section with a discussion of the Riesz measures 
$x^{s-1}\, dx$ ($s > 0$) on the positive half line as examples. 
The relation with reflection positivity is 
discussed in Section~\ref{sec:4}, and 
the support properties of the imaginary part of the distribution 
are investigated in Sections~\ref{sec:5} (for the general case), 
and in Section~\ref{sec:6} for Riesz measures. 
Two appendices collect background material on 
standard subspaces (Appendix~\ref{app:a}) and 
structure theoretic 
results on Jordan algebras (Appendix~\ref{app:b}) that are used in 
Section~\ref{sec:6} for the support analysis. They also
find applications in forthcoming work 
(\cite{NO21a}, \cite{Oeh21}). 
For instance, in \cite{NO21a} we construct isotone, covariant 
nets of standard subspaces with the Bisognano--Wichmann property on 
{\it Jordan space-times}, 
i.e., the universal covers of the conformal completion 
of a euclidean Jordan algebra~$E$. 
These are the Jordan space-times in the sense of 
G\"unaydin \cite{Gu93}, resp., the simple space-time manifolds in the sense of 
Mack--de Riese \cite{MdR07}.

We conclude this introduction with a discussion of some 
of the numerous connections of the techniques and the results of the 
present paper with the AQFT literature, where 
boundary values of holomorphic functions on various kinds of tube domains 
play a central role. 

In  \cite{Bo95a, Bo95b}, Borchers also considers analytic extensions of orbit maps to strips. His method 
to derive wedge duality from Lorentz covariance 
could be easily interpreted 
in terms of exhibiting a pair $(\sV,\sV')$ of mutually dual standard 
subspaces. This includes observations relating the standard right wedge 
and the future tube as in our Lemma~\ref{lem:4.2}. 
Here Lorentz covariance provides unitary one-parameter groups 
corresponding to boosts and the orbit maps of these boosts are 
extended analytically to a horizontal strip.
On the level of operators, analytic continuation to the strip 
plays a key role in the identification of unitary endomorphisms 
of von Neumann algebras (\cite[Prop.~2.1]{Bo95a}, \cite[Thm.~A]{Bo95b}). 
This is related to the Araki--Zsid\'o Theorem \cite{AZ05} 
characterizing operators between standard subspaces in the spirit 
of Proposition~\ref{prop:standchar}. We refer to 
 \cite{Ne19a} for applications of this result to the identification 
of endomorphism semigroup of standard subspaces. 

In \cite{Ep67} Epstein uses the action of the complex 
boost group $\SO_{1,1}(\C) \cong \C^\times$ on $\C^4$ (acting in the first 
two components). 
He applies group elements to certain tube domains 
in $\C^4$ to relate distributional boundary values from opposite tubes 
(a generalized retarded function and its CTP counterpart) 
by specializing to $\pm 1$.   
Epstein's matrix-valued ``functions'' $W^+(p)$ 
represent the spectral measure of the energy-momentum operator,  
resp., of the representation of the group of space-time translations. 
This corresponds to our $L^2$-picture in Section~\ref{subsec:3.2}. 
He uses analytic extensions to implement a conjugation 
representing a CTP operator, 
whereas the geometric nature of the modular conjugation   
is built into our approach from the outset because it is implemented 
by the involution $\tau$ on $E$. Accordingly, the nature 
of the corresponding results on boundary values is different. 
Likewise, in \cite{DHR74} similar 
analytic extension techniques are used to implement 
a $C$-operator (charge inversion), which is an antiunitary conjugation, 
but in our context its existence is assumed. 

Mund uses in \cite[Lemma~4]{Mu01} analytic continuation methods 
developed in \cite{BE85} and 
derives 
the Bisognano--Wichmann property from 
localization in spacelike cones~$C$, i.e., 
causal completions of pointed open cones in the time zero hyperplane.
For fields with a positive mass~$m$, he 
describes certain smooth functions on the mass shell 
\[ H_m = \{ k = (k_0,\bk) \in \R^4\:  k^2 = m^2, k_0 > 0 \} \] 
as boundary values of  holomorphic functions on the 
complexification of $H_m$, intersected  with a tube over a spacelike cone. 
In this context no geometric realization 
of the modular group is assumed, but proved. 
All these results extend and refine the 
classical PCT and Spin-Statistics Theorems one finds  
in \cite[Ch.~4]{SW64}. 

These techniques are related to our paper as follows. 
For $E = \R^{1,d-1}$ and $C = \oline{V_+}$, our 
setup covers all positive energy representations of the Poincar\'e group, 
restricted to the subgroup $E \rtimes \R^\times$, consisting 
of space-time translations and the full  $\SO_{1,1}(\R)$-group 
corresponding to a Lorentz boost. The vector-valued 
$L^2$-space then arises from the 
(energy-momentum) spectral measure which takes values in operators 
on multiplicity space. These measures are particularly simple 
for one-particle representations, but they also exist for 
representations defined by quantum fields, where the standard subspaces 
correspond to von Neumann algebras of local observables.  

Our perspective is to start from a linear $\R^\times$-action 
and to determine the nature of the corresponding standard subspaces 
for general representations satisfying a natural spectral condition. 
Jordan algebras $E$ of rank $r$ ($r=2$ corresponds to Minkowski spaces) 
provide in particular interesting measures $\nu_m$ 
on the level sets of the Jordan determinant $\det_E$ on the positive 
cone. Their Fourier transforms $\hat\nu_m$ solve the generalized 
Klein--Gordon equation $(\det_E(\partial) + m^2)\hat\mu_m = 0$, 
where the operator $\det_E(\partial)$ is of order $r$. 
We expect that many of the established techniques used for the 
Klein--Gordon equation can be applied to this context as well. 
As the measures $\nu_m$ live on hypersurfaces, they can be projected 
to hyperplanes which resembles a ``time-zero'' realization. 
One may also write the measues $\nu_m$ in a suitable way as differences 
of boundary values of holomorphic functions 
(see \cite[\S 4.1]{Re16}) and thus obtain decompositions of the imaginary 
part of Fourier transform $\hat\nu_m$ into an 
advanced and a retared fundamental solution 
(\cite[\S 2.2.2]{Ge19}) which can possibly be expressed in 
terms of Riesz distributions on $E$ 
(see \cite[\S 3]{Gi09} for the Minkowski case $r = 2$). 
In conclusion, we have been strongly motivated by the
deep structures in algebraic quantum field theory, 
and there are clear inspirations
in these physical theories to pursue further in the context 
of Jordan algebras, not the least from the basic notions 
of causality and positive energy.

\nin {\bf Acknowledgment:} We are most 
grateful to the  two referees for their constructive criticism and for 
pointing out several interesting references to  the AQFT literature.

\subsection*{Notation} 
Here we collect some notation that we will use in the article.
 
\begin{itemize}
\item $\cH$ is a complex Hilbert space with inner product
$\ip{\cdot }{\cdot } = \ip{\cdot }{\cdot }_\cH$ which is linear in the second argument.  
\item We write $\R_+ = [0,\infty)$ for the closed positive half line. 
\item If $M$ is a smooth manifold, we write $C^\infty_c(M)$ for the space 
of complex-valued test functions on $M$, endowed with the natural LF topology, 
and $C^{-\infty}_c(M)$ for the space of {\bf antilinear} continuous 
functionals  on this space, i.e., the space of 
distributions on~$M$. 
\item We likewise consider tempered distributions 
$D \in \cS'(E)$ on a real finite dimensional vector space $E$ 
as antilinear functionals on the Schwartz space $\cS(E)$. 
The {\it Fourier transform of an $L^1$-function $f$} on $E$ is defined by 
\begin{equation}
  \label{eq:ftrafo}
\hat f(\lambda) := \int_E e^{-i\lambda(x)} f(x) \,  d\mu_E(x), \quad 
\lambda \in E^*, 
\end{equation}
where $\mu_E$ denotes a Lebesgue measure on $E$.   
For tempered distributions $D \in \cS'(E)$, we define the Fourier transform by 
\begin{equation}
  \label{eq:2}
\hat D(\phi) := D(\tilde \phi), \quad \mbox{ where} 
\quad 
\tilde\phi(\lambda) 
:= \hat\phi(-\lambda) 
= \int_E e^{i\lambda(x)} \phi(x) \,  d\mu_E(x).
\end{equation}
For the distribution 
$D_f(\phi) := \int_E \oline{\phi(x)} f(x)\, d\mu_E(x)$ defined by 
an $L^1$-function, we then have 
$\hat{D_f} = D_{\hat f}$. 
\item $C^0$ is the interior of the convex cone $C$ 
and $C^\star = \{ \alpha \in E^* \: (\forall v \in C) \alpha(v) \geq 0\}$ 
its dual cone. 
\end{itemize}

\tableofcontents 

\section{Standard subspaces in Hilbert spaces of holomorphic functions}
\mlabel{sec:2}

We have collected basic facts about standard subspaces in Appendix \ref{app:a}. We
refer to that section for notation, but let us  recall that a 
{\it standard subspace} is a closed real subspace $\sV\subset \cH$ such
that   $\sV\cap i\sV=\{0\}$ and $\sV+i\sV$ is dense in $\cH$. Every standard subspace comes with a conjugation $J_\sV$
and
positive, densely defined operator $\Delta_\sV$ such that 
$T_\sV=J_\sV\Delta_\sV^{1/2}$ is the conjugation $u+iv\mapsto u-iv$, $u,v\in \sV$ and $J_\sV\Delta_\sV=\Delta_\sV^{-1}J_\sV$.
The pair $(\Delta_\sV,J_\sV)$ is {\it the pair of modular objects} associated to 
$\sV$, and $T_\sV$ is the 
\textit{Tomita operator} of $\sV$. In this section we 
study the situation where the Hilbert space is
a reproducing kernel Hilbert space of vector-valued
holomorphic functions on a complex manifold. In Proposition \ref{prop:standchar}
we derive a fundamental connection between the standard subspace 
$\sV$ and the space $\cH^{J_\sV}$ 
of $J_\sV$-fixed vectors is described in 
Lemma \ref{lem:hjvec}. 
These results are needed in Section \ref{subsec:3.3}.

\subsection{Standard subspaces and $J$-fixed points} 
\mlabel{subsec:1.3}
 
In this subsection we derive a characterization of the elements 
of a standard subspace $\sV$ specified by the pair $(\Delta, J)$ in 
terms of analytic continuation of orbit maps of the unitary 
one-parameter group $(\Delta^{it})_{t \in \R}$ and the real space
$\cH^J$.

\begin{prop} \mlabel{prop:standchar} Let $\sV \subeq \cH$ be a standard subspace 
with modular objects $(\Delta, J)$. For 
$\xi \in \cH$, we consider the orbit map $\alpha^\xi \: \R \to \cH, t \mapsto 
\Delta^{-it/2\pi}\xi$. Then the following are equivalent:
\begin{itemize}
\item[\rm(i)] $\xi \in \sV$. 
\item[\rm(ii)] $\xi \in \cD(\Delta^{1/2})$ with $\Delta^{1/2}\xi = J\xi$. 
\item[\rm(iii)] The orbit map $\alpha^\xi \: \R \to \cH$ 
extends to a continuous map  $\oline{\cS_\pi} \to \cH$ which is 
holomorphic on $\cS_\pi$ and satisfies $\alpha^\xi(\pi i) = J\xi$. 
\item[\rm(iv)] There exists an element $\eta \in \cH^J$ 
whose orbit map $\alpha^\eta$ 
extends to a continuous map  \break $\oline{\cS_{-\pi/2, \pi/2}} \to \cH$ which is 
holomorphic on the interior and satisfies $\alpha^\eta(-\pi i/2) = \xi$. 
\end{itemize}
\end{prop}

\begin{prf} The equivalence of (i) and (ii) follows from the definition 
of $\Delta$ and $J$. 

\nin (ii) $\Rarrow$ (iii): If $\xi \in \cD(\Delta^{1/2})$, then 
$\xi \in \cD(\Delta^z)$ for $0 \leq \Re z \leq 1/2$, so that the map 
\[ f \: \oline{\cS_\pi} \to \cH, \quad 
f(z) := \Delta^{-\frac{iz}{2\pi}} \xi \] 
is defined. 
Let $P$ denote the spectral measure of the selfadjoint operator  
\[ H := - \frac{1}{2\pi} \log\Delta \quad \mbox{ and let } \quad 
P^\xi = \la \xi, P(\cdot) \xi\ra \] 
denote the corresponding positive 
measure on $\R$ defined by $\xi \in \cH$. 
Then \cite[Lemma~A.2.5]{NO18}  shows that 
\[ \cL(P^\xi)(2\pi) = \int_\R e^{-2\pi\lambda}\, dP^\xi(\lambda) < \infty.\] 
This implies that the kernel 
\[ \la f(w), f(z) \ra 
= \la \Delta^{-\frac{iw}{2\pi}} \xi, \Delta^{-\frac{iz}{2\pi}} \xi \ra 
= \la  \xi, \Delta^{-\frac{i(z-\oline w)}{2\pi}} \xi \ra 
= \la  \xi, e^{(z-\oline w)i H} \xi \ra 
= \cL(P^\xi)\Big(\frac{z-\oline w}{i}\Big) \] 
is continuous on $\oline{\cS_{\pi}} \times \oline{\cS_{\pi}}$ 
by the Dominated Convergence Theorem,  
holomorphic in $z$, and antiholomorphic in $w$ on 
the interior (\cite[Prop.~V.4.6]{Ne00}). 
This implies (iii) because it shows that 
$f$ is holomorphic on $\cS_\pi$ (\cite[Lemma~A.III.1]{Ne00}) 
and continuous on $\oline{\cS_\pi}$. 

\nin (iii) $\Rarrow$ (iv): For $\alpha^\xi \: \oline{\cS_\pi} \to \cH$ 
as in (iii), we have 
\begin{equation}
  \label{eq:jequiv}
J \alpha^\xi(z) = \alpha^\xi(\pi i + \oline z) 
\end{equation}
by analytic continuation, so that 
\[ \eta := \alpha^\xi(\pi i/2) \in \cH^J \quad \mbox{ with } \quad 
\alpha^\eta(z) = \alpha^\xi\Big(z + \frac{\pi i}{2}\Big).\] 

\nin (iv) $\Rarrow$ (ii): We abbreviate $\cS := \cS_{-\pi/2, \pi/2}$. 
The kernel 
$K(z,w) := \la \alpha^\eta(w), \alpha^\eta(z) \ra$ 
is continuous on $\oline{\cS} \times \oline{\cS}$ and 
holomorphic in $z$ and antiholomorphic in $w$ on $\cS$. 
It also  satisfies 
$K(z + t, w) = K(z,w - t)$ for $t \in \R$. Hence there exists a 
continuous function $\phi$ on $\oline{\cS}$, holomorphic on $\cS$, such that 
\[ K(z,w) = \phi\Big(\frac{z - \oline w}{2}\Big).\] 
For $t \in \R$, we then have 
$\phi(t) = \la \eta, \alpha^\eta(2t) \ra = \int_\R e^{2it\lambda}\, dP^\eta(\lambda)$, 
so that \cite[Lemma~A.2.5]{NO18} yields 
$\cL(P^\eta)(\pm\pi) < \infty$ and 
$\eta \in \cD(\Delta^{\pm 1/4})$. This implies that 
$\alpha^\eta(z) = \Delta^{-iz/2\pi} \eta$ for $z \in \oline{\cS}$. 

From 
$\xi = \alpha^\eta(-\pi i/2) = \Delta^{-1/4} \eta$ we derive that 
\[  \alpha^\xi(z) = \alpha^\eta\Big(z- \frac{\pi i}{2}\Big) 
=  \Delta^{-iz/2\pi} \xi \quad \mbox{ for } \quad z \in \oline{\cS_\pi}.\] 
Further, $J\eta = \eta$ implies  
\[  J\alpha^\xi(z) 
= J\alpha^\eta\Big(z- \frac{\pi i}{2}\Big)
= \alpha^\eta\Big(\oline z+  \frac{\pi i}{2}\Big)
= \alpha^\xi(\pi i + \oline z).\] 
For $z = 0$, we obtain in particular 
$J\xi = \alpha^\xi(\pi i) = \Delta^{1/2}\xi$. 
\end{prf}

\subsection{The general setting for 
spaces of holomorphic functions} 
\mlabel{subsec:1.2}

In Appendix~\ref{subsec:hardy}  
we show that a standard subspace $\sV \subeq \cH$ 
always specifies a realization of the complex Hilbert space 
$\cH$ as a vector-valued Hardy space 
on a strip, even if $\cH$ has no specific geometric structure. 
In this subsection we consider an enriched geometric context. 
A key observation is Lemma \ref{lem:hjvec} that 
will later be applied to the situation where
the complex manifold is an open tube domain $\cT$.

On the geometric side, we consider 
a connected complex manifold $\Xi$, endowed with a smooth action 
\[ \sigma  \: \R^\times \times \Xi \to \Xi, \quad 
(r,m) \mapsto r.m =: \sigma_r(m) =: \sigma^m(r) \] 
for which the diffeomorphisms 
$\sigma_r$ are holomorphic for $r > 0$ and antiholomorphic for $r < 0$. 
In particular, $\tau_\Xi := \sigma_{-1}$ is an antiholomorphic involution of $\Xi$. 
We further assume that $\Xi$ is an open domain in a larger complex manifold 
and that the boundary $\partial \Xi$ contains a real submanifold $M$ with 
the property that, for every fixed point 
$m \in \Xi^{\tau_{\Xi}}$, the orbit map 
$\R \to \Xi, t \mapsto \sigma^m(e^t)$ extends to a holomorphic 
map $\sigma^m \: \cS_{-\pi/2,\pi/2} \to \Xi$ which further 
extends to a continuous map 
\begin{equation}
  \label{eq:extcond}
\sigma^m \: \oline{\cS_{-\pi/2,\pi/2}} \to \Xi \cup M \quad \mbox{ with } \quad 
\sigma^m(\pm i \pi/2) \in M.
\end{equation}

\begin{exs} (Domains in $\C$) \mlabel{ex:g.1} 
In one dimension we have the following standard examples 
of simply connected proper domains in $\C$ with 
their natural $\R^\times$-actions. 

\nin (a) (Strips) On the  strip 
$\cS_\pi = \{ z \in \C \: 0 < \Im z < \pi\}$ 
we have the antiholomorphic involution 
$\tau_{\cS_\pi}(z) = \pi i+ \oline z$ with fixed point set 
\[ \cS_\pi^{\tau_{\cS_\pi}} = \Big\{ z \in \cS_\pi \: \Im z = \frac{\pi}{2}\Big\}.\] 
The group $\R^\times_+$ acts by translations via $\sigma_{e^t}(z) = z + t$, 
$M := \R \cup (\pi i + \R) = \partial \cS_\pi$ is a real submanifold, and 
for $\Im z = \pi/2$, the orbit map $\sigma^z(t)$ extends to the 
closure of the strip $\cS_{-\frac{\pi}{2}, \frac{\pi}{2}}$ with 
$\sigma^z\big(\pm \frac{\pi i}{2}\big) = z \pm \frac{\pi i}{2} \in M$. 

\nin (b) (Upper half plane) On the upper half plane 
$\C_+  = \{ z \in \C \: \Im z > 0\}$, we have the 
antiholomorphic involution $\tau_{\C_+}(z) = -\oline z$ 
and the action of $\R^\times_+$ by dilations $\sigma_r(z) = rz$. 
Here $M := \R = \partial \C_+$ is a real submanifold, and 
for $z = i y$, $y > 0$, the orbit map $\sigma^z(t) = e^t  z$ extends to the 
closure of the strip 
$\cS_{-\frac{\pi}{2}, \frac{\pi}{2}}$ with 
$\sigma^z\big(\pm \frac{\pi i}{2}\big) = \pm i (i y) = \mp y$. 

\nin (c) (Unit disc) On the unit disc 
$\cD = \{ z \in \C \: |z| < 1\}$ we have the antiholomorphic 
involution $\tau_\cD(z) = \oline z$ and the 
action of $\R^\times_+ \cong \SO_{1,1}(\R)_0$ by the maps 
\begin{equation}
  \label{eq:discact}
 \sigma_t(z) = \frac{ \cosh(t/2) z + \sinh(t/2)}{\sinh(t/2)z + \cosh(t/2)}.
\end{equation} 
Here $M := \bS^1 = \partial \cD$ is a real submanifold, and 
for $z \in \cD \cap \R$, the orbit map $\sigma^z(t)$ extends to the 
closure of the strip $\cS_{-\pi/2,\pi/2}$ with 
\[ \sigma^z(\pm \pi i/2) 
= \frac{ \cos(\pi/4) z \pm i \sin(\pi/4)}{\pm i \sin(\pi/4)z + \cos(\pi/4)}
= \frac{z \pm i}{\pm iz + 1} 
= \mp i \cdot \frac{z \pm i}{z \mp i}.\]

The biholomorphic maps 
\begin{equation}
  \label{eq:cayley}
 \Exp \: \cS_\pi \to \C_+, \ \   z \mapsto e^z \quad \mbox{ and } \quad 
\Cay \: \C_+ \to \cD, \ \  \Cay(z) := \frac{z-i}{z+ i} 
\end{equation}
are equivariant for the described $\R^\times$-actions 
on the respective domains. 

The Riemann mapping theorem implies that any 
antiholomorphic $\R^\times$-action on a proper simply connected 
domain $\cO \subeq \C$ 
is equivalent to the examples (a)-(c). 
In fact, we may w.l.o.g.\ assume that $\cO = \cD$ and, 
since every isometric involution of the hyperbolic plane has a fixed 
point, we may also assume that 
 $\sigma_{-1}(0) = 0$. Then $z \mapsto \oline{\sigma_{-1}(z)}$ 
is biholomorphic on $\cD$ fixing $0$, hence of the form 
$z \mapsto e^{it}z$, and from that it follows that, 
up to conjugation with biholomorphic maps, we may assume that 
$\sigma_{-1}(z) = \oline z$. Now we simply observe that the centralizer 
of $\sigma_{-1}$ in the group $\PSU_{1,1}(\C) \cong \Aut(\cD)$ is 
$\PSO_{1,1}(\R)$, and this leads to the action in \eqref{eq:discact}. 
\end{exs}

On the representation theoretic side, we 
consider a Hilbert space $\cH$, realized on a connected complex manifold 
$\Xi$ for another complex Hilbert space $\cK$ as subspace of the 
Fr\'echet space $\Hol(\Xi,\cK)$ of holomorphic functions 
$f \: \Xi \to \cK$. Here we assume that the point evaluations 
\[ K_z \: \cH \to \cK, \quad f \mapsto f(z) \] 
are continuous. Then 
\[ K \: \Xi \times \Xi \to B(\cK), \quad 
K(z,w) := K_z K_w^* \] 
is a positive definite operator-valued kernel which determines 
$\cH$ uniquely, so that we write $\cH_K$ to emphasize the dependence of 
$\cH$ from the kernel~$K$ 
(see \cite[Ch.~I]{Ne00} for details). 

To connect this structure to the antiunitary representation 
$(U^{\sV},\cH)$ of $\R^\times$, corresponding to a standard subspace $\sV$, 
\begin{footnote} {See
Appendix \ref{subsec:1.1} for this correspondence and 
in particular Theorem \ref{thm:bij}.} 
\end{footnote}
we also need a representation of $\R^\times$ on $\cK$. 
This is specified by a conjugation $J_\cK$ on $\cK$ 
and a strongly continuous homomorphism  
\[ \rho \: \R^\times_+ \to \GL(\cK) \] 
whose range commutes with $J_\cK$, so that it extends by 
$\rho(-1) := J_\cK$ to a representation of $\R^\times$ on $\cK$.
We also assume that the operators $\rho(e^t)$ are normal and that 
the hermitian 
one-parameter group $t \mapsto \rho(e^t) \rho(e^t)^*$ is norm-continuous. 
This implies that 
\[ \rho(e^t) = e^{t \Lambda}, \qquad 
\mbox{ where } \quad \Lambda \: \cD(\Lambda) \to \cK,\ \  \Lambda\xi 
=\dtZ  \rho(e^t)\xi,\]
 is an unbounded operator of the form 
$\Lambda = \Lambda_- + \Lambda_+$, where $\Lambda_+$ is a bounded symmetric 
operator and $\Lambda_-$ is a skew-adjoint operator (possibly unbounded) 
that commutes with $\Lambda_+$. For $t \geq 0$, we have 
\[ \|e^{t\Lambda}\| = \|e^{t\Lambda_+}\| 
= e^{t \sup(\Spec(\Lambda_+))},\] 
showing 
that the boundedness of $\Lambda_+$ is required to obtain a one-parameter 
group of bounded operators. 

\subsubsection{Three standard subspaces of $\cK$} 
\mlabel{rem:3.1}
Standard subspaces in $\cH $ are closely related to standard subspaces in $\cK$. 
In this section we associate three standard subspaces of $\cK$ 
to the representation $(\rho,\cK)$. 
In particular, we construct a modular pair $(\Delta_\cK ,J_\cK)$ on $\cK$ 
from the operator $\Lambda_-$, which specifies a standard subspace $\sV_\cK$ 
(cf.\ Theorem~\ref{thm:bij}). 
Further, the symmetric part $\Lambda_+$ of $\Lambda$ leads to 
twists $\sV_\cK^\sharp$ of $\sV_\cK$ and $\sV_\cK^\flat$ of $\sV_\cK'$. 
These subspaces will be needed below in our description of the standard 
subspace $\sV \subeq \cH$. 
We shall see in Section~\ref{sec:5} that 
the case where  $\sV_\cK^\sharp = \sV_\cK^\flat$ 
(Lemma~\ref{lem:VsharpVflat}) is of particular interest 
for the analysis of support properties.

With the measurable functional calculus of normal operators, 
we obtain the (possibly unbounded) normal operator 
\[ e^{\pi i \Lambda} 
:=  e^{\pi i \Lambda_+} e^{\pi i \Lambda_-}
= e^{\pi i \Lambda_-}  e^{\pi i \Lambda_+}  \quad \mbox{ with } \quad 
\cD(e^{\pi i \Lambda}) = \cD(e^{\pi i \Lambda_-}).\] 
The operator $\Delta_\cK := e^{2\pi i \Lambda_-}$ is strictly positive selfadjoint, 
and, as $J_\cK$ commutes with $\Lambda_\pm$, we obtain the modular relation 
\[ J_\cK \Delta_\cK J_\cK = \Delta_\cK^{-1}.\] 
Hence $T_\cK := J_\cK \Delta_\cK^{1/2} = J_\cK e^{\pi i \Lambda_-}$ 
is the Tomita operator of the standard subspace 
\[ \sV_\cK := \Fix(T_\cK) \subeq \cK,\] 
and its adjoint $T_\cK^* = J_\cK \Delta_\cK^{-1/2}$ 
is the Tomita operator of $\sV_\cK'$ 
(cf.\ Lemma~\ref{lem:incl}(iv)). 
If $\Lambda_+$ is non-zero, 
we have to deal with slight modifications. 
Then $e^{\frac{\pi i}{2} \Lambda_+}$ is a unitary 
operator, and 
\begin{equation}
 \label{eq:fixsp}
\sV_\cK^\sharp := e^{-\frac{\pi i}{2}\Lambda_+} \sV_\cK \quad \mbox{ and } \quad 
\sV_\cK^\flat := e^{\frac{\pi i}{2}\Lambda_+} \sV_\cK' 
\end{equation}
are standard subspaces of $\cK$. The corresponding Tomita operators are 
\begin{equation}
  \label{eq:taujrel}
T_\cK^\sharp :=  e^{-\frac{\pi i}{2}\Lambda_+} T_\cK  e^{\frac{\pi i}{2}\Lambda_+} 
  = T_\cK e^{\pi i \Lambda_+} = J_\cK e^{\pi i \Lambda}  = e^{-\pi i \Lambda} J_\cK 
\end{equation}
and 
\begin{equation}
  \label{eq:taujrel2}
T_\cK^\flat := 
  e^{\frac{\pi i}{2}\Lambda_+} T_\cK^*  e^{-\frac{\pi i}{2}\Lambda_+} 
 = T_\cK^* e^{-\pi i \Lambda_+}= J_\cK e^{-\pi i \Lambda} = e^{\pi i \Lambda} J_\cK.
\end{equation}
We observe that 
\begin{equation}
 \label{eq:flatvrel}
 (\sV_\cK^\sharp)' 
= e^{-\frac{\pi i}{2}\Lambda_+} \sV_\cK' 
\ {\buildrel \eqref{eq:fixsp} \over =}\ e^{-i \pi \Lambda_+} \sV_\cK^\flat.
\end{equation}
We also note that 
\begin{equation} \label{eq:jflatrel}
J_\cK T_\cK^\sharp J_\cK 
= e^{\pi i \Lambda} J_\cK 
= J_\cK e^{-\pi i \Lambda} = T_\cK^\flat 
\quad \mbox{ implies } \quad 
J_\cK \sV_\cK^\sharp  = \sV_\cK^\flat.
\end{equation}

We will need the following lemma in Proposition~\ref{prop:supp-control}: 
\begin{lem}\mlabel{lem:VsharpVflat}
The subspaces $\sV_\cK^\sharp$ and $\sV_\cK^\flat$ coincide if and only
if $\Lambda =\Lambda_+$ and $\Spec(\Lambda_+) \subeq \Z$. 
\end{lem} 
\begin{prf} As $T_\cK^\sharp =e^{-2\pi i\Lambda}T_\cK^\flat$ by 
\eqref{eq:taujrel} and \eqref{eq:taujrel2}, 
the equality $\sV_\cK^\sharp = \sV_\cK^\flat$ is equivalent to $e^{2\pi i \Lambda} = \1$. 
By spectral calculus for unbounded normal operators, 
this is equivalent to $\Spec (\Lambda )\subeq \Z$, 
which immediately translates into the two conditions $\Lambda_- = 0$ 
and $\Spec (\Lambda_+)\subeq \Z$. 
\end{prf} 
 
\begin{ex} \mlabel{ex:3.2}
For $\cK = \C$, $J_\cK(z) = \oline z$, 
and $\Lambda = \Lambda_+=\lambda \1$, $\lambda \in \R$, we have 
\[ \sV_\cK = \R, \quad \sV_\cK^\sharp = e^{-\frac{\pi i}{2}\lambda} \R
\quad \mbox{ and } \quad 
\sV_\cK^\flat = e^{\frac{\pi i}{2}\lambda}\R. \]

Now assume that $ \lambda \in \Z$. If $\lambda =2\mu$ is even, 
then $e^{\frac{\pi i}{2}\Lambda}=e^{\pi i \mu}\1$ so that $\sV^\sharp=\sV^\flat =\sV_\cK$. But if $\lambda =2\mu +1$ is odd, then 
\[ \sV_\cK^\flat   = e^{\frac{ \pi i}{2} \Lambda_+} \sV_\cK'  = i \R = \sV_\cK^\sharp.\] 
So the two subspaces $\sV_\cK^\flat$ and $\sV_\cK^\sharp$ need not be equal to $\sV_\cK$,  
not even if they are equal. 
\end{ex}

\subsubsection{The real space $\cH^J$} 

On a  reproducing kernel space $\cH_K \subeq \Hol(\Xi,\cK)$, we now 
consider an antiunitary representation of the form 
\begin{equation}
  \label{eq:cocyc2}
  (U(r)F)(z) = \rho(r^{-1})^* F(r^{-1}.z), \qquad r \in \R^\times. 
\end{equation}
This means that $K_z U(r) = \rho(r^{-1})^* K_{r^{-1}.z}$. 
Replacing $r$ by $r^{-1}$ and $F$ by $K_z^*\eta$ we get 
\begin{equation}
  \label{eq:genact}
 U(r) K_z^*\eta = K_{r.z}^* \rho(r)\eta \quad \mbox{ for } 
\quad r \in \R^\times, z \in \Xi, \eta \in \cK. 
\end{equation} 
For the kernel $K$, this corresponds to the equivariance condition 
\begin{equation}
  \label{eq:kercov}
  K(r.z, r.w) = \rho(r^{-1})^* K(z,w) \rho(r^{-1}), \qquad 
z,w \in \Xi, r \in \R^\times_+.
\end{equation}

We are interested in a more concrete description of the standard 
subspace $\sV$ associated to the pair $(\Delta, J)$, specified by 
\[ J := U(-1) \quad \mbox{ and }\quad 
\Delta^{-it/2\pi} = U(e^t), \quad t \in \R,  \] 
in terms of an injective boundary value map 
\[ b \: \cH_K \to C^{-\infty}(M)\] 
in distributions on $M \subeq \partial\Xi$. 
In Section~\ref{sec:3} we shall study the case where 
$M$ is a finite dimensional vector space $E$  
and $\Xi = E+i C^0$ a tube domain specified by a convex cone in~$E$ 
(see \cite{NO21a,NO21b, NO21c} for more general situations).

To this end, we would like to use Proposition~\ref{prop:standchar} 
which describes $\sV$ in terms of the real subspace 
$\cH^J_K$ of $J$-fixed elements. 
This space is easily 
characterized by Lemma~\ref{lem:hjvec} below. 
The corresponding standard subspaces are harder to describe 
because they require information on analytic extensions of orbit maps 
of elements of $\cH_K^J$ to the closure of the strip~$\cS_{-\pi/2,\pi/2}$. 

\begin{lem} \mlabel{lem:hjvec} 
Suppose that the submanifold $\Xi^{\tau_\Xi}$ of $\tau_\Xi$-fixed 
points is not empty. Then, for every open subset 
$\cO \subeq \Xi^{\tau_\Xi}$, we have 
\begin{align*}
\cH_K^J &
=\{F\in\cH_K\: (\forall z\in \Xi^{\tau_\Xi}) \, F(z)\in \cK^{J_\cK}\}
= \oline{ \spann_\R \{ K_z^* \eta  \: z \in \cO, \eta \in\cK^{J_\cK} \}}.
\end{align*}
\end{lem}

\begin{prf} We have $JF(z)=J_\cK F(\tau_\Xi z)$. Hence
$\cH_K^J =\{F\in\cH_K\: (\forall z\in \Xi^{\tau_\Xi}) \, F(z)\in \cK^{J_\cK}\}$,  
as any holomorphic function on $\Xi$ is uniquely determined by 
its restriction to the totally real submanifold $\Xi^{\tau_\Xi}$ 
because $\Xi$ is connected.

To verify the second equality, let 
$\cE \subeq \cH_K$ denote the right hand side. Then \eqref{eq:genact}, 
applied to $r = -1$, 
implies that $\cE \subeq \cH_K^J$. It remains to show that $\cE$ is total 
in $\cH_K$. To verify this claim, suppose that $F \in \cH_K$ is orthogonal 
to $\cE$. Then $\la K_z^*\eta, F \ra = \la \eta, F(z)\ra= 0$ for $z \in \cO$, 
$\eta \in \cK^{J_\cK}$. As $\cK^{J_\cK}$ generates $\cK$ as a complex 
Hilbert space, it follows that $F\res_{\cO} = 0$. Since $\cO$ is open 
in the totally real submanifold $\Xi^{\tau_\Xi}$ of 
the connected complex manifold $\Xi$,  
it follows that $F = 0$. 
Therefore the closed real subspace 
$\cE$ is total in $\cH_K$, hence coincides with~$\cH_K^J$. 
\end{prf}

\section{Standard subspaces and  tube domains} 
\mlabel{sec:3}

This is the core 
section of the article, culminating in Theorem \ref{thm:2.1}, where
we characterize the standard subspace $\sV$ corresponding to an 
anti-unitary representation
$U$ of $\R^\times$ in a vector-valued $L^2$-space. 
The setup is as follows. We
consider tube domains $\cT := E+iC^0\subset E_\C$ 
in the following environment: 
\begin{itemize}
\item[\rm(A1)] $E$ is a finite dimensional real vector space.
\item[\rm(A2)] $h \in \End(E)$ 
is diagonalizable with eigenvalues $\{-1,0,1\}$ and 
$\tau := e^{\pi i h}$.
\item[\rm(A3)] $C \subeq E$ is a pointed, generating 
closed convex cone invariant under $e^{\R h}$ and $-\tau$.
\end{itemize}

After discussing these conditions in Subsection~\ref{subsec:3.1}, we 
study in this section standard subspaces of vector-valued 
$L^2$-spaces $\cH = L^2(E^*, \mu;\cK)$, where 
$\cK$ is a Hilbert space and $\mu$ is a $\Herm^+(\cK)$-valued tempered 
measure supported in the dual cone $C^\star$. Then we have a natural realization of 
$\cH$ as a Hilbert space $\cH_{\hat\mu} \subeq \cS'(E;\cK)$ of $\cK$-valued 
tempered distributions and all these distributions extend to holomorphic 
functions~$\cT \to \cK$. Under suitable invariance conditions on the 
measure $\mu$, all these Hilbert spaces carry a natural antiunitary representation 
of $\R^\times$, corresponding to the geometric action on $E$ specified by the pair 
$(h,\tau)$ (Subsection~\ref{subsec:3.2}). 
Our main results are obtained in Subsection~\ref{subsec:3.3}, where 
we identify the standard subspace $\sV \subeq \cH_{\hat\mu} \subeq \cS'(E;\cK)$ 
as the real subspace generated by acting with real-valued test functions 
supported on a certain wedge domain $W \subeq E$ on a real subspace $\sV_\cK \subeq \cK$ 
(Theorem~\ref{thm:2.1}).

Writing $E_\lambda = E_\lambda(h) := \ker(h - \lambda\1)$ for the $h$-eigenspaces 
and $E^\pm := \ker(\tau \mp \1)$ for the $\tau$-eigenspaces, 
(A2) implies 
\begin{equation}\label{eq:decomp}
E=E_{1}\oplus E_0 \oplus E_{-1}, \quad
E^- = E_1\oplus E_{-1}, \quad \mbox{ and } \quad
E^+ =E_0.
\end{equation}
Accordingly, we write $x= x_1 + x_0+ x_{-1}$ for the decomposition of $x\in E$ 
into $h$-eigenvectors. 
As we shall see below, (A3) implies in particular that the wedge domain 
\begin{equation}
  \label{eq:w-decomp0}
W := W(E,C,h) := C_+^0 \oplus E_0 \oplus C_-^0 \quad \mbox{ for } \quad 
C_\pm := \pm C \cap E_{\pm 1}
\end{equation}
is nonempty. For generalizations of such configurations 
to non-abelian Lie groups and their properties, 
we refer to \cite{Ne19a, Ne19b, NO21b, NO21c, Oeh21}.

\begin{exs} \mlabel{ex:1} 
(a) The simplest examples arise for $h = \id_E$, 
$\tau = - \id_E$, and a pointed generating closed convex cone $C \subeq E$. 
Then $W = C^0 = C_+^0$. 

\nin (b) An example of importance in physics arises from 
$d$-dimensional Minkowski space $E = \R^d$ with the Lorentzian scalar product 
\[ (x_0,\bx)(y_0,\by) = x_0 y_0 - \bx \by \quad \mbox{ for } \quad 
x_0, y_0 \in \R, \bx,\by \in \R^{d-1}.\] 
Then the {\it upper light cone}
\[ C := \{ (x_0,\bx) \: x_0 \geq 0, x_0^2 \geq \bx^2 \} \] 
is pointed and generating. 
We consider the 
generator $h \in \so_{1,d-1}(\R)$ of the Lorentz boost in the 
$(x_0,x_1)$-plane 
\[ h(x_0,x_1, \ldots, x_{d-1}) = (x_1, x_0, 0,\ldots, 0).\] 
It is diagonalizable with the eigenvalues $0,-1,1$, and the eigenspaces are
\[ E_{\pm } = \R(e_0 \pm e_1) \quad \mbox{ and }\quad 
E_0 = \{(0,0)\} \times \R^{d-2}.\] 
For $\tau =e^{\pi i h}$, we obtain 
\[ \tau(x_0,x_1, \ldots, x_{d-1}) = (-x_0, -x_1, x_2, \ldots, x_{d-1}) 
\quad \mbox{ and } \quad E^- = \R^2\times \{0\}.\] 
The two cones
$ \pm C_\pm = C \cap E_{\pm } = \R_+ (e_0 \pm e_1)$ 
are simply half-lines, so that 
\[ W = \R^\times_+ (e_1 + e_0) \oplus \R^\times_+ (e_1 - e_0) \oplus \R^{d-2},\] 
is the standard open right wedge. 

\nin (c) On $E = \R^3$ with basis $e_1, e_2, e_3$, we consider the matrices 
\[ h = \diag(1,-1,0) \quad \mbox{ and } \quad 
\tau = \diag(-1,-1,1).\] 
We now describe all cones $C$ satisfying (A3).
Up to sign choices, we may assume that 
\[ C_+ = \R_+ e_1 \quad \mbox{ and }  \quad C_- = -\R_+ e_2\] 
(cf.\ Lemma~\ref{lem:Project} below).
As $C$ is generating and $-\tau$-invariant, it is determined by the cone 
$\{x \in C \: x_3 \geq 0\}$, which in turn is determined by 
the closed convex subset $D \subeq \R^2$, given by 
\[ D = \{ (x_1, x_2) \in \R^2 \: (x_1, x_2, 1) \in C \}.\]
This set has to be closed, convex, 
contained in $\R_+ e_1 \oplus \R_+ e_2$ (Lemma~\ref{lem:Project}), 
not containing $(0,0)$ (to ensure that $C$ is pointed), 
and invariant under $e^{\R h}$. 
This only leaves the sets 
\[ 
D_m := \{ (x_1, x_2) \: x_1, x_2 > 0, x_1 x_2 \geq  m\}, \ m > 0.\] 
Then 
\[ C^m := \R_+ (e_3 + D_m) \cup \R_+ (-e_3 + D_m) \] 
is a closed convex pointed generating 
cone satisfying (A3). Up  to sign changes, 
the cones satisfying (A3) are all of this form. 
They are Lorentzian with respect to the Lorentzian quadratic form 
\[ q(x_1, x_2, x_3) = x_1 x_2 - m x_3^2, \] 
so they all arise from a $3$-dimensional Minkowski space as in~(b).
\end{exs}

\subsection{The tube and associated wedge domains} 
\mlabel{subsec:3.1} 

In this section we focus on the tube domain $\cT =E +iC^0$, the wedge  $
W = C_+^0 \oplus E_0 \oplus C_-^0$ introduced in (\ref{eq:w-decomp0}) 
and the holomorphic extension of the one-parameter group 
$(U(e^t))_{t \in \R}$. The main result is Lemma \ref{lem:4.2}.

The
tube domain $\cT$
is obviously invariant under $e^{\R h}$ and $-\tau$, where we use the same 
notation for the complex linear extensions to $E_\C$. 
Denote by $\btau : E_\C\to E_\C$ the \textit{conjugate linear} extension of
$\tau $ to $E_\C$. Then
\[E^c:=(E_\C)^{\btau}=E^+ + iE^-.\]
As $\btau$ acts on $iE$ as $-\tau$ and $C$ is $-\tau $ invariant, 
$\btau (\cT)=\cT$, and 
\[\cT^{\btau} =  \cT \cap E^c = E^+ + i (C^0 \cap E^-) \] 
is the cone of $\btau$-fixed points in $\cT$, a real tube domain.

\begin{lemma}\mlabel{lem:Project}
For the projections 
\[ p_{\pm 1}: E\to E_{\pm 1}, x \mapsto x_{\pm 1}, \quad \mbox{ and } \quad 
p^-:E\to E_1\oplus E_{-1}=E^-, x \mapsto x_1 + x_{-1} = \frac{1}{2}(x-\tau x),\] 
the  following assertions hold:
\begin{itemize}
\item[\rm (i)] $p_{\pm 1}(C)=\pm C_\pm$ 
and $p_{\pm 1}(C^0)=\pm C_\pm^0\not=\eset$. 
\item[\rm (ii)] $p^- (C)=C\cap E^-=C_+\oplus -C_-$ and 
$p^- (C^0)=C^0\cap E^-=C^0_+\oplus -C_-^0$,
\item[\rm (iii)] $C\subeq C_+\oplus E_0 \oplus -C_-$. 
\end{itemize}
\end{lemma}

\begin{prf}  (i) As $\pm C_\pm \subset C$, we have 
$\pm C_\pm \subset p_{\pm 1}(C)$. Using the $e^{th}$-invariance of $C$
and writing $x=x_1+x_0+x_{-1}$ as before, 
$e^{th}x=e^{t}x_1 + x_0 + e^{-t} x_{-1}.$ 
Now take the limit $t\to \infty$ to see that 
\[ C \ni e^{-t}e^{th}x = x_1 +e^{-t}x_0+e^{-2t}x_{-1}\to x_1  \quad \mbox{as } \quad t\to \infty.\]
We likewise get 
$x_{-1} = \lim_{t \to -\infty}   e^{t}e^{th}x \in C$. 
It follows that $x_\pm \in \pm C_\pm$, so that 
$p_{\pm 1}(C) = \pm C_\pm$. 
As $p_{\pm 1}$ are projections and $C^0 \not=\eset$, it follows that 
$p_{\pm 1}(C^0) \subeq \pm C_{\pm}^0$. To obtain equality, it suffices to observe that 
$C_+^0 \oplus - C_-^0 \subeq (E^- \cap C)^0 \subeq C^0$ follows from 
$-\tau(C) = C$. 

\nin (ii) The two leftmost equalities follow from $\tau(C) = - C$, 
and the second two rightmost equalities from (i) and $p^- = p_1 + p_{-1}$. 

\nin (iii) follows from (ii). 
\end{prf}

We now describe the wedge $W$ and its closure $\oline W=C_+\oplus E_0\oplus C_-$
in terms of the tube domain $\cT$:

\begin{lem}\mlabel{lem:4.2} 
The wedge $W$ and the tube $\cT$ are related as follows: 
\begin{itemize}
\item[\rm(1)] $W=\{x\in E\: (\forall z\in \cS_\pi)\, e^{zh}x\in \cT\}$
\item[\rm(2)] $\oline W=\{x\in E\: (\forall z\in \cS_\pi)\, e^{zh}x\in E+iC\}.$
\end{itemize}
\end{lem}

 \begin{prf} For $z=a+ib \in \cS_\pi$ and $x = x_1 + x_0 + x_{-1}$, we have
 \begin{align}
 e^{z h}(x_1+x_0+x_{-1})&= e^{z}x_1 + x_0 +e^{-z}x_{-1}\nonumber\\
&=\cos (b)(e^ax_1 + e^{-a}x_{-1}) +x_0 +i\sin (b)(e^ax_1-e^{-a}x_{-1}),\label{eq:eih}
\end{align}
with $\sin(b)> 0$. 
As the imaginary part determines whether this element is contained 
in $\cT$, we see with 
Lemma~\ref{lem:Project}(ii) that $e^{zh} \in \cT$ holds for every $z \in \cS_\pi$ 
if and only if $x_1 \in C_+^0$ and $x_{-1} \in C_-^0$. 
Likewise $e^{zh} \in \oline{\cT} = E + i C$ holds for every $z \in \cS_\pi$ 
if and only if $x_1 \in C_+$ and $x_{-1} \in C_-$. 
\end{prf}

\subsection{Operator-valued 
measures and the corresponding Hilbert spaces} 
\mlabel{subsec:3.2}

In this subsection we construct the Hilbert spaces $\cH$ that we are 
interested in, first as vector-valued $L^2$-spaces $ L^2(E^*,\mu;\cK)$, 
defined by $\Herm^+(\cK)$-valued measures $\mu$ on the dual 
cone $C^\star \subeq E^*$ (cf.\ \cite[Thm.~B.3]{NO15}, \cite{Ne98}). 
Any such Hilbert space $\cH$ carries a natural 
antiunitary representation of the group $G=E\rtimes_\sigma\R^\times$ 
(Lemma \ref{lem:UnitRep}), but this representation has several 
other interesting realizations. 
In Lemma \ref{lem:4.13} we realize it as 
a reproducing kernel Hilbert space of holomorphic functions on $\cT$
and use that realization in Lemma~\ref{lem:hjv} to describe the space of
$J$-fixed elements. We also describe a third realization of
this representation in the space of distributions generated by the 
positive definite operator-valued distribution $\hat\mu$ on~$E$. 
The latter two realizations are connected by taking 
suitable boundary values. In particular, the Fourier transform 
$\tilde{\mu}(z)=\int_{C^*} e^{i\lambda (z)}d\mu (\lambda)$ will play an
important role in the proof of Theorem~\ref{thm:2.1}, 
both, as a holomorphic function on $\cT$ and as a distribution on~$E$.

Let $\cK$ be a separable Hilbert space and 
$\mu$ be a tempered $\Herm^+(\cK)$-valued Borel measure on $E^*$,  
supported in the dual cone 
\[ C^\star = \{ \lambda \in E^* \: \lambda(C) \subeq [0,\infty) \}.\] 
We then define the Hilbert space 
\[ \cH := L^2(E^*,\mu;\cK)  = L^2(C^\star,\mu;\cK)  \] 
 of measurable functions
$f: E^*\to \cK$ such that the norm of $f$ with respect to the  scalar product 
\[ \la f, g \ra = \int_{E^*} \la f(\lambda), d\mu(\lambda) g(\lambda)\ra_\cK\] 
is finite. We refer to \cite{Ne98} for more details on operator-valued measures 
and the corresponding $L^2$-spaces. 

We let $G:=E\rtimes_\sigma \R^\times$ 
with Lie algebra $\fg = E\rtimes_h \R$, where 
$\sigma(e^t)(v) = e^{th} v$ and $\sigma(-1) = \tau$.
The involution $\tau$ extends naturally to $G$ 
by $\tau_G(v,r) = (\tau(v),r)$. 
We assume that 
\begin{equation}
  \label{eq:mug}
 (-\tau)_*\mu = J_\cK \mu J_\cK
\quad \mbox{ and } \quad 
\sigma(r)_*\mu = \rho(r)^* \cdot \mu \cdot \rho(r)
\quad \mbox{ hold for } \quad r > 0. 
\end{equation}

\begin{ex} If $h = 0$ and $\tau = \id_E$, then 
the assumption that $C$ is pointed and generating, combined with 
$-C = \tau(C) = C$, leads to $C = \{0\}$ and hence to $E = \{0\}$. 
Then we may assume that $\mu(\{0\}) = \1_\cK$, and 
\eqref{eq:mug} means that the operators $\rho(e^t)$, $t \in\R$, 
are unitary. So $(\rho,\cK)$ is an antiunitary 
representation of $\R^\times$ on $\cK$ which coincides 
with $U^{\sV_\cK}$ (cf.\ Section~\ref{rem:3.1}).  
\end{ex}

\begin{lem}\mlabel{lem:UnitRep}
 We obtain an antiunitary representation of 
$G$ on $\cH = L^2(E^*,\mu;\cK)$ by 
\begin{align}
 (U(x,\1) f)(\lambda) &= e^{i\lambda(x)}f(\lambda), \label{e10} \\
(U(0,e^t) f)(\lambda) &
= \rho(e^t) f(\lambda \circ e^{th}) 
= \rho(e^t) ((e^{th})_* f) (\lambda), \label{e11} \\ 
(U(0,-1) f)(\lambda) &= J_\cK f(-\lambda \circ \tau).\label{e12} 
\end{align}
\end{lem}

\begin{prf}
The relations \eqref{eq:mug} lead for $f,h \in L^2(E^*,\mu;\cK)$ 
to the transformation formulas 
\begin{align}
  \int_{E^*} \la f(\sigma(r)\lambda), d\mu(\lambda) h(\sigma(r)\lambda)\ra
& =   \int_{E^*} \la \rho(r) f(\lambda), d\mu(\lambda) \rho(r) h(\lambda)\ra  \nonumber\\
&= \int_{E^*} \la f(\lambda),\rho(r) ^* d\mu(\lambda) \rho(r) h(\lambda)\ra 
\label{eq:trag} 
\end{align}
and 
\begin{equation}
  \label{eq:trag2}
  \int_{E^*} \la f(-\tau\lambda), d\mu(\lambda) h(-\tau\lambda)\ra 
=   \int_{E^*} \la f(\lambda), J_\cK d\mu(\lambda) J_\cK h(\lambda)\ra 
=   \int_{E^*} \la J_\cK h(\lambda), d\mu(\lambda) J_\cK f(\lambda)\ra. 
\end{equation}
This implies that $U(r)$ is (anti-)unitary for $r \in \R^\times$. 
That $U$ is a homomorphism is a standard calculation. 
\end{prf}

\begin{rem} 
The assumption $\supp(\mu) \subeq C^\star$ is equivalent to 
\[  C \subeq C_U := \{ x \in \fg \: -i \partial U(x) \geq 0\},\]
where $\partial U(x) = \derat0 U(\exp tx)$ is the infinitesimal generator 
of the unitary one-parameter group $(U(\exp tx))_{t \in \R}$. 
\end{rem}

\begin{rem} \mlabel{rem:4.4} 
All Schwartz functions in $\cS(E^*;\cK)$ define elements of $L^2(E^*,\mu;\cK)$ because 
$\mu$ is tempered. This leads to the embedding 
\[ \Psi \: L^2(E^*,\mu;\cK) \to \cS'(E^*;\cK), \qquad 
\Psi(f) =  \mu f,\quad \mbox{ resp.,} \quad 
\Psi(f)(\phi) = \int_{E^*} \oline{\phi(\lambda)}\, d\mu(\lambda) f(\lambda).\] 
As the function $e_{iz}(\alpha) =e^{i\alpha(z)}$ on $C^\star$ satisfies 
$|e_{iz}| \leq 1$ for $z \in E + i C = \oline{\cT}$, the map 
\[ \Gamma \: E + i C  \to \cS'(E^*;B(\cK)), \quad z \mapsto e_{iz} \mu \] 
is defined, weakly continuous 
(by the Dominated Convergence Theorem), 
and weakly holomorphic on the interior $\cT$. 
We further obtain with \eqref{eq:mug} 
\[ \Psi(U(0,e^{t})f) 
= \mu \cdot (\rho(e^{t}) e^{th}_*f) 
\buildrel {\eqref{eq:mug}}\over {=} \rho(e^{-t})^* (e^{th}_*\mu) \cdot e^{th}_*f 
= \rho(e^{-t})^* e^{th}_*\Psi(f).\] 
As 
\[ \big((e^{th}_*) e_{iz}\big)(\lambda) 
= e_{iz}(\lambda \circ e^{th})
= e^{i \lambda(e^{th}z)} = e_{i e^{th}z}(\lambda),\] 
we have the equivariance relation 
\begin{equation}
  \label{eq:gammatrans}
\Gamma(e^{th}z) = \rho(e^{-t})^* (e^{th})_* \Gamma(z).
\end{equation}
Therefore Lemma~\ref{lem:4.2}(b) implies that 
$\Gamma$ maps the closed wedge 
$\oline W$ into distributions which may produce elements of 
the standard subspace $\sV$ 
when smeared with suitable test functions (cf.~Proposition~\ref{prop:standchar}). 
\end{rem}

\begin{lem} \mlabel{lem:4.8}
 The Fourier transform \begin{footnote}{Note that 
$\tilde\mu(z) = \hat\mu(-z)$, also on the level of distribution 
boundary values. In our context it minimizes the number of artificial 
minus signs to work with $\, \tilde{\,\, }\ $ instead of  $\, \hat{\,\, }\ $. 
As 
$\tilde\mu(\phi) = \int_{E^*} \tilde{\oline\phi}\, d\mu$ and 
$\hat\mu(\phi) 
= \int_{E^*} \oline{\tilde\phi}\, d\mu
= \int_{E^*} \hat{\oline\phi}\, d\mu$, we have on real-valued test functions 
$\phi \in C^\infty_c(E^*,\R)$ the relations  
$\tilde\mu(\phi) = \int_{E^*} \tilde\phi\, d\mu$ and 
$\hat\mu(\phi)= \int_{E^*} \hat\phi\, d\mu.$ This means that 
$\hat\mu = \oline{\tilde\mu}$ as distributions on~$E$.
}   
 \end{footnote}

\begin{equation}
  \label{eq:muft}
 \tilde\mu(z) 
:= \int_{E^*} e^{i\lambda(z)}\, d\mu(\lambda) 
= \int_{C^\star} e^{i\lambda(z)}\, d\mu(\lambda) \in B(\cK), \quad 
z \in \cT = E + i C^0, 
\end{equation}
defines a holomorphic  function on $\cT$ with  the following properties: 
\begin{itemize}
\item[\rm(a)] The holomorphic function $\tilde\mu$ on $\cT$ has distributional 
boundary values in $\cS'(E;B(\cK))$ in the sense that the 
 tempered distribution $\tilde\mu$, defined by 
$\tilde\mu(\phi) := \int_{E^*} \tilde{\oline \phi}\, d\mu$ 
satisfies 
\begin{equation}
  \label{eq:pw}
 \tilde\mu(\phi) = \lim_{C^0 \ni y \to 0} 
\int_E \oline{\phi(x)} \tilde\mu(x + i y)\, d\mu_E(x) \quad \mbox{ for } \quad 
\phi \in \cS(E). 
\end{equation}
\item[\rm(b)]
For the antilinear extension $\btau (x + i y) = \tau(x) - i \tau(y)$ 
of $\tau $ to $E_\C$, 
we have 
\begin{equation}
  \label{eq:hatmutra}
\tilde\mu(e^{-th}z)  = \rho(e^t)^* \tilde\mu(z) \rho(e^t) 
\quad \mbox{ and } \quad 
\tilde\mu(\btau z) = J_\cK \tilde\mu(z) J_\cK. \end{equation}
\end{itemize}
\end{lem}

\begin{prf} (a) This follows easily from Fubini's theorem and 
Lebesgue's theorem on dominated convergence because 
$e^{-\lambda(y)} \leq 1$ for $\lambda\in C^\star$ and $y \in C$:
\begin{align*}
 \int_E \oline{\phi(x)} \tilde\mu(x + i y)\, d\mu_E(x) 
&=  \int_E \oline{\phi(x)} \int_{C^\star} e^{i\lambda(x + i y)}\, d\mu(\lambda)\, d\mu_E(x) \\
&=  \int_E \int_{C^\star} e^{i\lambda(x)}\oline{\phi(x)} e^{-\lambda(y)}\, d\mu(\lambda)\, d\mu_E(x) \\
&=  \int_{C^\star} \tilde{\oline\phi}(\lambda)\, e^{-\lambda(y)}\, d\mu(\lambda)
\ \overset{y \to 0}{\ssarr}\ 
\int_{C^\star} \tilde{\oline\phi}(\lambda)\, d\mu(\lambda)
= \tilde\mu(\phi).\end{align*}

\nin (b) The first formula is a direct consequence of the transformation 
properties \eqref{eq:mug}. The second formula follows from 
\[ \tilde\mu(\oline{\tau}(z)) = \int_{E^*} e^{i \lambda(\tau\oline z)}\, d\mu(\lambda)
= \int_{E^*} e^{-i \lambda(\oline z)} J_\cK\ d\mu(\lambda) J_\cK
= J_\cK \Big(\int_{E^*} e^{i \lambda(z)} \ d\mu(\lambda)\Big) J_\cK. 
\qedhere\]
\end{prf}

\begin{rem} The covariance relation in Lemma~\ref{lem:4.8} has an interesting 
consequence. 
For $x = x_0 + x_1 + x_{-1} \in W$, 
we have $e^{zh} x \in \cT$ for $z \in \cS_\pi$ by Lemma~\ref{lem:4.2}. 
In particular 
$\iota(x) := e^{\frac{\pi i}{2}h}x = x_0 + i (x_1 - x_{-1}) \in \cT$. 
If $\Lambda$ is a bounded operator, we therefore 
expect for the boundary values of $\tilde\mu$ on the wedge domain $W$ a 
relation of the form 
\[  \tilde \mu(x) 
= \tilde\mu\big(e^{-\frac{\pi i}{2}h}\iota(x)\big)
= e^{\frac{\pi i}{2}\Lambda^*} \tilde\mu(\iota(x)) e^{\frac{\pi i}{2}\Lambda} 
= e^{\frac{\pi i}{2}\Lambda_+}e^{-\frac{\pi i}{2}\Lambda_-} 
\tilde\mu(\iota(x)) e^{\frac{\pi i}{2}\Lambda_-} e^{\frac{\pi i}{2}\Lambda_+}.\]
We shall see in Proposition~\ref{prop:analyticfunc} below 
that such a relation holds indeed in the sense 
that the distributional boundary values of $\tilde\mu$ on $E$ are represented 
on the open cone $W$ by an operator-valued function.
\end{rem}

For the proof of the Standard Subspace Theorem~\ref{thm:2.1} below  
we shall use the following realization of the unitary 
representation $(U,\cH)$ by holomorphic 
functions on the tube domain $\cT = E + i C^0$:

\begin{lem} \mlabel{lem:4.13}
The map 
\begin{align*}
\Phi &\: L^2(E^*,\mu;\cK) \to \Hol(\cT, \cK), \quad 
\Phi(f)(z) = \int_{E^*} e^{i\lambda(z)} d\mu(\lambda) f(\lambda), \\ 
&\la \xi, \Phi(f)(z) \ra  = \la e_{-i \oline z} \xi, \mu \cdot f \ra 
\quad \mbox{ for } \quad \xi  \in \cK, z \in \cT, 
\end{align*}
maps $L^2(E^*,\mu;\cK)$ injectively 
onto a reproducing kernel Hilbert space $\cH_K$ with $B(\cK)$-valued kernel 
\[ K(z,w) = \tilde\mu(z- \oline w).\]
In particular, $\cH_K$ is generated by the functions 
\begin{equation}
  \label{eq:5star}
(K_w^*\xi)(z) = K(z,w)\xi = \tilde\mu(z - \oline w)\xi, \quad 
w \in \cT, \xi \in \cK, 
\end{equation}
satisfying 
\begin{equation}
  \label{eq:Phirel}
\Phi(e_{-i \oline w} \xi) = K_w^* \xi.
\end{equation}
The antiunitary representation $(U,L^2(E^*,\mu;\cK))$ of $G$ 
is intertwined by $\Phi$ with the 
anti\-unitary representation $U'$ on $\cH_K$ given by 
\begin{align}
(U'(x,\1) F)(z) &= F(z +x), \label{e15} \\  
(U'(0,e^t) F)(z) &= \rho(e^{-t})^* F(e^{-t}z), \label{e16}  \\ 
(U'(0,-1) F)(z) &= J_\cK F(\btau z). \label{e17} 
\end{align}
On the generators $K_w^*\eta$, this leads to 
\begin{align}
U'(x,\1) K_w^*\eta &= K_{w-x}^*\eta, \label{e18} \\
U'(0,e^t) K_w^*\eta &= K_{e^{th}w}^*\rho(e^t)\eta \label{e19}\\
U'(0,-1) K_w^*\eta &= K_{\btau w}^* J_\cK \eta. \label{e20} 
\end{align}
\end{lem}

\begin{prf} As $\supp(\mu) \subeq C^\star$, for $z \in \cT$ we have 
$|e_{iz}(\lambda)| = |e^{i\lambda(z)}| \leq 1$ for $\mu$-almost all $\lambda \in E^*$ 
and the temperedness of $\mu$ implies that 
$e_{iz} \eta\in L^2(E^*,\mu;\cK)$ for $z \in \cT$ and $\eta \in \cK$ 
(\cite[Lemma~B.1]{HN01}). 
This implies that $\Phi$ is determined by the relation 
\begin{equation}
  \label{eq:evalform}
\la \eta, \Phi(f)(z) \ra 
= \la e_{-i\oline z} \eta, f \ra_{L^2} 
\quad \mbox{ for } \quad f \in L^2(E^*, \mu;\cK). 
\end{equation}
Hence the point evaluations on $\cH_K$ are continuous, and given by the 
scalar product with $\Phi(e_{-i\oline z}\eta)$, so that the reproducing 
kernel is given by 
\[ K(z,w) = K_z K_w^* = \int_{E^*} e^{i\lambda(z-\oline w)}\, d\mu(\lambda) 
= \tilde\mu(z - \oline w),\] 
resp., 
\begin{equation}
  \label{eq:kerform}
\la \xi, K(z, w) \eta \ra 
= \la e_{-i\oline z} \xi, e_{-i\oline w} \eta\ra 
 \quad \mbox{ for } \quad 
z,w\in \cT, \xi, \eta \in \cK.
\end{equation}

For $z,w \in \cT$ and $\xi,\eta \in \cK$, we derive from 
\eqref{eq:evalform}  the relation 
\[ \la \eta, K_z \Phi(e_{-i \oline w} \xi) \ra 
= \la \eta, \Phi(e_{-i \oline w} \xi)(z) \ra 
=  \la e_{-i \oline z} \eta, e_{-i \oline w} \xi \ra 
= \la \eta, \tilde\mu(z - \oline w) \xi \ra 
= \la \eta, K_z K_w^* \xi \ra, \] 
which implies \eqref{eq:Phirel}. 
The remaining assertions are easily verified. 
We refer to \cite[Thm.~III.9]{Ne98} for further details. 
\end{prf}
  From the general Lemma~\ref{lem:hjvec}, we obtain in particular: 

\begin{lem} \mlabel{lem:hjv} 
For $J := U'(0,-1)$ as in \eqref{e17} and 
$\eset\not= \cO \subeq E^+ + i (E^- \cap C^0)$ open, we have 
\[ \cH_K^J = \oline{ \spann_\R \{ K_z^* \eta  \: z \in \cO, \eta \in 
\cK^{J_\cK}  \}}=\{F\in \cH_K\: F(\cO) \subeq  \cK^{J_\cK}\}.\] 
\end{lem}

\begin{defn} \mlabel{def:2.3} 
For the positive definite tempered distribution 
$D := \hat\mu \in \cS'(E, B(\cK))$, defined by 
\[ D(\phi) := \int_{E^*} \oline{\tilde\phi(\lambda)}\, d\mu(\lambda) \] 
(cf.~\eqref{eq:2}), 
we write $\cH_D \subeq \cS'(E;\cK)$ for the 
corresponding reproducing kernel Hilbert space whose $B(\cK)$-valued 
kernel is given on 
$\cS(E)$ by 
\begin{equation}
  \label{eq:dag5}
\la \eta,  K_D(\phi, \psi) \xi \ra := 
\la \eta, D(\psi^* * \phi) \xi \ra 
= \int_{E^*} \oline{\tilde\phi(\lambda)} \tilde\psi(\lambda)\, \la \eta, d\mu(\lambda) \xi \ra
= \la \tilde\phi \eta, \tilde\psi \xi\ra_{L^2}.
\end{equation}
The Hilbert space $\cH_D$ is generated by the $\cK$-valued distributions 
\[ (\psi * D\eta)(\phi) := D(\psi^* * \phi)\eta
\quad \mbox{ for } \quad \psi \in C^\infty_c(E), \eta \in \cK\] 
satisfying
\begin{equation}
  \label{eq:pos-def-hilb}
\la \phi * D\eta, \psi * D\xi\ra = \la \eta, D(\psi^* * \phi)\xi\ra 
= \la \eta, K_D(\phi, \psi) \xi \ra 
\end{equation}
(cf.\ \cite[Def.~7.1.5]{NO18}). 
\end{defn}

\begin{rem} \mlabel{rem:4.16}
For the Hilbert space $\cH$ we now have four pictures: 
  \begin{itemize}
  \item[\rm(a)] as the $L^2$-space $\cH = L^2(E^*,\mu;\cK)$, 
  \item[\rm(b)] as a subspace $\Psi(L^2(E^*,\mu;\cK)) \subeq \cS'(E^*,\mu;\cK)$ 
(distributions on $E^*$) (Remark~\ref{rem:4.4}), and 
  \item[\rm(c)] as $\cH_D \subeq \cS'(E;\cK)$ (distributions on $E$) 
(Definition~\ref{def:2.3}), and 
  \item[\rm(d)] as the reproducing kernel space 
$\cH_K = \Phi(L^2(E^*,\mu;\cK)) \subeq \Hol(\cT;\cK)$ 
(Lemma~\ref{lem:4.13}))
 \end{itemize}
The realizations (b) and (c) are connected by the Fourier transform 
\[ \cF \: \cS'(E^*;\cK) \to \cS'(E;\cK), \quad 
D \mapsto \hat D, \qquad 
\hat D(\phi) := D(\tilde \phi).\] 
For a Schwartz function $\phi \in \cS'(E^*)$ and $\eta \in \cK$, we have 
\[ \cF(\Psi(\tilde\phi \eta)) = \cF(\mu \tilde\phi \eta) = \phi * \hat\mu \eta,\] 
so that 
\[ \cF \: \cS'(E^*;\cK) \supeq \Psi(L^2(E^*,\mu;\cK)) \to \cH_{\hat\mu} 
\subeq \cS(E;\cK) \] 
is unitary by \eqref{eq:dag5} and \eqref{eq:pos-def-hilb}. 
\end{rem}

\subsection{Standard subspaces from wedge domains} 
\mlabel{subsec:3.3}

In this section we prove one of the main results of this paper 
(Theorem \ref{thm:2.1}). It describes the standard subspace $\sV$ corresponding to 
$J=U(-1)$ and $\Delta^{-it/2\pi } = U(e^{t})$ for the antiunitary 
representation of $\R^\times$ on $L^2(E^*,\mu;\cK)$, introduced in 
Lemma \ref{lem:UnitRep}. In Corollary \ref{cor:2.1b} we also 
describe its symplectic complement $\sV^\prime$ in similar terms.
Most of the section is devoted to the proof of Theorem \ref{thm:2.1}.

Recall the notation from Section \ref{subsec:1.2}: We have
$\rho (e^t ) = e^{t\Lambda}$ with $\Lambda=\Lambda_+ + \Lambda_-$
with $\Lambda_+$ bounded and symmetric, $\Lambda_-$
skew-adjoint and possibly  unbounded, and $[\Lambda_+,\Lambda_-]=0$.
On $\cK$ we define the Tomita operator  $T_\cK = J_\cK \Delta_\cK^{1/2}$.
The corresponding standard subspace is  $\sV_\cK$, and we also 
consider 
\[\sV_\cK^\sharp=e^{-\frac{\pi i}{2}\Lambda_+}\sV_\cK\qquad \text{and}\qquad 
\sV_\cK^\flat=e^{\frac{\pi i}{2}\Lambda_+}\sV_\cK .\]

The goal of this subsection is to prove Theorem~\ref{thm:2.1} below. 
To formulate it, we introduce for a standard subspace $\sW\subset \cK$ and an open subset 
$\eset\not=\cO\subeq E$, the real subspace 
\begin{equation}
  \label{eq:realsubs}
\sH_{\sW}(\cO)  := 
\oline{\spann_\R \{ \tilde\phi \eta \: \phi \in C^\infty_c(\cO,\R), 
\eta \in \sW\}} \subeq L^2(E^*,\mu;\cK),
\end{equation}
where $\tilde\phi(\lambda) = \int_E e^{i\lambda(x)}\phi(x)\, d\mu_E(x).$ 

\begin{thm} \mlabel{thm:2.1} {\rm(Standard Subspace Theorem)} 
Let $U$ be as in  \eqref{e10}-\eqref{e12} and recall the wedge domain 
\[ W = C_+^0 \oplus E_0 \oplus C_-^0 \subeq E. \]  
Then the standard subspace $\sV \subeq L^2(E^*,\mu;\cK)$, 
defined by $J_\sV = U(0,-1)$ and $\Delta_\sV^{-it/2\pi} = U(0,e^{t})$   
is 
\begin{equation}
  \label{eq:standloc}
 \sV =\sH_{\sV_\cK^\sharp}(W) =e^{-\frac{\pi i}{2}\Lambda_+} \sH_{\sV_\cK}(W).
\end{equation}
\end{thm}

Combining Theorem~\ref{thm:2.1} with  Example~\ref{ex:3.2}, we obtain: 

\begin{cor} \mlabel{cor:2.1} 
If $\cK = \C$, $J_\cK(\eta) = \oline \eta$, and $\Lambda_- = 0$, then 
$\sV = e^{-\frac{\pi i \Lambda}{2}}\sH_\R(W).$ 
\end{cor}

\begin{cor}\mlabel{cor:2.1b} Let the assumptions and notation be as in 
{\rm Theorem \ref{thm:2.1}}. Then the
symplectic complement of $\sV$ is 
$\sV^\prime = \sH_{\sV_\cK^\flat}(-W)
=e^{\frac{\pi i}{2}\Lambda_+} \sH_{\sV_\cK}(-W).$ 
\end{cor}

\begin{prf} For $\phi \in C^\infty_c(-W,\R)$ we have 
$\psi := \phi \circ \tau \in C^\infty_c(W,\R)$, so that the functions 
$\tilde{\psi}\eta \in \sV$, $\eta \in \sV_\cK^\sharp$ 
generate~$\sV$ by Theorem~\ref{thm:2.1}. We also recall from 
\eqref{eq:jflatrel} that 
$J_\cK \sV_\cK^\sharp  = \sV_\cK^\flat$. 
Now the assertion follows from $J\sV = \sV'$ and 
\[ (J \tilde{\psi}\eta)(\lambda) 
=  J_\cK \tilde{\psi}(-\lambda \circ \tau) \eta 
=   \oline{\tilde{\psi}(-\lambda \circ \tau)} J_\cK \eta 
=   \tilde{\psi}(\lambda \circ \tau) J_\cK \eta 
=   \tilde{\phi}(\lambda)J_\cK \eta.\qedhere\] 
\end{prf}

We now prepare the notation and first steps for the proof of
Theorem \ref{thm:2.1}.
For $z \in \cT \cap E^c= E^+ \oplus i (C^0 \cap E^-)$ 
and 
\[ \eta \in \cK^{J_\cK} \cap \cD(\Delta_\cK^{1/4}) = \cK^{J_\cK} \cap \cD(\Delta_\cK^{-1/4}),\] 
we have $J K_z^* = K_z^* J_\cK$ since $\btau(z) = z$, and the orbit 
map  $\alpha^{K_z^*\eta}(t) = K^*_{e^{th}.z} \rho(e^{t}) \eta$  
extends holomorphically to $\cS_{-\pi/2,\pi/2}$. 
However, in general this orbit map need not extend continuously to the boundary 
if $\mu$ is an infinite measure. So we have to use some regularization procedure 
to construct elements of the standard subspace~$\sV$ 
by the characterization in 
Proposition~\ref{prop:standchar}(iv). 

Instead of $z \in \cT\cap E^c$, 
which specifies the element $K_z^*\eta \in \cH_K^J$, we 
consider the boundary value for $z = -\pi i/2$ in a smeared 
version. Recall the open wedge 
$W = C_+^0 \oplus E_0 \oplus C_-^0$. 
For $\phi \in C^\infty_c(E)$ and $z \in \cT$, we define 
 \begin{align}
K_\phi^*(z) 
&:= \int_E \phi(x) K_x^*(z)\, d\mu_E(x) 
\ {\buildrel \eqref{eq:5star} \over =}\ \int_E \phi(x) \tilde\mu(z -x)\, d\mu_E(x)\nonumber \\
&= \int_{E^*}e^{i\lambda(z)} \hat\varphi (\lambda )d\mu(\lambda)= (\phi * \tilde\mu)(z) \in B(\cK,\cH_K).  \label{eq:kphi}
\end{align}
Lebesgue's Dominated Convergence Theorem and \eqref{eq:kphi}, 
see also Lemma~\ref{lem:4.8}, imply that this   holomorphic operator-valued 
function has the distributional boundary values 
\[ \int_{E^*}e^{i\lambda (\cdot )} \hat\varphi (\lambda )d\mu (\lambda) =\phi * \tilde\mu \in C^\infty(E, B(\cK)) \cap \cS'(E, B(\cK)).\] 
We also note that the distributions 
$ \phi * \tilde\mu \eta$, $\phi \in C^\infty_c(E)$, $\eta \in \cK$, 
are contained in the Hilbert space 
$\cH_{\hat\mu} \subeq \cS'(E;\cK)$ (Definition~\ref{def:2.3}). 
With the notation $\phi^\vee(x) := \phi(-x)$, we have by (\ref{eq:kphi})
\[
 K_{\phi^\vee}^*(z) \eta  = \int_{E^*} e^{i\lambda(z)} \tilde\phi(\lambda)\, d\mu(\lambda)\cdot \eta 
= \Phi(\tilde\phi \cdot \eta)(z)\]
(where we used Lemma~\ref{lem:4.13} for the last equality), i.e., 
\begin{equation}
  \label{eq:4.5}
K_{\phi^\vee}^* \eta = \Phi(\tilde\phi \eta). 
\end{equation}

\begin{lem} \mlabel{lem:esti} 
Let $\phi \in C^\infty_c(W,\C)$ and $y \in [0,\pi]$. Then the functions 
\[ \tilde\phi_y(\lambda) := \tilde\phi(\lambda \circ e^{yi h}) = 
 \int_E e^{i\lambda(e^{yi h}x)}\, \phi(x)\, d\mu_E(x) \] 
 have the following properties: 
 \begin{itemize}
 \item[\rm(i)] $|\tilde\phi_y(\lambda)| \leq \|\phi\|_1$ for every $\lambda \in C^\star$. 
 \item[\rm(ii)] For every $k > 0$, there exists a constant $d_k$, such that 
\[ |\tilde\phi_y(\lambda)| \leq \frac{d_k}{1 + \|\lambda\|^{2k}} 
\quad \mbox{ for all } \quad \lambda \in C^\star, y \in [0,\pi].\] 
 \item[\rm(iii)] For $\eta \in \sV_\cK^\sharp$, 
the map $[0,\pi] \to L^2(E^*,\mu;\cK), y \mapsto \tilde\phi_y e^{iy\Lambda} \eta$ is 
continuous. 
 \item[\rm(iv)] For $\xi\in \cK,\eta \in \sV_\cK^\sharp$, 
the map $[0,\pi] \to \C, y \mapsto 
\int_{C^\star} \la \xi, \tilde\phi_y(\lambda)\, d\mu(\lambda) e^{iy\Lambda} \eta\ra$ 
is continuous. 
 \end{itemize}
\end{lem}

\begin{prf} (i) We clearly have for $y \in [0,\pi]$ the estimate 
\[ |\tilde\phi_y(\lambda)| \leq  \int_E e^{-\Im\lambda(e^{yi h}x)}|\phi(x)|\, d\mu_E(x), \] 
so that (i) follows from $\Im \lambda(e^{yih}x) \geq 0$ for $x \in W$ 
and $\lambda \in C^\star$, which in turn follows from Lemma~\ref{lem:4.2}. 

\par\nin (ii) If $P$ is a polynomial function on $E$ and 
$P(D)$ the corresponding constant coefficient differential operator on $E$, 
then (i) applies to $P(D)\phi \in C^\infty_c(W,\C)$. On the other hand, 
\begin{equation}
  \label{eq:square}
   (P(D)\phi)\, \tilde{}_y(\lambda) 
= (P(D)\phi)\, \hat{}_y(-\lambda) 
= P(-i\lambda \circ e^{i yh}) \hat \phi_y(-\lambda) 
= P(-i\lambda \circ e^{i yh}) \tilde\phi_y(\lambda). 
\end{equation}
Choosing coordinates on $E$ adapted to the $h$-eigenspaces 
and a scalar product for which $h$ is symmetric, we can choose 
the polynomial~$P$ in such a way that 
\[ |P(-i\lambda \circ e^{i yh})| \geq 1 + \|\lambda\|^{2k} 
\quad \mbox{ for all } \quad 
y \in [0,\pi], \lambda \in E.  \] 
With (i) and \eqref{eq:square}, this implies (ii).


\par\nin (iii) First we observe that, for $\eta \in \sV_\cK^\sharp$, we have 
\[ e^{iy\Lambda} \eta 
=  e^{iy\Lambda_+} e^{iy\Lambda_-} \eta 
=  e^{iy\Lambda_+} \Delta_\cK^{y/2\pi} \eta.\] 
Since $\Lambda_+$ is bounded by assumption (Section~\ref{rem:3.1}) 
and $\Delta_\cK = \Delta_{\sV_\cK^\sharp}$, 
this vector depends continuously on $y \in [0,\pi]$. 
Hence (iii) follows from (ii), the temperedness of $\mu$, and the 
Dominated Convergence Theorem. 

\par\nin (iv) As in (iii), this follows from 
the temperedness of $\mu$ and the Dominated Convergence Theorem. 
\end{prf}

\begin{lem} \mlabel{lem:2.4} 
For $\phi \in C^\infty_c(W,\R)$ and $\eta \in \sV_\cK^\sharp$, we have 
$\tilde\phi \eta\in \sV\subeq L^2(E^*,\mu)$. Moreover, 
the analytic continuation of the map 
$\R \to \cH, t \mapsto U(0,e^t)\tilde\phi \eta$, to $\cS_\pi$ is given by 
\begin{equation}
  \label{eq:anacontform}
 z \mapsto \tilde \phi_{-iz}(\lambda)e^{z\Lambda}\eta.  
\end{equation}
\end{lem}

\begin{prf} Let $\eta \in \sV_\cK^\sharp$. 
We consider the map 
  \begin{equation}
    \label{eq:gamma}
 \gamma_\eta \: \cS_\pi \to \cH_K, \qquad 
\gamma_\eta(z) := \Big(\int_E \phi(-x) K^*_{e^{\oline zh}x}\, d\mu_E(x)\Big) \cdot  
e^{z \Lambda}\eta.
  \end{equation}
To see that this map is defined, we first note that, 
for $z \in \cS_\pi$ and $x \in - W$, Lemma~\ref{lem:4.2} implies that 
$\Im(e^{\oline z h} x)=  \Im(e^{z h} (-x)) \in C^0,$ 
so that $e^{\oline z h}x \in \cT$. We therefore obtain a continuous map 
\[  \cS_\pi \times (-W) \to B(\cK,\cH_K), \quad 
(z,x) \mapsto K^*_{e^{\oline z h}x} \] 
which is  holomorphic in $z$. Integrating over the compact support of 
$\phi$, thus defines a holomorphic operator-valued 
function \[ F \: \cS_\pi \to B(\cK,\cH_\cK), \quad 
F(z) = \int_E \phi(-x) K^*_{e^{\oline zh}x}\, d\mu_E(x)\] 
(cf.~\cite[Thm.~2.1.12, Ex.~2.1.7]{GN}). We now have 
$\gamma_\eta(z) = F(z) e^{z\Lambda}\eta,$ 
where 
$\cS_\pi \to \cK, z \mapsto e^{z\Lambda}\eta = e^{z\Lambda_+} \Delta_\cK^{-iz/2\pi} \eta$ 
is a holomorphic $\cK$-valued function 
(Proposition~\ref{prop:standchar}). 
As the evaluation map 
\[ B(\cK,\cH_K) \times \cK \to \cH_K \] 
is complex bilinear and continuous, 
this implies the holomorphy of~$\gamma_\eta$ on $\cS_\pi$. 

By \eqref{e19}, $\gamma$ satisfies the equivariance relation 
\begin{equation}
  \label{eq:equiv}
\gamma_\eta(z + t) = U'(0,e^{t}) \gamma_\eta(z)  \quad \mbox{ for } \quad 
t\in \R, z \in \cS_\pi.
\end{equation}
Next we observe that $\eta = T_\cK^\sharp \eta =  e^{-\pi i \Lambda} J_\cK \eta$ 
leads to 
\begin{equation}
  \label{eq:jrel1}
J_\cK e^{z\Lambda} \eta 
=  e^{\oline z\Lambda} J_\cK \eta 
\ {\buildrel \eqref{eq:taujrel} \over =}\  e^{\oline z\Lambda} e^{\pi i \Lambda} \eta
=  e^{(\oline z + \pi i )\Lambda} \eta,
\end{equation}
so that \eqref{e20} further yields 
\[ J K^*_{e^{\oline zh}x} e^{z \Lambda}\eta
= K^*_{e^{zh} \tau(x)} J_\cK e^{z \Lambda}\eta  
\ {\buildrel \eqref{eq:jrel1} \over =}\  K^*_{e^{(z-\pi i)h} x} e^{(\oline z + \pi i )\Lambda} \eta.\]
We thus arrive at the relation 
 \begin{align}
   J\gamma_\eta(z) 
&= \int_E \phi(-x) J K^*_{e^{\oline zh}x} e^{z \Lambda}\eta\, d\mu_E(x)\nonumber \\
&= \int_E \phi(-x) K^*_{e^{(z-\pi i)h} x}e^{(\oline z + \pi i )\Lambda} \eta\, d\mu_E(x) 
= \gamma_\eta(\pi i + \oline z). 
\label{eq:equivb}
\end{align}
In view of \eqref{eq:equiv}, \eqref{eq:equivb}, and 
Proposition~\ref{prop:standchar}(iv), it remains to show 
that $\gamma_\eta$ extends continuously to the closed strip~$\oline{\cS_\pi}$  
with $\gamma_\eta(0) = \Phi(\tilde\phi \eta) = K^*_{\phi^\vee} \eta$ 
(cf.~\eqref{eq:4.5}), to verify that $\tilde\phi \eta \in \sV$. 
In view of the equivariance properties 
\eqref{eq:equiv} and \eqref{eq:equivb}, this boils down to showing that 
\begin{equation}
  \label{eq:limrel}
  \lim_{y \to 0+} \gamma_\eta(yi) = K_{\phi^\vee}^*\eta.
\end{equation}
To this end, we consider the $L^2$-realization and observe that, for 
$z \in \cS_\pi$:  
\begin{align*} \Phi^{-1}(\gamma_\eta(z))(\lambda)
&= \int_E \phi(-x) \Phi^{-1}(K^*_{e^{\oline z h}x}e^{z\Lambda}\eta)(\lambda)\, d\mu_E(x) \\
&\!\!\! \overset{\eqref{eq:Phirel}}{=} 
\int_E \phi(-x) e_{-i e^{z h} x}(\lambda)e^{z\Lambda}\eta\, d\mu_E(x) =
\int_E \phi(x) e_{i e^{z h} x}(\lambda)\, d\mu_E(x)\cdot e^{z\Lambda}\eta\\
&= \tilde\phi(\lambda \circ e^{zh})e^{z\Lambda}\eta 
= \tilde \phi_{-iz}(\lambda)
e^{z\Lambda}\eta.
\end{align*}
As $\Phi^{-1}(K_{\phi^\vee}^*\eta) = \tilde\phi\eta$ 
by \eqref{eq:4.5}, the assertion now follows 
from Lemma~\ref{lem:esti}(iii). 
\end{prf}

We are now ready to prove Theorem~\ref{thm:2.1}. 

\begin{prf} {\bf(of Theorem~\ref{thm:2.1})}
We first observe that 
$\sV_0:= \sH_{\sV_\cK^\sharp}(W)$ is a closed real subspace of $\cH = L^2(E^*,\mu;\cK)$. 
It is invariant under the unitary one-parameter group 
$\Delta^{-it/2\pi} = U(0,e^{t})$ 
because the cone $W$ is invariant under~$e^{\R h}$,  
and $\sV_\cK^\sharp$ is invariant under~$\rho(\R_+^\times)$. 

We claim that $\sV_0$ is a standard subspace. From $\sV_0 \subeq \sV$ 
(Lemma~\ref{lem:2.4}) it follows 
that 
\[ \sV_0 \cap i \sV_0 \subeq \sV \cap i \sV = \{0\}.\] So it remains to show that 
$\sV_0$ is total in $\cH$, i.e., that $\sV_0 + i \sV_0$ is dense in $\cH$. 
Let $\cV_0 := \oline{\sV_0 + i \sV_0}$ and observe that this subspace 
is also invariant under the operators $U(0,e^{t}) = \Delta^{-it/2\pi}$, $t \in \R$. 
This implies that, 
for $\xi \in \sV_0$, the range of the extended orbit map 
$\alpha^{\xi} \:  \oline{\cS_\pi} \to \cH$ 
is also contained in $\cV_0$. 

For $z = \pi i/2$ and $f= \tilde\phi\eta$, 
$\phi \in C^\infty_c(W,\R), \eta \in \sV_\cK^\sharp$, 
we first recall from \eqref{eq:4.5} that 
\[ \Phi(\tilde\phi \eta) = 
K_{\phi^\vee}^* \eta = \gamma_\eta(0).\] 
We therefore obtain for $\zeta= e^{\frac{\pi i}{2}h}$ by \eqref{eq:equivb} 
in the proof of Lemma~\ref{lem:2.4}: 
\[ \Phi(\alpha^f(\pi i/2)) 
= \gamma_\eta\Big(\frac{\pi i}{2}\Big) 
= \int_E \phi(-x) K^*_{\zeta^{-1}(x)}e^{\frac{\pi i}{2}\Lambda}\eta\, d\mu_E(x).\] 
For a sequence of test functions 
$\phi_n \in C^\infty_c(W,\R)$ with total integral $1$ whose supports 
converge to $x_0 \in W$, we thus obtain 
\[ \cV_0 \ni \int_E \phi_n (-x) K_{\zeta^{-1}(x)}^*\, d\mu_E(x) \cdot 
e^{\frac{\pi i}{2}\Lambda}\eta 
\to K_{\zeta^{-1}(-x_0)}^* e^{\frac{\pi i}{2}\Lambda}\eta , \] 
and thus $K_{\zeta^{-1}(-x_0)}^* e^{\frac{\pi i}{2}\Lambda} \sV_\cK^\sharp \subeq 
 \cV_0$ for every $x_0 \in W$. 
As $e^{\frac{\pi i}{2}\Lambda} \sV_\cK^\sharp \subeq \cK^{J_\cK}$ 
by \eqref{eq:fixsp} is a dense subspace, we obtain 
$K_{\zeta^{-1}(-x_0)}^* \cK^{J_\cK} \subeq  \cV_0$. 
From 
\[ \zeta^{-1}(-W) 
= E^+ \oplus (-i)(-C_+^0) \oplus i (-C_-^0) 
= E^+ \oplus i(C_+^0 - C_-^0) = E^+ \oplus i (C^0 \cap E^-) \] 
and Lemma~\ref{lem:hjv} it now follows that $\cH^J = \Phi^{-1}(\cH_K^J) \subeq 
\cV_0$, 
and this implies that $\cV_0 = \cH$. 

This shows that $\sV_0$ is a standard subspace contained in the standard 
subspace $\sV$. As it is 
invariant under the modular group $(\Delta_\sV^{it})_{t \in \R}$, Lemma \ref{prop:lo081}
implies that $\sV = \sV_0$. 
\end{prf}

Our approach to the standard subspace $\sV$ also provides 
refined information on the tempered 
distribution $\tilde\mu \in \cS(E,B(\cK))$, namely that its 
restriction to the wedge domain $W$ is actually given by an 
operator-valued function. 

\begin{prop} \mlabel{prop:distonW}
On the open wedge 
$W \subeq E$, the distribution $\tilde\mu \in \cS'(E,B(\cK))$ is 
represented by the functions 
\[ \la \xi, \tilde\mu(x) \eta \ra 
= \la e^{-\frac{\pi i}{2}\Lambda} \xi, 
\tilde\mu(\iota(x)) e^{\frac{\pi i}{2}\Lambda} \eta \ra 
\quad \mbox{ for } \quad 
x \in W, \xi \in \cD(\Delta_\cK^{-1/2}),\eta \in \cD(\Delta_\cK^{1/2}).\]
\end{prop}

\begin{prf} Let $\phi \in C^\infty_c(W,\R)$ and $0 < y < \pi$. Then 
Fubini's Theorem implies that 
\begin{align*}
\int_W \phi(x) \tilde\mu(e^{iyh}x)\, d\mu_E(x) 
&= \int_W \phi(x)  \int_{C^\star} e^{i\lambda(e^{iyh}x)}\, d\mu(\lambda)\, d\mu_E(x) \\
&= \int_{C^\star} \int_W \phi(x)  e^{i\lambda(e^{iyh}x)}\, d\mu_E(x) \, d\mu(\lambda)\\
&= \int_{C^\star} \tilde\phi(\lambda\circ e^{iyh})\, d\mu(\lambda)
= \int_{C^\star} \tilde\phi_y(\lambda)\, d\mu(\lambda).  
\end{align*}
Lemma~\ref{lem:esti}(iv) now implies that, for 
$\xi \in \cK$ and $\eta \in \sV_\cK^\sharp$, we have 
\begin{align*}
&\lim_{y \to 0+} \int_W \phi(x) \la \xi, 
\tilde\mu(e^{iyh}x) e^{iy\Lambda}\eta \ra\, d\mu_E(x) 
= \lim_{y \to 0+} 
 \int_{C^\star} \tilde\phi_y(\lambda)\la \xi, d\mu(\lambda) 
e^{iy\Lambda}\eta \ra \\
&=  \int_{C^\star} \tilde\phi(\lambda)\la \xi, d\mu(\lambda) \eta \ra
=  \int_W \phi(x) \la \xi, \tilde\mu(x) \eta \ra\, d\mu_E(x).
\end{align*}
Therefore the restriction of the distribution 
$\tilde\mu^{\xi,\eta} := \la \xi, \tilde\mu \cdot \eta \ra$ 
to $W$ is represented by the following function 
\[ \tilde\mu^{\xi,\eta}(x) 
= \lim_{y \to 0+} \la \xi, \tilde\mu(e^{iyh}x) e^{iy\Lambda}\eta \ra.\] 
We now evaluate the right hand side using Lemma~\ref{lem:4.8}(b), which 
asserts that, for $t \in \R$ and $w \in \cT$, we have 
\begin{equation}
  \label{eq:tildemurel}
\tilde\mu(e^{-th}w) = e^{t\Lambda^*} \tilde\mu(w) e^{t \Lambda}.
\end{equation}

For $x \in W$, the element 
$\iota(x) = x_0 + i(x_1 - x_{-1})$ is contained in $\cT$ 
and $e^{-zh} \iota(x) \in \cT$ for $|\Im z| < \frac{\pi}{2}$ 
(Lemma~\ref{lem:4.2}). Further, for $\eta \in 
 \cD(e^{\pi i \Lambda}) = \cD(e^{\pi i \Lambda_-})$, 
the curve $t \mapsto e^{t\Lambda}\eta$ extends analytically to $\cS_\pi$ 
(\cite[Lemma~A.2.5]{NO18}), 
so that, for $\xi \in \cD(e^{-\pi i \Lambda})$, the function 
\[ t \mapsto
\la \xi, \tilde\mu(e^{-th}\iota(x)) \eta \ra  
= \la \xi,  e^{t\Lambda^*} \tilde\mu(\iota(x)) e^{t \Lambda} \eta \ra 
= \la e^{t\Lambda} \xi,  \tilde\mu(\iota(x)) e^{t \Lambda} \eta \ra \] 
extends analytically to the function 
\[ \Big\{ z \in \C \: |\Im z| < \frac{\pi}{2}\Big\}  \to \C, \qquad z \mapsto 
\la \xi, \tilde\mu(e^{-zh}\iota(x)) \eta \ra  
= \la e^{\oline z\Lambda} \xi,  
\tilde\mu(\iota(x)) e^{z \Lambda} \eta \ra. \]
From \eqref{eq:tildemurel}, we thus obtain for 
$0 < y < \frac{\pi}{2}$ 
\[  \la \xi, \tilde\mu\big(e^{(y-\frac{\pi}{2})i h}\iota(x)\big) \eta \ra
= \la e^{(y-\frac{\pi}{2})i\Lambda} \xi,  
\tilde\mu(\iota(x)) e^{(\frac{\pi}{2}-y)i\Lambda} \eta \ra.\] 
Next we note that 
\[ \tilde\mu(e^{iyh}x)
= \tilde\mu(e^{i(y-\frac{\pi}{2})h}e^{i\frac{\pi}{2}h}x)
= \tilde\mu(e^{i(y-\frac{\pi}{2})h} \iota(x)),\]
so that 
\begin{align*}
\la \xi, \tilde\mu(x) \eta \ra 
&= \lim_{y \to 0+} \la \xi, \tilde\mu(e^{iyh}x) e^{iy\Lambda}\eta \ra
= \lim_{y \to 0+} 
\la \xi, \tilde\mu(e^{i(y-\frac{\pi}{2})h} \iota(x)) e^{iy\Lambda}\eta \ra \\ 
&= \lim_{y \to 0+} 
\la e^{(y-\frac{\pi}{2})i\Lambda} \xi, 
\tilde\mu(\iota(x)) e^{(\frac{\pi}{2}-y)i\Lambda} e^{iy\Lambda}\eta \ra 
= \lim_{y \to 0+} 
\la e^{(y-\frac{\pi}{2})i\Lambda} \xi, 
\tilde\mu(\iota(x)) e^{\frac{\pi i}{2}\Lambda} \eta \ra \\ 
&= \la e^{-\frac{\pi i}{2}\Lambda} \xi, 
\tilde\mu(\iota(x)) e^{\frac{\pi i}{2}\Lambda} \eta \ra.
\end{align*}
This completes the proof.
\end{prf}

\subsubsection*{The Riesz measures on the half-line as examples}
\mlabel{ex:Riesz} 

We consider the case $E = \R$, $C = \R_+ = [0,\infty)$, 
$h = \id_E$,  $\tau = -\1$, and  the open
tube domain $\C_+=\cT = \R + i (0,\infty)$. In this case we have $E=E_1$
and $W= C^0$.

The 
antiholomorphic extension $\oline \tau$ of $\tau $ to $\cT$ is given 
by $\oline\tau (z)=-\oline z$. It leaves 
$\cT$ invariant with $\cT^{\oline \tau} = i \R^\times_+$. On  
$\C_+  $ we have the positive definite kernels, 
\begin{equation}
  \label{eq:ks}
 K_s(z,w) 
= \tilde\mu_s(z-\oline w) 
= \cL(\mu_s)\Big(\frac{z - \oline w}{i}\Big) 
= \Big(\frac{z - \oline w}{i}\Big)^{-s}, 
\quad s > 0,
\end{equation}
where 
\[ d\mu_s(\lambda) = \Gamma(s)^{-1} \lambda^{s -1}\, d\lambda 
\quad \mbox{ on } \quad (0,\infty) \subeq C^\star\]
see \cite[Lem. 2.13]{NO14}. 

From $h = \id_E$ we derive for the action on the dual space 
$e^{th}.\lambda = e^{-t}\lambda$, that 
\begin{equation}
  \label{eq:scaltrans}
(e^{th})_* \mu_s  = e^{s t} \mu_s \quad \mbox{ for } \quad t \in \R.
\end{equation}
Comparing with \eqref{eq:mug}, we 
may therefore consider this situation as arising from 
$\cK = \C$, $J_\cK(z) = \oline z$, and $\rho(e^t) = e^{st/2}$. 
This corresponds to $\Lambda = \Lambda_+ = \frac{s}{2}$. In particular
we have, as in Example~\ref{ex:3.2}:
\[\sV_\cK=\R,\quad \sV^\sharp_\cK 
= e^{-\frac{\pi i s}{4}}\R\quad\text{and}\quad \sV^\flat_\cK = e^{\frac{\pi i s}{4}}\R.\]
Finally $\sV^\sharp = \sV^\flat$ if and only if $s\in 2\Z$ 
(cf.\ Lemma \ref{lem:VsharpVflat}).

For the Fourier--Laplace transform $\tilde\mu_s$ we have 
$\tilde\mu_s(z) = (-iz)^{-s}$ by \eqref{eq:ks}. 
For the boundary values on $\R$, this leads for $\pm x > 0$ to 
\begin{equation}\label{eq:mus}
 \tilde\mu_s(x) 
= e^{-s \log(-ix)} 
= e^{-s (\log |x| \mp \pi i/2)} 
=  e^{\pm s \frac{\pi i}{2}} |x|^{-s}
=  e^{\sgn(x) s \frac{\pi i}{2}} |x|^{-s}.
\end{equation}
The imaginary part is given on $\R^\times$ by   
\[ \tilde\mu_{s,-}(x) 
=  \pm \sin\Big(s \frac{\pi}{2}\Big) |x|^{-s}.\] 
It vanishes if and only if $s \in 2 \Z$. We will
discuss a generalization of this phenomenon for more general
Riesz measures in Section~\ref{sec:6} below. 
For $s = 2k$, $k \in \N$, 
we have on $\R^\times$ the formula 
$\tilde\mu_s(x) = (-1)^k x^{-2k}$,  
which is even and real. 
On $W = (0,\infty)$, we have by \eqref{eq:mus}  
\begin{equation}
  \label{eq:laponedim}
\tilde\mu_s(x) 
= e^{ s \frac{\pi i}{2}} x^{-s} 
= e^{ s \frac{\pi i}{2}}\cL(\mu_s)(x), 
\end{equation}
where $\mu_s$ is considered as a measure on $[0,\infty)$.
In Proposition~\ref{prop:analyticfunc}, we shall see a generalization of this 
relation to our general context.

We denote by $\cH_s = \cH_{K_s}$ the corresponding reproducing kernel Hilbert space. 
The conjugation on $\cH_s$ is given by 
\[ (Jf)(z) = \oline{f(-\oline z)} \] 
as in \eqref{e17}, so that 
\[\cH_s^J =\{f\in\cH_s\: (\forall z\in \C_+)\,\, \oline{f(-\oline z)} 
= f(z) \}\]
is the subspace of functions which are real-valued on~$i \R^\times_+$. We also have
\begin{equation}\label{eq:Riesz}
\sV 
=e^{-\frac{\pi i s}{4}}\overline{\{ \tilde \phi \: \phi \in 
C_c^\infty ((0,\infty), \R )\}}   \quad\text{and}\quad 
\sV^\prime = e^{-\frac{\pi i s}{4}}\overline{
\{ \tilde\phi \: \phi \in C_c^\infty  ((-\infty,0), \R) \}}.
\end{equation}

\section{Reflection positive representations} 
\mlabel{sec:4}

 Let us briefly recall the concept of a 
\textit{reflection positive representation}, see
 \cite{JOl98,JOl00,NO14,NO18,S86} for details.
 
A reflection positive Hilbert space is a triple $(\cE,\cE_+,\theta)$, where
$\cE$ is a Hilbert space, $\cE_+$ is a closed subspace and $\theta :\cE \to
\cE$ is a unitary involution such that 
\[ \|u\|_\theta^2 := \ip{\theta u}{u}\ge 0 \quad \mbox{ for all }\quad 
u\in \cE_+.\] 
Then $\cN= \{u\in \cE_+\: \|u\|_\theta=0\}$ is a closed subspace of $\cE_+$, 
and we write $\widehat{\cE} $ for the Hilbert space completion of the quotient 
$\cE_+/\cN$ with respect to $\|\cdot\|_\theta$. 

A {\it symmetric semigroup is a triple} 
$(G,S,\tau)$, consisting of a Lie group $G$, 
an involutive automorphism $\tau$ of~$G$, and a subsemigroup $S \subeq G$ 
invariant under the map $g \mapsto g^\sharp := \tau(g)^{-1}$. 
A {\it reflection positive representation} of 
$(G,S,\tau)$ on the reflection positive Hilbert space 
$(\cE,\cE_+,\theta)$ is a unitary representation $(U,\cH)$ of $G$ on $\cE$, 
such that $\theta U(g)\theta = U(\tau(g))$ for all $g\in G$ 
and $U(S)\cE_+ \subeq \cE_+$.  Then 
\begin{equation}\label{eq:E}
\ip{\theta U(g)u}{v} = \ip{\theta u}{U(g^\sharp)v}
\quad \mbox{ for }\quad g \in G, v \in \cE, \end{equation}
and \eqref{eq:E} leads to a $*$-representation of the involutive semigroup 
$(S,\sharp)$ by contractions on~$\widehat{\cE}$ 
(\cite[Prop.~3.3.3]{NO18}). 
The passage from operators on $\cE_+$ to operators on $\hat\cE$ 
is called the {\it Osterwalder--Schrader transform}. 
 
In the articles cited above, $\cE$ and $\cE_+$ are complex Hilbert spaces 
and  $\theta$ is complex linear. In the context of standard subspaces, 
we encounter reflection positivity in the context of 
real Hilbert spaces. 
Any standard subspace $\sV$ of a complex Hilbert space $\cH$ 
satisfies by Lemma \ref{lem:incl} the reflection positivity condition 
$\la \xi, J \xi \ra \geq 0$ for $\xi \in \sV$, so that 
we obtain the real reflection positive Hilbert space 
\[ (\cE,\cE_+,\theta) = (\cH^\R, \sV, J).\] 

In this section we take a closer look at the 
reflection positivity of 
the unitary representation $U$ of the translation group $(E,+)$ on 
$\cH = L^2(E^*,\mu;\cK)$, the additive subsemigroup 
$S = (W,+)$ which is invariant under $x \mapsto x^\sharp = - \tau(x)$, 
and the standard subspace $\sV=\sH_{\sV_\cK^\sharp}(W)$. 
Here $\hat\cE  \cong \cH^J$ by Proposition~\ref{prop:12.1}, and 
we connect our description of $\sV$ with the Osterwalder--Schrader transform 
from the isometric representation of $(W,+)$ on $\sV$ 
to the $*$-representation of $(W,\sharp)$ on $\cH^J$.

\subsection*{Reflection positivity and the wedge semigroup $W$} 

\mlabel{subsec:4.1}

To see that the unitary representation $U$ of the vector group $(E,+)$ on 
$\cH = L^2(E^*,\mu;\cK)$ defines a real reflection positive unitary representation 
of the triple $(E,W,\tau)$, we first observe that 
\begin{itemize}
\item $J U(v) J = U(\tau(v))$ for $v \in E$ 
(Lemma~\ref{lem:UnitRep}), 
\item  $W$ is obviously invariant under the involution 
$x^\sharp = -\tau(x) = x_1 + x_{-1} - x_0$,  
hence inherits the structure of an open involutive semigroup $(W,\sharp)$, 
\item the subsemigroup $W\subeq E$ satisfies 
$U(W)\sV \subeq \sV$. 
\end{itemize}
The last item easily follows from Theorem~\ref{thm:2.1}, 
where we have seen that 
\[ \sV = \oline{\spann_\R 
\{\tilde\phi \eta \: \phi \in C^\infty_c(W,\R), \eta \in \sV_\cK^\sharp\}}.\] 
Now 
\begin{equation}
  \label{eq:star}
 (U(x,1)\tilde\phi \eta)(\lambda)
= e^{i\lambda(x)} \tilde\phi(\lambda)
= (\phi(\cdot - x))\,\tilde{}(\lambda) \eta
\end{equation}
for $x \in E$ shows that $\sV$ is invariant 
under~$U(\oline W)$.\begin{footnote}{Alternatively, this can 
also be derived from the Borchers--Wiesbrock Theorem; 
see \cite{NO17} and \cite[Thm.~4.1]{Ne19a}.}  
\end{footnote}
We thus obtain a reflection positive representation 
of $(E, W, \tau)$ on the real reflection positive 
Hilbert space~$(\cH^\R, \sV,J)$. 
Now the Osterwalder--Schrader transform of the representation 
of $W$ on $\sV$ by isometries is a $*$-representation  
by contractions on the real Hilbert space $\hat \cE \cong \cH^J$ 
 (Proposition~\ref{prop:12.1}; \cite[Prop.~3.3.3]{NO18}). 
From \cite[Rem.~4.13, Prop.~3.6]{Ne19a} we know already that 
the so obtained representation of $(W,\sharp)$ on $\cH^J$ 
coincides with the representation 
$x \mapsto U(\iota(x))$, obtained from the embedding 
\begin{equation}
  \label{eq:iota}
 \iota \: W \into \cT = E + i C^0, \quad x  = x_1 + x_0 + x_{-1} 
\mapsto  i x_1 + x_0 - i x_{-1}
\end{equation}
and the holomorphic extension of $U$ to $\cT$ by 
\[ U(x + iy) := U(x) e^{i \partial U(y)} \quad \mbox{ for } \quad  
x\in E, y \in C^0, \quad 
\partial U(y) = \frac{d}{dt}\Big|_{t = 0} U(ty).\]

\begin{rem} From \cite{NO14} we recall some of the
background concerning bounded representations of the involutive semigroup 
$(W,\sharp)$. The involution $\sharp$ defines on 
the convolution algebra $C^\infty_c(W)$ the structure of a $*$-algebra 
by 
\begin{equation}
  \label{eq:phisharp}
\phi^\sharp(x) := \oline{\phi(x^\sharp)} = (\tau_*\phi^*)(x).
\end{equation}
The corresponding Fourier--Laplace transform is 
\begin{equation}
  \label{eq:fltrafor}
\cF\cL(\phi)(\lambda) 
= \int_W e^{-\lambda^-(x^-) + i \lambda^+(x^+)}\phi(x)\, dx 
\quad \mbox{ for } \quad \phi \in C^\infty_c(W).
\end{equation}
It defines a morphism of complex $*$-algebras 
$C^\infty_c(W) \to C_0(\hat W)$ (Gelfand-Transform), where 
\begin{equation}
  \label{eq:hatw}
\hat W \cong C_+^\star \times E_0^* \times  C_-^\star
\end{equation}
is the cone of hermitian bounded characters of $W$ and 
$C_0(\hat W)$ carries the canonical structure of a $C^*$-algebra. 
\end{rem}

For the unitary representation 
$(U,L^2(E^*,\mu;\cK))$ and the positive definite distribution 
$D = \hat\mu \in C^{-\infty}(E;B(\cK))$, we therefore expect that 
the real subspace $\cH^J$ can be identified with an 
$L^2$-space on the character cone $\hat W$ 
(see the Generalized Bochner Theorem, \cite[Thm.~4.11]{NO14}).
To see which measure on the cone $\hat W$ occurs here, we first 
observe that the projection of $\mu$ to $\hat W$ is closely related to 
$\tilde\mu$:

\begin{lem} \mlabel{lem:flaprel}
The holomorphic function 
$\tilde\mu(z) = \int_{C^\star} e^{i\lambda(z)}\, d\mu(z)$ 
on $\cT$ defines by composition with $\iota \: W \into \cT, 
\iota(x) = x_0 + i x_1 - i x_{-1}$,  
a positive definite analytic function on $(W,\sharp)$ which 
can be expressed as a Fourier--Laplace transform by 
\begin{equation}
  \label{eq:osrel}
\tilde\mu \circ \iota = \cF\cL(p_*\mu),
\end{equation}
where 
\[ p \: C^\star \to \hat W, \quad 
p(\lambda_1 + \lambda_0 + \lambda_{-1}) 
= \lambda_1 + \lambda_0 - \lambda_{-1} \] 
is a slightly modified restriction map. 
\end{lem}

\begin{prf} This follows directly from the definition of $\tilde\mu$ in 
Lemma~\ref{lem:4.8}: 
\[
\tilde\mu(\iota(x_1 + x_0 + x_{-1})) 
= \tilde\mu(i x_1 + x_0 - i x_{-1}) 
= \int_{C^\star} e^{-\lambda(x_1) + i \lambda(x_0) + \lambda(x_{-1})}\, d\mu(\lambda)
\quad \mbox{ for }\quad x \in W. \qedhere\] 
\end{prf}
For $\phi \in C^\infty_c(W,\R)$ and $\eta \in \sV_\cK^\sharp$, we observe that 
\begin{equation}
  \label{eq:jact}
  J(\tilde\phi\eta) 
= \oline{((-\tau)_*\tilde\phi)} J_\cK \eta 
= (\tau_*\tilde\phi) J_\cK \eta 
= (\tau_*\phi)\,\tilde{} \cdot  J_\cK \eta.
\end{equation}
Accordingly, we have for the positive definite distribution $D = \hat\mu$ 
the relation: 
\begin{equation}
  \label{eq:jact2}
  J(\phi * D \eta) 
= (\tau_*\phi) * D  J_\cK \eta.
\end{equation}
This leads to 
\begin{align*}
\la \phi * D\eta, J(\psi * D \xi) \ra 
&\ {\buildrel{\eqref{eq:jact2}} \over =}\ 
 \la \phi * D\eta, (\tau_*\psi) * D J_\cK \xi \ra 
\ {\buildrel {\eqref{eq:pos-def-hilb}} \over =}\ 
 \la \eta, D((\tau_*\psi)^* * \phi)J_\cK \xi \ra \notag \\
&\ {\buildrel \eqref{eq:phisharp} \over = }\ 
\la \eta, D(\psi^\sharp* \phi)J_\cK \xi \ra,
\end{align*}
so that 
\begin{equation}  \label{eq:rprel1}
\la \phi * D\eta, J(\psi * D \xi) \ra 
=\la \eta, D(\psi^\sharp* \phi)J_\cK \xi \ra 
\quad \mbox{ for } \quad 
\phi, \psi \in C^\infty_c(W,\R), \xi, \eta \in \sV_\cK^\sharp.
\end{equation}
Equation \eqref{eq:rprel1} expresses a reflection positivity condition 
for the distribution $D = \hat\mu\in C^{-\infty}(E;B(\cK))$ with respect 
to the involutive semigroup $(W,\sharp)$. 

We also know from Proposition~\ref{prop:distonW} that 
\[ \la \xi, \tilde\mu(x) \eta \ra 
= \la e^{-\frac{\pi i}{2}\Lambda} \xi, 
\tilde\mu(\iota(x)) e^{\frac{\pi i}{2}\Lambda} \eta \ra 
\quad \mbox{ for } \quad 
\xi \in \sV_\cK^\flat, \eta \in \sV_\cK^\sharp.\] 
For $\xi \in \sV_\cK^\sharp$ we have 
$J_\cK \xi \in \sV_\cK^\flat$, so that this leads to 
\begin{equation}
  \label{eq:4.1.x}
\la J_\cK \xi, \tilde\mu(x) \eta \ra 
= \la e^{-\frac{\pi i}{2}\Lambda} J_\cK\xi, 
\tilde\mu(\iota(x)) e^{\frac{\pi i}{2}\Lambda} \eta \ra.
\end{equation}
Since $T_\cK^\sharp = e^{-\pi i \Lambda} J_\cK$ by \eqref{eq:jflatrel}, 
\eqref{eq:taujrel2} implies 
\[ J_\cK \xi = e^{\pi i \Lambda} \xi \quad \mbox{ for } 
\quad \xi \in \sV_{\cK}^\sharp.\] 
Using 
\[ \tilde\mu(\iota(x)) = \cF\cL(p_*\mu)(x) \quad \mbox{ for }\quad 
x \in W\] 
from Lemma~\ref{lem:flaprel}, 
we obtain with \eqref{eq:4.1.x}:

\begin{prop} \mlabel{prop:analyticfunc}
On the open cone $W \subeq E$, 
the distribution $\tilde\mu$ is given by a density, an 
analytic operator-valued function determined by 
\[  \la J_\cK \xi, \tilde\mu(x) \eta \ra 
= \la e^{\frac{\pi i}{2}\Lambda} \xi, 
\tilde\mu(\iota(x)) e^{\frac{\pi i}{2}\Lambda} \eta \ra
= \la e^{\frac{\pi i}{2} \Lambda} \xi,
\cF\cL(p_*\mu)(x)\, e^{\frac{\pi i}{2} \Lambda}\eta\ra
\quad \mbox{ for }\quad \xi,\eta \in \sV_\cK^\sharp, x \in W.\]
\end{prop}

For $\phi, \psi \in C^\infty_c(W,\R)$ and $\xi, \eta \in \sV_\cK^\sharp$, 
the preceding proposition 
leads in particular to 
\begin{align} \label{eq:b}
\la \tilde\psi \xi, J(\tilde\phi \eta) \ra 
&{\buildrel {\eqref{eq:jact}} \over =}\   
\la \tilde\psi \xi, (\tau_*\phi)\,\tilde{} J_\cK \eta \ra 
= \la \xi, \oline{\tilde\psi} (\tau_*\phi)\,\tilde{} J_\cK \eta \ra 
= \la \xi, \tilde{\psi^*} (\tau_*\phi)\,\tilde{}\  J_\cK \eta \ra \notag\\
&= \la \xi, (\psi^* * \tau_*\phi)\,\tilde{}\ J_\cK \eta \ra 
{\buildrel{!}\over=}\ 
 \la \xi, \tilde\mu( \psi^* * (\tau_*\phi))\ J_\cK \eta \ra \notag \\
&{\buildrel {\eqref{eq:jact2}} \over =}  
\la \xi, J_\cK \tilde\mu( \psi^\sharp *\phi) \eta \ra
= \la J_\cK \xi, \tilde\mu( \psi^\sharp *\phi) \eta \ra \notag\\
&= \int_W (\psi^\sharp *\phi)(x) 
\la e^{\frac{\pi i}{2} \Lambda} \xi,
\cF\cL(p_*\mu)(x)\, e^{\frac{\pi i}{2} \Lambda}\eta\ra.
\end{align}
As $\sV_\cK^\sharp = e^{-\frac{\pi i}{2}\Lambda_+} \sV_\cK$, we have 
\[   e^{\frac{\pi i}{2} \Lambda}\sV_\cK^\sharp
= e^{\frac{\pi i}{2}\Lambda_-} \sV_\cK
= \Delta_\cK^{1/4} \sV_\cK \subeq \cK^{J_\cK}.\] 
Therefore the above formula contains the main information 
of the Osterwalder--Schrader transform that passes from 
the $J$-twisted scalar product on $\sV$ to the real 
scalar product on $\cH^J$.

\section{Support properties of the imaginary part of $\hat{\mu}$} 
\mlabel{sec:5} 

For $D = \hat\mu \in \cS'(E;B(\cK))$, we  consider the Hilbert space 
$\cH =\cH_D \subeq \cS'(E;\cK)$ (Definition~\ref{def:2.3}). 
Then, as before, for an open subset $\cO \subeq E$ and a closed 
real subspace $\sK\subeq \cK$, we
define the  closed real subspaces 
\begin{equation}\label{eq:HD}
 \sH_\sK (\cO) := \oline{\spann_\R\{ \phi * D\eta \: \phi \in C^\infty_c(\cO,\R), 
\eta \in \sK\}}.
\end{equation}
In Theorem~\ref{thm:2.1}, 
we have seen that, for the open wedge 
$W \subeq E$, the subspace $\sH := \sH_{\sV_\cK^\sharp}(W)$
is standard. By Corollary~\ref{cor:2.1b} its symplectic complement is 
$\sH' = \sH_{\sV_\cK^\flat}(-W)$. Furthermore, by Lemma \ref{lem:VsharpVflat} we have $\sV_\cK^\sharp=
\sV_\cK^\flat$ if and only if $\Lambda = \Lambda_+$ has spectrum contained 
in~$\Z$.

For real-valued test functions $\phi$, we have 
\[ D(\phi) 
= \int_{E^*} \oline{\tilde\phi}\, d\mu 
= \int_{E^*} \hat\phi\, d\mu \] 
and thus 
\begin{equation}\label{eq:DHerm}
 D(\phi)^* 
= \Big(\int_{E^*} \oline{\tilde\phi(\alpha)}\, d\mu(\alpha)\Big)^* 
= \int_{E^*} \tilde\phi(\alpha)\, d\mu(\alpha) 
= \int_{E^*} \oline{\tilde\phi(-\alpha)}\, d\mu(\alpha) 
= \int_{E^*} \oline{\tilde\phi(\alpha)}\, d\mu(-\alpha).
\end{equation}
Hence the operators $D(\phi)$, $\phi \in C^\infty_c(E,\R)$, 
are all hermitian if and only if 
$\mu$ is invariant under the reflection $r(x) = -x$ in the sense that 
$r_*\mu = \mu$.  As $\supp(\mu) \subeq C^\star$ and 
$C$ is generating, this condition implies $\supp(\mu) = \{0\}$.

In general, the distribution $D$ decomposes as 
\[ D = D_+ + i D_-, \] 
where $D_+$ is the Fourier transform of the measure 
$\shalf(\mu + r_*\mu)$, and 
$D_-$ is the Fourier transform of $\frac{1}{2i}(\mu - r_*\mu)$. 

\begin{lem} The distributions $D_{\pm}$ are hermitian in the sense that
\begin{equation}
  \label{eq:dpmherm}
D_\pm(\phi) \in \Herm(\cK) \quad \mbox{ for } \quad \phi \in \cS(E,\R).
\end{equation}
Furthermore, if  $\phi \in C_c^\infty (E,\R)$, then
\[D(\phi\circ (-\tau))J_\cK =J_\cK D(\phi )^*\quad \text{and} 
\quad D_\pm (\phi \circ (-\tau)  )J_\cK =
J_\cK D_\pm ( \phi).\] 
\end{lem}
\begin{prf} That $D_\pm(\phi )^*=D_\pm (\phi )$ follows from
(\ref{eq:DHerm}) and the relation $r_*\mu_\pm = \pm \mu_\pm$.

For $\phi \in C^\infty_c(E,\R)$, we have 
\[ D (\phi \circ (-\tau)) 
= \int_{E^*} \oline{\tilde\phi}\, d((-\tau)_*\mu) 
= \int_{E^*} \oline{\tilde\phi}\, J_\cK d\mu J_\cK 
= J_\cK \Big(\int_{E^*} \tilde\phi\, d\mu\Big) J_\cK 
\ {\buildrel \eqref{eq:DHerm}\over =}\  J_\cK D(\phi)^* J_\cK\] 
which by the first part implies that
$D_\pm(\phi \circ(-\tau)) = J_\cK D_\pm(\phi) J_\cK$.
\end{prf}

\begin{cor}\mlabel{cor:DpmCom}
If $\mu$ is invariant under the involution $-\tau$ and 
$\phi = \phi^\sharp$, then
the hermitian operators $D_\pm(\phi)$ commute with $J_\cK$. 
\end{cor}
\begin{rem}
By \eqref{eq:mug}, $\mu$ is invariant under $-\tau$ if and only if 
all values of the measure $\mu$ commute with $J_\cK$.
\end{rem}

To explore the support properties of $D_-$, we observe that, 
for $\phi, \psi \in C^\infty_c(E,\R)$, by \eqref{eq:dpmherm} 
\begin{equation}
  \label{eq:sympfc}
\omega(\phi * D\xi, \psi * D\xi)  
= \Im \la \xi, D(\psi^* * \phi) \xi \ra 
= \la \xi, D_-(\psi^* * \phi) \xi \ra
= \la \xi, (\psi * D_-)(\phi) \xi \ra. 
\end{equation}

\begin{prop}\mlabel{prop:supp-control}  
Assume that $\sV_\cK^\sharp = \sV_\cK^\flat$. For
$\xi \in \sV_\cK^\sharp $ define $D^{\xi}_-(\phi )=\la \xi ,D_-(\phi)\xi \ra$. Then
$\supp D^\xi_-\subset W^c$. 
\end{prop}

In particular, if $\cK = \C$ with $J_\cK(z)=\overline{z}$
and $\Lambda \in \Z \1$, we have $\supp D_-\subset W^c$. 
As the examples in Section~\ref{sec:6} (with $\cK = \C$) show, 
we do not expect substantial restrictions on the support of the 
distributions $D^\xi_-$ if $\sV_\cK^\sharp \not= \sV_\cK^\flat$, i.e., 
if $\Spec(\Lambda) \not\subeq \Z$.

\begin{prf} By \eqref{eq:sympfc} and the fact 
that $\sH_D(W, \sV_\cK^\sharp)^{\bot_\omega} = \sH_D(-W, \sV_\cK^\sharp)$
under our assumption $\sV^\sharp_\cK = \sV_\cK^\flat$,
 it follows that   the real-valued distribution $D^\xi_-$ 
satisfies 
\[ \supp(\psi * D_-^{\xi}) \subeq W^c := E\setminus W  \quad \mbox{ for } \quad 
\psi \in C^\infty_c(-W).\] 
As $C^\infty_c(-W,\R)$ contains an approximate identity $\psi_n \to \delta_0$, 
we obtain from 
$\supp(\psi_n * D_-^{\xi}) \subeq W^c$ for every $n$, that 
$\supp(D_-^{\xi}) \subeq W^c.$ 
\end{prf}

\begin{ex} \mlabel{ex:lorentz} 
In the Lorentz context $E = \R^{1,d-1}$, where 
$\mu$ is a Lorentz invariant scalar-valued measure ($\cK = \C$), 
the distribution 
$D_-$ is also Lorentz invariant, hence in particular 
invariant under the rotation group $\SO_{d-1}(\R)$. Therefore 
\[ \supp(D_-) \subeq \bigcap_{g \in \SO_{d-1}(\R)} g W^c 
= C \cup  - C,\] 
so that $\supp(D_-)$ is contained in the closed  double light cone.  
We refer to \cite[\S X.7, p.~215]{RS75} for a different derivation of this result 
from concrete information on the nature of the distribution~$\hat\mu$. 
\end{ex}

\section{The Fourier transform of Riesz measures} 
\mlabel{sec:6}

In this section we specialize the setting of Section \ref{sec:3} to 
simple euclidean Jordan algebras. We briefly recall the relevant concepts. 
 A {\it Jordan algebra} 
is a, not necessarily associative, algebra $E$ such that 
the product satisfies 
\[ xy = yx \quad \mbox{ and } \quad x(x^2y)=x^2 (xy) \quad \mbox{  for all } \quad x,y\in E.\]
 We then define 
\[ L(x)y =xy \quad \mbox{ and }  \quad P(x)=2L(x)^2-L(x^2).\]
$P$ is called \textit{the quadratic representation } of $E$. 
We always assume that $E$ has an identity $e$, which means that $L(e)=\id_E$. 

A Jordan algebra $E$ over $\R$ is called 
\textit{euclidean} if there exists an inner product 
$(\cdot,\cdot)$ on
$E$ such that $L(x)$ is symmetric for all $x\in E$.  If $E$ is euclidean, then
the interior $C^0$ of the closed convex cone 
$C=\{x^2\: x\in E\}$ of squares in $E$ is an open symmetric cone. 
It is the connected component of $e$ in the set $E^\times$ of 
invertible elements in $E$, as well as the
set of all $x\in E$ such that $L(x)$ is strictly positive 
(\cite[Thm. III.2.1]{FK94}).

An element $c$ in $E$ is
\textit{idempotent} if $c^2 = c$. 
The idempotent $c$ is \textit{primitive} if it can not
be written as a sum of two 
non-zero idempotents. The idempotents $c_1,\ldots ,c_r$
form a  \textit{Jordan frame} if each $c_j$ is primitive, 
$c_ic_j =0$ if $i\not= j$, 
and $e=c_1+ \cdots +c_r$. Jordan frames always exist and the group 
$\Aut(E)$ of unital automorphisms of $E$ acts transitively on the set 
of Jordan frames (\cite[Cor.~IV.2.7]{FK94}). 
In particular, the number $r$ of elements in a Jordan frame 
is independent of the frame. It is called the {\it rank of $E$}. 

\begin{exs} (a) Minkowski space $E = \R^{1,d-1}$ carries the 
structure of a euclidean Jordan algebra of rank $r = 2$ which is simple 
for $d \not=2$. The product is given by 
\[ (x_0, \bx)(y_0, \by) := (x_0 y_0 + \bx\by, x_0\by + y_0 \bx).\] 
Here $e = (1,0)$ is a unit element and 
\[ c_1 = \frac{1}{2}(1,1,0,\ldots,0), 
\quad c_2 = \frac{1}{2}(1,-1,0,\ldots,0) \]
form a Jordan frame. 

\nin (b) The other simple euclidean Jordan algebras of rank $r$ are 
\[ \Sym_r(\R), \quad \Herm_r(\C), \quad 
\Herm_r(\H) \quad \mbox{ for } \quad r \in \N 
\quad \mbox{ and } \quad \Herm_3(\bO),\] 
where $\bO$ is the alternative algebra of octonions 
(see  \cite{JvNW34}, \cite{FK94} for the classification). 
Here the Jordan product is given by 
\[ x * y := \frac{xy+yx}{2}, \] 
the euclidean form is $(x,y) = \tr(xy)$, 
the identity matrix $e = \1$ is the unit, 
and the diagonal matrices $c_j := E_{jj}$ form a Jordan frame.
\end{exs}

In this section $E$ is a simple euclidean Jordan algebra of rank $r$ whose Pierce subspaces \eqref{eq:pierce} 
are of dimension~$d$. 
For $z\in E_\C$ (the complexified Jordan algebra), 
we define the \textit{Jordan determinant} by  
\[ \Delta (z) = \det \big(L(z)|_{\C[z]}\big),\] 
 where
$\C [z] \subeq E$ is the unital subalgebra 
generated by $z$. 

For 
\begin{equation}
  \label{eq:walachset}
s \in \Big\{0, \cdots, (r-1)\frac{d}{2}\Big\} \cup 
\Big((r-1)\frac{d}{2}, \infty\Big), 
\end{equation} 
we consider the corresponding Riesz measure $\mu_s$ whose 
Fourier (Laplace) transform satisfies 
\[ \tilde\mu_s(z) = \Delta(-iz)^{-s} \quad \mbox{ for } \quad 
z \in E + i C^0\] 
(\cite[Thm.~VII.3.1]{FK94}). 

\begin{rem} The structure group $\Str (E)$ is the group 
of all $g\in \GL (E)$ such that for $x\in E$, we
have $P(gx)=gP(x)g^\top$ where $P(x)=2L(x)^2-L(x^2)$ is the quadratic representation of~$E$. The structure
group contains the automorphism group 
\[ G(C^0) = \{ g \in \GL(E) \: gC^0= C^0\} \] 
of the open cone $C^0$ as a subgroup of index~$2$.

For $g\in  \Str (E)$, and $x,y\in E$, we have 
\[ \Delta(g.x) = \det(g)^{r/n} \Delta(x)\quad \mbox{ and } \quad 
\Delta(P(y)x) =\Delta(y)^2 \Delta(x) 
\quad \mbox{ for }\quad g \in \Str(E), x,y \in E \] 
(\cite[Prop.~III.4.2]{FK94}). It follows that $\mu_s$ and its Fourier transform 
are semi-invariant under the identity component $\Str(E)_0$ 
of the structure group, so that the 
support of real and imaginary part are closed unions of orbits of this group.  
More concretely, $\cL(g_*\mu_s) = g_*\cL(\mu_s)$ and 
\[ (g_*\cL(\mu_s))(z) = 
\cL(\mu_s)(g^{-1}z) = \Delta(g^{-1}z)^{-s} 
= |\det(g)|^{rs/n} \Delta(z)^{-s} 
= |\det(g)|^{rs/n} \cL(\mu_s)(z) \] 
imply that 
\begin{equation}
  \label{eq:mstrafo}
 g_*\mu_s = |\det(g)|^{rs/n}\mu_s.
\end{equation}
\end{rem}

\begin{prop} \mlabel{prop:opensupp} 
The imaginary part of the tempered distribution 
$\tilde\mu_s$ vanishes on the connected component 
$E^\times_j = \{ x \in E^\times \: \ind(x) = j\}$ of the set 
$E^\times$ of invertible elements of $E$ if and only if $s j \in 2 \Z$.   
\end{prop}

\begin{prf} As a distribution on $E$, the Fourier transform $\tilde\mu_s$ is given on 
the open subset $E^\times$ of invertible elements by the limit 
\[ \tilde\mu_s(x) = \lim_{y \to 0, y > 0} \Delta(-ix + y)^{-s}.\] 
For $s \geq 0$ and $x \in \R^\times$ with $\pm x > 0$, we have 
\begin{equation}
  \label{eq:stars1}
\lim_{y \to 0+} (-ix + y)^{-s} 
= e^{-s\log(-ix)} 
= e^{-s(\log|x| \mp \frac{\pi i}{2})} 
= |x|^{-s} e^{\pm  \frac{s\pi i}{2}}.
\end{equation}
As $\Delta(x) = \prod_{j = 1}^r x_j$ is the product of the spectral 
values of $x$, this leads to 
\begin{equation}
  \label{eq:indform}
\tilde\mu_s(x) =  |\Delta(x)|^{-s} e^{\ind(x) \cdot \frac{s\pi i}{2}}
\quad \mbox{ for } \quad x \in E^\times.
\end{equation}
It follows in particular that 
\begin{equation}
  \label{eq:s-index-cond}
\tilde\mu_s(x) \in \R \quad \Leftrightarrow \quad 
s \ind(x) \in 2 \Z.
\end{equation}
This completes the proof. 
\end{prf}

\begin{rem}  (a) (Analogy with Huygen's principle) 
Proposition~\ref{prop:opensupp}  shows that 
the support properties of $\tilde\mu_s$ depend crucially 
on the parity of the rank~$r$. 
If $r$ is even, then there exist invertible elements of index~$0$ and 
\[  \supp(\Im \tilde\mu_s) \cap E^\times_0 = \eset.\] 
If $r$ is odd, then $\ind(x)$ is odd for every invertible element, so that 
there exist parameters $s$ as in \eqref{eq:walachset}, 
for which $\Im(\tilde\mu_s)$ has full support. 

\nin (b) The Riesz measures $\mu_s$ satisfy the differential equation 
\[  \Delta(\partial) \mu_s = \mu_{s-1} \] 
(\cite[Thm.~VII.2.2]{FK94}), so that $s \in \N$ implies 
\[  \Delta(\partial)^s \mu_s = \delta_0,  \] 
i.e., $\mu_s$ is a fundamental solution of the differential operator 
$\Delta(\partial)^s$ of order $rs$. 
This relation also provides information on the Fourier transform: 
\[  \Delta(-ix)^s \tilde\mu_s(x) = 1.  \] 
As $\Delta$ is homogeneous of degree $r$, this can also be written as 
\begin{equation}
  \label{eq:multform}
 (-i)^{rs}  \Delta^s \tilde\mu_s = 1.  
\end{equation}
As a consequence, $\supp(\tilde\mu_s) = E$ for $s \in \N_0$.
\end{rem}

\begin{ex} \mlabel{ex:6.4} 
(a) $r = 2$ (Minkowski space of dimension $n = d + 2$). 
Then the admissible positive values of $s$ are given by 
$s \geq \frac{d}{2} = \frac{n-2}{2}$ and the possible values 
of the index are $2,0,-2$. 

As $\supp(\Im \tilde\mu_s) \cap E^\times_0 = \eset$ 
(Proposition~\ref{prop:opensupp}), the support is always contained 
in the closed double cone $C \cup - C$. 
It is contained in the boundary of the closed double cone if and only if 
$s \in \Z$. 

For $n = 4$, resp., $d = 2$, we have $s \geq 1$. For 
$s = 1$ the distribution $\Im(\tilde\mu_s)$ is supported in the boundary 
of the double cone. 

\nin (b) $\Im(\tilde\mu_s)$ is supported in the complement of the open double 
cone $C^0 \cup - C^0$ if and only if $rs \in 2 \Z$ 
(Proposition~\ref{prop:opensupp}). 

\nin (c)  
We consider the case $s = \frac{d}{2}$ (the minimal positive value). 
Then $\Im(\tilde\mu_s)$ vanishes on $E^\times_j$ if and only if $jd \in 4 \Z$ 
(Proposition~\ref{prop:opensupp}). 
For $d = 1$, this means that $j \in 4 \Z$ which can only happen for 
$r \in 2 \Z$. Here the Hilbert space is the even part of the Fock space, 
carrying the metaplectic representation of the $2$-fold covering 
group $\Mp_{2r}(\R)$ of $\Sp_{2r}(\R)$ (\cite[Sect.~V]{HNO96}). 
\end{ex} 

\begin{rem} {\rm(Locality condition)} 
The closed convex cone $C \subeq E$ defines an order structure on~$E$. 
In terms of this order, the requirement on the distribution 
$\Im \tilde\mu_s$ that corresponds in the case of Minkowski space 
(Example~\ref{ex:6.4}(a)) 
to the locality condition of corresponding quantum fields is 
\begin{equation}
  \label{eq:loc-cond}
\supp(\Im \tilde\mu_s) \subeq C \cup - C.   
\end{equation}
If this condition is satisfied, then $\Delta(-ix)^s$ is real 
on the components $E^\times_j$ for $j \not=\pm r$. 
This is equivalent to $sj \in 2 \Z$ for $j = r-2,r-4,\ldots, 2-r$ 
(Proposition~\ref{prop:opensupp}). For $r \geq 3$ odd, this implies for $j= 1$ that 
$s \in 2\N_0$, and if $r \geq 4$ is even, we obtain for $j = 2$ that $s \in \N$. 
In both cases we obtain with Proposition~\ref{prop:opensupp} that 
$\supp(\Im\tilde\mu_s) \cap E^\times = \eset$. In Proposition~\ref{prop:7.10} 
below we show that 
$\supp(\Im\tilde\mu_s) = E \setminus  E^\times$ in this case. 
In particular, the locality condition \eqref{eq:loc-cond} 
is never satisfied for $r \geq 3$ because there exist non-invertible elements 
$x \not\in \pm C$, i.e., at least one spectral value is positive and another one 
negative. 
\end{rem}

\subsection{The case $E = \R$}

We now specialize the result from the previous section to the case $E=\R$ considered in Subsection~\ref{ex:Riesz}. We have already seen in the 
introduction how this related to the $\U(1)$-current in CFT.
In this case  the Riesz measures are given by 
\[ d\mu_s(x) = \Gamma(s)^{-1} x^{s-1}\, dx \quad \mbox{ on } \quad 
C^\star = [0,\infty) \quad \mbox{ for } \quad s > 0 \] 
and by $\mu_0 := \delta_0$ (Dirac measure in $0$). 
For the Fourier--Laplace transform $\tilde\mu_s$ we have 
$\tilde\mu_s(z) = (-iz)^{-s}$ for $\Im z > 0$. 
For the boundary values on $\R^\times$, we obtain with \eqref{eq:stars1}
\begin{equation}
  \label{eq:1d-impart}
\tilde\mu_s(x) 
=  e^{\pm s \frac{\pi i}{2}} |x|^{-s} 
\quad \mbox{ and } \quad  \Im\tilde\mu_s(x) 
=  \pm \sin\Big(s \frac{\pi}{2}\Big) |x|^{-s}.
\end{equation}
This shows that $\Im\tilde\mu_s$ vanishes on $\R^\times$ 
if and only if $s \in 2 \Z$. 

\begin{lem} \mlabel{lem:6.9} 
For $E = \R$, the following assertions hold: 
  \begin{itemize}
  \item[\rm(i)] $\supp(\Im\tilde\mu_s) = \R$ for $s \not\in 2 \Z$. 
  \item[\rm(ii)] $\supp(\Im\tilde\mu_s) = \R$ and 
$\Re\tilde\mu_s = c_s \delta_0^{(s-1)}$ with $c_s \in \R^\times$ if $s \in \N$ is odd.
  \item[\rm(iii)] $\supp(\Re\tilde\mu_s) = \R$ and 
$\Im\tilde\mu_s = c_s \delta_0^{(s-1)}$ with $c_s \in \R^\times$ if $s \in \N$ is even. 
  \end{itemize}
\end{lem}

\begin{prf} (i) follows from \eqref{eq:1d-impart}. 

\nin (ii),(iii) For $z \in \C_+$ we have 
$\tilde\mu_s(z) = (-iz)^{-s}$. Taking derivatives, we get 
\[ \tilde\mu_s'(z) = is (-iz)^{-s-1} = is \cdot \tilde\mu_{s+1}(z),\] 
so that we obtain for the boundary values in $\cS'(\R)$ 
\begin{equation}
  \label{eq:recform}
 \tilde\mu_{s+1} = \frac{1}{i s} \tilde\mu_s' \quad \mbox{ for } \quad s > 0.
\end{equation}
As $s \in \N$, we have to take a closer look at $\mu_1$. 
We know already that $\supp(\Im\tilde\mu_1) = \R$. We claim that $\Re\tilde\mu_1 = \pi \delta_0$. 
In fact, if the real test function $\phi$ vanishes in $0$, then 
$\tilde\mu_1(z) = i z^{-1}$ for $\Im z> 0$ yields 
\[ \tilde\mu_1(\phi) 
= \lim_{\eps \to 0} i \int_\R \frac{\phi(x)}{x + i \eps}\, dx 
=  i \int_\R \frac{\phi(x)}{x}\, dx \in i \R, \] 
so that $\Re\tilde\mu_1(\phi) = 0$. We conclude that $\Re \tilde\mu_1 = c_1 \delta_0$ 
for some $c_1 \in \R$. To evaluate this constant, we consider 
a test function $\phi$ constant on an interval $[-\delta, \delta]$ for 
$\delta > 0$. Then 
\[ c_1 \phi(0)
= \Re \lim_{\eps \to 0+} i \int_\R \frac{\phi(x)}{x + i \eps}\, dx 
= \Re \lim_{\eps \to 0+} i \int_{-\delta}^\delta \frac{\phi(x)}{x + i \eps}\, dx 
= \Re  \Big(i \phi(0) 
\lim_{\eps \to 0}\int_{-\delta}^\delta \frac{1}{1+ i \eps}\, dx\Big).\] 
This integral is easily evaluated using the holomorphic logarithm 
on $\C \setminus (-\infty,0]$: 
\[ \int_{-\delta}^\delta \frac{1}{1+ i \eps}\, dx 
= \log(\delta + i \eps) - \log(-\delta + i \eps) 
= \log(\delta + i \eps) - i \pi - \log(\delta - i \eps) \] 
which tends to $-\pi i$ for $\eps \to 0$. We thus obtain 
$c_1 = \pi$. This shows that $\Re\tilde\mu_1 = \pi \delta_0$. 

With the recursion formula \eqref{eq:recform}, we obtain
\[ \Re\tilde\mu_s 
= \frac{(-1)^k\pi }{2k(2k-1)\cdots 1}  \delta_0^{(2k)}
= \frac{(-1)^k\pi}{(2k)!} \delta_0^{(2k)}
\quad \mbox{ for } \quad s = 1 + 2k \in 1 + 2 \N_0,\] 
and
\[ \Im\tilde\mu_s 
= \frac{(-1)^{k+1}\pi}{(2k-1)(2k-2)\cdots 2}  \delta_0^{(2k+1)}
= \frac{(-1)^{k+1}\pi}{(2k-1)!}  \delta_0^{(2k+1)}
\quad \mbox{ for } \quad s = 2k, k \in \N.\] 

The remaining assertions now follow immediately from the recursion formula~\eqref{eq:recform}.
\end{prf}

\subsection{The generalization to   $E = \R^r$}
In this subsection we extend the results from the last section to the euclidean space $E=\R^r$. It turns
out that, as for the Huygens principle, there is a fundamental
difference between $r$ even and odd. 
We have already noted above that, for $r = 1$, this case relates 
to the $\U(1)$-current in CFT. For $r = 2$, it relates to 
$2$-dimensional Minkowski case, considered as a Jordan algebra, 
where coordinates refer to a Jordan frame $(c_1, c_2)$, consisting 
of future pointing lightlike vectors. For $r > 2$, it corresponds to 
the diagonal matrices in the Jordan algebras 
$\Herm_r(\K)$.

\begin{prop} \mlabel{prop:7.7} Let $E = \R^r$ with pointwise multiplication and 
$s \in \N$. 
If either $r$ is even, or if $r$ is odd and $s$ is even, then 
\[ \supp(\Im \tilde\mu_s) = E \setminus E^\times.\] 
\end{prop}

\begin{prf} {\bf $r$ even:} 
Proposition~\ref{prop:opensupp} implies that 
$\Im\tilde\mu_s$ vanishes on $E^\times$ because $\ind(x) \in 2 \Z$. 
It remains to show that any element of the form 
$x = (\bx, 0)$ with $\bx$ invertible in $\R^{r-1}$ is contained in 
$\supp(\Im\tilde\mu_s)$. 
We use the relation 
$\tilde\mu_s(\bx,x_r)= \tilde\mu_1^{\R^{r-1}}(\bx) \tilde\mu_1^\R(x_r)$ 
in the sense of distributions. 
If $\phi_1 \in C^\infty_c((\R^{r-1})^\times)$ and $\phi_2 \in C^\infty_c(\R)$, then 
\begin{equation}
  \label{eq:prodform}
\tilde\mu_s(\phi_1 \otimes \phi_2) 
= \underbrace{\tilde\mu_s^{\R^{r-1}}(\phi_1)}_{\textstyle{\in i^s \R}} \tilde\mu_s^\R(\phi_2)
\end{equation}
by \eqref{eq:indform} because $\ind(\bx) \in r -1 + 2 \Z$ is odd. 
If $s$ is even, this leads to 
\[\Im \tilde\mu_s(\phi_1 \otimes \phi_2) 
= \tilde\mu_s^{\R^{r-1}}(\phi_1) \cdot \Im  \tilde\mu_s^\R(\phi_2) 
=  \tilde\mu_s^{\R^{r-1}}(\phi_1) c_s (-1)^{s-1} \phi_2^{(s-1)}(0)\]
(Lemma~\ref{lem:6.9}(iii)), 
so that $(\bx,0) \in \supp(\Im\tilde\mu_s)$. 
If $s$ is odd, \eqref{eq:prodform} leads to 
\[\Im \tilde\mu_s(\phi_1 \otimes \phi_2) 
= \Im \tilde\mu_s^{\R^{r-1}}(\phi_1) \Re \tilde\mu_s^\R(\phi_2) 
= \Im \tilde\mu_s^{\R^{r-1}}(\phi_1) c_s (-1)^{s-1} \phi_2^{(s-1)}(0)\]
(Lemma~\ref{lem:6.9}(ii)), 
so that $(\bx,0) \in \supp(\Im\tilde\mu_s)$. 
Using the invariance of $\tilde\mu_s$ under permutations of the coordinates, 
the assertion follows. 

\nin {\bf $r$ odd:} Then $s$ is even by assumption, so that 
Proposition~\ref{prop:opensupp} implies that 
$\Im\tilde\mu_s$ vanishes on $E^\times$. 
The same argument as in the case where $r$ is even now shows that 
$\supp(\Im\tilde\mu_s) = E \setminus E^\times$. 
\end{prf}

\subsection{The support of the Fourier transform of 
Riesz measures}

Let $(c_1, \ldots, c_r)$ be a Jordan frame in $E$ and 
$E_0 := \spann \{ c_1, \ldots, c_r\}$ denote the corresponding euclidean 
Jordan subalgebra isomorphic to $\R^r$, endowed with the componentwise 
multiplication. 
We write $p \: E^* \to E_0^*$ for the restriction map and observe that this map 
is proper on $C^\star$ with $p(C^\star) = (C_0)^\star \cong (\R_+)^r$, where 
$C_0 := E_0 \cap C$ is the closed positive cone in $E_0$. 
For a Riesz measure $\mu_s$, the measure 
\[ \mu_s^0 := p_*\mu_s \] 
then satisfies 
\[ \cL(\mu_s)(z) = \Delta(z)^{-s} = \Delta_0(z)^{-s} 
= (z_1 \cdots z_r)^{-s} = \cL(\mu_s^0)(z) \quad \mbox{ for }\quad 
z = \sum_j z_j c_j \in E_0 + i C_0^0.\] 

To transfer information on the support of $\Im(\tilde\mu_s)$ from 
$E_0$ to $E$, we need the following lemma.

\begin{lem} \mlabel{lem:project} 
Let $E_0 \subeq E$ be a real subspace, 
let $p \: E^* \to E_0^*$ be the restriction map, 
and $\mu$ be a tempered measure on $E^*$ for which 
$\mu_0 := p_* \mu$ is also tempered. Then 
\[ \supp(\Im\tilde\mu_0) \subeq \supp(\Im\tilde\mu).\] 
\end{lem}

\begin{prf} Pick a vector space complement $E_1 \subeq E$ for $E_0$, 
so that $E = E_0 \oplus E_1$. 
For $x_0 \in \supp(\Im\tilde\mu_0)$ and an open neighborhood $U_0$ of $x_0$,  
there exists real a test function $\phi_0$ on $E_0$ with 
\[ 0 \not = \Im\tilde\mu_0(\phi_0) 
= \Im\int_{E_0^*} \oline{\tilde\phi_0(-\lambda_0)}\, d\mu_0(\lambda_0)
= \Im\int_{E_0^*}  \tilde\phi_0(\lambda_0)\, d\mu_0(\lambda_0)
= \Im\int_{E^*}  \tilde\phi_0(\lambda_0)\, d\mu(\lambda_0, \lambda_1),\] 
where the existence of the integral follows from the temperedness 
of $\mu_0 = p_*\mu$. 

Now let $(\delta_n)_{n \in \N}$ be a  $\delta$-sequence 
in $C^\infty_c(E_1,\R)$, i.e., 
\[ \supp(\delta_n) \to \{0\}, \quad 0 \leq \delta_n, \quad \mbox{ and }  \quad 
\int_{E_1} \delta_n(x)\, dx = 1.\] 
Then $|\tilde\delta_n| \leq 1$ and the sequence 
$\tilde\delta_n$ converges pointwise to $1$. This shows that 
\[ 0 \not 
= \Im\int_E \tilde\phi_0(\lambda_0)\, d\mu(\lambda)
= \lim_{n \to \infty} \Im\int_E \tilde\phi_0(\lambda_0)\tilde{\delta_n(\lambda_1)}\, d\mu(\lambda)
= \lim_{n \to \infty} \tilde\mu(\phi_0 \otimes \delta_n).\] 
This shows that $(x_1,0) \in \supp(\Im\tilde\mu)$.
\end{prf}

\begin{prop}\mlabel{prop:7.10} Let $E$ be a simple euclidean Jordan algebra 
or rank $r$ and $s \in \N$. 
If either $r$ is even, or if $r$ is odd and $s$ is even, then 
\[ \supp(\Im \tilde\mu_s) = E \setminus E^\times.\] 
\end{prop}

\begin{prf} As the measure $\mu_s^0$ is a tensor product of tempered measures, 
it is tempered. Therefore Lemma~\ref{lem:project} implies that 
\begin{equation}
  \label{eq:suppcond2}
\supp(\Im \tilde\mu_s^0) \subeq \supp(\Im \mu_s). 
\end{equation}

Since $\mu_s$ is semi-invariant with respect to 
$\Str(E)_0$, the support of its imaginary part is a closed union 
of orbits of this group. Any such orbit meets the 
Jordan subalgebra $E_0$. Therefore the support of $\Im\tilde\mu_s$ can 
is determined completely by the support of $\Im\tilde{\mu_s^0}$, 
which corresponds to Riesz measures on the associative Jordan algebra 
$E_0 \cong \R^r$. 
\end{prf}

\begin{exs}  For $r= 3$, the possible indices of invertible elements are $\pm 1, \pm 3$. 
Hence $\Im\tilde\mu_s$ vanishes on some $E^\times_j$ if 
and only if vanishes on $E^\times_3$, which is equivalent to 
$s \in \frac{2}{3} \Z$ (Proposition~\ref{prop:opensupp}). 
It vanishes on all of $E^\times$ if and only if $s \in 2 \N_0$. 
In the latter case, 
\[ \supp(\Im\tilde\mu_s) = E\setminus E^\times\] 
by Proposition~\ref{prop:7.10}.
If $s \in \frac{2}{3}\Z \setminus \Z$, we immediately obtain 
\[ \supp(\Im\tilde\mu_s) 
= E\setminus (E^\times_3 \cup E^\times_{-3}) = E\setminus (C^0 \cup - C^0)\] 
because $E^\times_1 \cup E^\times_{-1}$ is dense in this set. 
\end{exs}

\subsection{Jordan wedges}

For $k \in\{0,\ldots, r\}$, we consider the endomorphisms 
\[ h_k := L(c_1 + \cdots + c_k) - L(c_{k+1} + \cdots + c_r)\in \End(E)  \] 
(cf.\ Appendix~\ref{app:b} for the notation). 
For the Riesz measures $\mu_s$, we obtain with Lemma~\ref{lem:trace-hk} 
\[ (e^{th_k})_* \mu_s 
= e^{t \tr(h_k) \frac{rs}{n}} \mu_s
= e^{t s (2k-r)}\mu_s,\] 
which leads to 
\[ \rho(e^{t}) = e^{t\nu_k} \quad \mbox{ with } \quad 
\nu_k = s\Big(k-\frac{r}{2}\Big).\]
Assume that $s \in \N_0$. The factor $\nu_k$ is either integral for each 
$k \in \{0,\ldots, r\}$ (if $r$ or $s$ is even) 
or never (if $r$ and $s$ are odd). If $\nu_k$ is integral, 
then 
\[ \supp(\Im \tilde\mu_s) = E \setminus E^\times\] 
by Proposition~\ref{prop:7.10}. 
We now relate this to the support conditions 
derived from Proposition~\ref{prop:supp-control}. 

\begin{thm} \mlabel{thm:6.11} 
The following assertions are equivalent for the wedge 
domains $W(h_k) \subeq E$: 
  \begin{itemize}
  \item[\rm(i)] $\nu_k \in \Z$.  
  \item[\rm(ii)] $\supp(\Im \tilde\mu_s) \subeq W(h_k)^c$.
  \item[\rm(iii)] $\supp(\Im \tilde\mu_s) \cap E^\times_{2k-r} = \eset$. 
  \end{itemize}
\end{thm}

\begin{prf} ``(i) $\Rarrow$ (ii)'' follows from Proposition~\ref{prop:supp-control}. 

\par\nin (ii) $\Rarrow$ (iii): 
As $W(h_k) \subeq E^\times_{2k-r}$ by  Corollary~\ref{cor:j2}, condition (ii) 
implies that $\Im\tilde\mu_s$ vanishes on $E^\times_{2k-r}$.

\par\nin (iii) $\Rarrow$ (i):  
By Proposition~\ref{prop:opensupp}, (iii) 
implies that $2\nu_k = (2k-r)s \in 2 \Z$, i.e., that $\nu_k \in \Z$ 
\end{prf}

The preceding theorem shows that 
Proposition~\ref{prop:supp-control} does not provide any information 
on non-invertible elements in the support of $\Im(\tilde\mu_s)$. 
In particular, if every $\nu_k$ is integral, it only shows that 
$\Im(\tilde\mu_s) \cap E^\times = \eset$, so that 
Proposition~\ref{prop:supp-control} provides strictly finer information 
if $r > 2$. 

Theorem~\ref{thm:6.11} also shows that, 
if $\sV_\cK^\sharp \not= \sV_\cK^\flat$, i.e., 
if $\nu_k$ is not integral (Lemma~\ref{lem:VsharpVflat}), 
then we do not expect restrictions on the support of 
$\Im(\tilde\mu_s)$.

\appendix

\section{Standard subspaces}
\mlabel{app:a} 

In this appendix we collect some facts about 
standard subspaces $\sV \subeq \cH$. 
In particular we
describe the connection to antiunitary representations of 
the multiplicative group $\R^\times$, 
and the connection to KMS conditions and modular objects. Most of the material 
in this section is standard and well known. 
We refer to \cite{Lo08} for the basic theory of standard subspaces, 
other references 
are \cite{NO17,NO19}. Proofs are sometimes included 
for the sake of clarity of exposition.

\subsection{Standard subspaces and antiunitary 
representations} 
\mlabel{subsec:1.1}

\begin{definition}
A closed real subspace $\sV$ of a complex Hilbert space $\cH$ 
is called {\it standard} if 
\begin{equation}
  \label{eq:stansub}
  \sV \cap i \sV = \{0\}\quad \mbox{ and } \quad \cH = \oline{\sV + i \sV}.
\end{equation} 
\end{definition}
\nin If $\sV \subeq \cH$ is a standard subspace, 
then 
\begin{equation}
  \label{eq:tomitaop}
 T_\sV \: \cD(T_\sV) := \sV + i \sV \to \cH, \quad 
x + i y \mapsto x- iy 
\end{equation}
defines a closed operator with $\sV = \Fix(T_\sV)$. It is called the 
\textit{Tomita operator} of~$\sV$. Its polar decomposition can be written as 
$T_\sV = J_\sV \Delta_\sV^{1/2}$, 
where $J_\sV$ is a {\it conjugation} (an antiunitary 
involution) and $\Delta_\sV$ is a positive selfadjoint operator such that the 
{\it modular relation} 
\begin{equation}
  \label{eq:modrel}
 J_\sV \Delta_\sV J_\sV = \Delta_\sV^{-1} 
\end{equation}
holds. We call  $(\Delta_\sV, J_\sV)$ the {\it pair of modular 
objects} associated to~$\sV$. 

Denote the inner product on $\cH$ by $\ip{\cdot }{\cdot}$ and
let $\omega (u,v)=\Im \ip{u}{v}$. Then $\omega$ is a symplectic form on $\cH$.
For a real subspace $\sW\subset \cH$ let 
\[\sW^\prime =\{v\in \cH\: (\forall w\in \sW)\, \omega (w,v)=0\}.\]
Then $\sW^\prime $ is also a real subspace and $\sW^{\prime\prime} = \overline{\sW}$, the
closure of $\sW$.

In the following lemma we collect several properties of standard subspaces 
that will be used in this article: 
\begin{lemma}\mlabel{lem:incl}
Let $\sV,\sV_1,\sV_2$ be standard subspaces. Then the following assertions 
hold:
\begin{itemize}
\item[\rm (i)]  $\sV_1\subseteq \sV_2$ implies $\sV_2^\prime \subseteq \sV_1^\prime$. 
\item[\rm (ii)] $\ip{\xi}{J_\sV\xi}\ge 0$ for all $\xi \in\sV$.
\item[\rm (iii)] $\sV$ is standard if and only if $\sV^\prime$ is
standard. 
\item[\rm (iv)] $J_\sV = J_{\sV^\prime}$ and $\Delta_{\sV^\prime} = \Delta_{\sV}^{-1}$.
\item[\rm (v)] $\sV=\sV^\prime$ if and only if $\Delta_\sV =\1$.
\item[\rm (vi)] $J_\sV \sV=\sV^\prime$. 
\item[\rm (vii)] $(\sV^\prime)^\prime = \sV$.
\end{itemize}
\end{lemma}

\begin{prf} (i) is obvious. 

\nin (ii) Let $\xi\in\sV$. Then $\xi = T_\sV(\xi) = J_\sV \Delta_\sV^{1/2} \xi$ 
implies that $\Delta^{1/2}_\sV\xi = J_\sV \xi$. As $\Delta^{1/2}_\sV$ is
positive selfadjoint, it follows that $\ip{\xi}{J_\sV\xi}\ge 0$.

\nin (iii) follows from \cite[\S 3.1]{Lo08} and 
(iv) is contained in \cite[Prop.~3.3]{Lo08}. 

\nin (v) follows from (iii), the fact that the pair $(\Delta_\sV, J_\sV)$ 
determines~$\sV$ and the observation that $\Delta_\sV = \Delta_\sV^{-1}$ 
is equivalent to $\Delta_\sV = \1$. 

\nin (vi) As $ \ip{\xi}{J_\sV\xi}$ is real by (ii), it follows that 
$J_\sV\sV \subseteq \sV^\prime$. 
Applying this argument to $\sV'$ and using (iii), we also obtain 
$J_\sV \sV' \subeq \sV$, so that (v) follows from the fact that $J_\sV$ is 
an involution. 

\nin (vii) follows from (iv) which entails 
$J_{\sV''} = J_\sV$ and $\Delta_{\sV''} = \Delta_\sV$. 
\end{prf}

We have already seen that every standard subspace 
$\sV$ determines a pair $(\Delta_\sV, J_\sV)$ of modular objects 
and that $\sV$ can be recovered from this pair by $\sV = \Fix(J_\sV \Delta_\sV^{1/2})$. 
This observation can be used to obtain a representation theoretic 
parametrization of the set of standard subspaces of $\cH$ 
(cf.~\cite{BGL02}, \cite{NO17}):
Each standard subspace $\sV$ specifies a homomorphism 
$U^\sV : \R^\times \to \AU(\cH)$ by
\begin{equation}
  \label{eq:uv-rep}
U^\sV(e^t) := \Delta_\sV^{-it/2\pi} =e^{it H_\sV}, \quad 
U^\sV(-1) := J_\sV, \quad\text{where } H_\sV =-\frac{1}{2\pi} \log \Delta_\sV .
\end{equation}

\begin{theorem}\mlabel{thm:bij}
The map $\sV\mapsto U^\sV$ defines a bijection between standard subspaces  and
antiunitary representations of the graded group $(\R^\times, \eps_{\R^\times})$. 
The inverse is given by assigning to the antiunitary representation $U \: \R^\times \to \AU(\cH)$ the operators  
\[ H =- i \frac{d}{dt}\Big|_{t = 0} U(e^t), \quad 
 \Delta := e^{-2\pi H}, \quad \text{and}\quad  J := U(-1).\] 
\end{theorem}

\begin{lemma}\mlabel{lem:UtV} Let $\sV$ be a standard subspace. Then the following 
assertions hold: 
  \begin{itemize}
  \item[\rm(a)] $U^\sV(e^t)\sV=\sV$ for all $t\in \R$.
  \item[\rm(b)] $U^{\sV^\prime}(r)=U^{\sV}(r^{-1})$ for $r \in \R^\times$.
\item[\rm(c)] $\sV\cap \sV^\prime =\cH^{U^\sV}$.
  \end{itemize}
\end{lemma}

\begin{proof} (a) Let $\xi \in\sV$ and $t\in\R$. Then
\[T_\sV(U^\sV(e^t)\xi)= J_\sV\Delta_\sV^{\frac{1}{2}-it/2\pi}\xi 
= \Delta_\sV^{-it/2\pi}(J_\sV\Delta_\sV^{1/2}\xi)=
U^\sV (e^t)T_\sV \xi =  U^{\sV}(e^t)\xi .\]

\nin (b), (c) follow from \cite[Lemma~3.7]{NO17}. 
\end{proof}

\begin{defn}
Let $\sV \subeq \cH$ be a real subspace and $J$ be a conjugation on $\cH$. 
We say that~$\sV$ is {\it $J$-positive} if 
$\la \xi, J \xi \ra \geq 0$ for $\xi \in \sV$. 
\end{defn}

Recall that a conjugation on $\cH$ is an antiunitary involution. 
The following lemma explores the question when the positivity of a conjugation $J$ 
on a real subspace $\sV$ implies that $\sV$ is standard with $J=J_\sV$. 

 \begin{lem} \mlabel{lem:8.2} 
For a closed real subspace $\sV \subeq \cH$ and a 
conjugation $J$, the following assertions hold:
\begin{itemize}
\item[\rm(i)] If $\sV$ is $J$-positive, then $J\sV \subeq \sV^\prime$. 
\item[\rm(ii)] If $\sV + i \sV$ is dense in $\cH$
and $J\sV \subeq \sV^\prime$, then  $\sV\cap i\sV =\{0\}$. 
\item[\rm(iii)] Assume that $\sV$ is standard. Then the following are equivalent
\begin{itemize}
\item[\rm (a)] $J=J_\sV$.
\item[\rm (b)] $\sV^\prime$ is $J$-positive and $J\sV \subseteq \sV^\prime$.
\item[\rm (c)] $\sV$ and $\sV^\prime$ are both $J$-positive.
\end{itemize}
\end{itemize}
\end{lem}

\begin{prf} (i) The form $\beta(\xi,\eta) := \la J\xi, \eta \ra$ on 
$\cH$ is complex bilinear and symmetric. That $\sV$ is $J$-positive 
implies that $\beta$ is real on all pairs $(\xi,\xi)$, $\xi \in \sV$, 
hence by polarization also on $\sV \times \sV$. This means that 
$J\sV \subeq \sV'$. 

\nin (ii)  The subspace $\sV_0 := \sV \cap i \sV$ of $\cH$ is complex 
and satisfies $J\sV_0 \subeq \sV'$. Since $J\sV_0$ is also a complex subspace, 
it follows that $J\sV_0$ is orthogonal to the total subset $\sV$, hence 
trivial. 

\nin (iii) That (a) implies (b),(c) follows from Lemma~\ref{lem:incl}(ii),(iv),(vi). 
Further, (b) implies 
$J\sV' \supeq JJ\sV = \sV$, so that the $J$-positivity of $\sV'$ implies by 
\cite[Prop.~3.9]{Lo08} that $J = J_{\sV'} = J_\sV$, hence~(a). 
If (c) holds, then (i) shows that the $J$-positivity of 
$\sV$ implies $J\sV \subeq \sV'$. Hence (c) implies~(b). 
This proves (iii). 
\end{prf}

\begin{prop} \mlabel{prop:12.1} 
{\rm(Reflection positivity and standard subspaces)} 
Let $\sV \subeq \cH$ be a standard subspace with modular objects 
$(\Delta, J)$. Then the following assertions hold: 
\begin{itemize}
\item[\rm(i)] $(\cE,\cE_+, \theta) := (\cH^\R, \sV,J_\sV)$ is a 
real reflection positive Hilbert space.
\begin{footnote}
{See the introduction to Section~\ref{sec:4} for details.}
\end{footnote}
\item[\rm(ii)] The map $\Delta^{1/4} \: \sV \to \cH^J$ extends to an 
isometric isomorphism $\hat \sV \to \cH^J$, where 
$\hat \sV$ is the completion of $\sV$ with respect to scalar product 
$\la v,w \ra_J := \la v, J w \ra$ for $v,w \in \sV$.
\end{itemize}
\end{prop}

\begin{prf} (i) follows directly from 
$\la v, J v \ra = \la v, \Delta^{1/2} v \ra = \|\Delta^{1/4} v\|^2.$

\nin (ii) Next we note that $\sV \subeq \cD(\Delta^{1/2}) \subeq \cD(\Delta^{1/4})$ 
implies that $\Delta^{1/4}$ is defined on $\sV$. 
For $v \in \sV$, we have 
\[ J \Delta^{1/4}v = \Delta^{-1/4} Jv = \Delta^{-1/4} \Delta^{1/2}v 
=  \Delta^{1/4}v, \] 
so that $\Delta^{1/4}\sV \subeq  \cH^J$. 
Using the spectral decomposition of $\Delta$ it follows 
easily that $\Delta^{1/4}\sV$ is dense in $\cH^J$. This implies (ii). 
\end{prf} 

The following simple observation  is  taken from \cite{MN21}.
It slightly extends \cite[Prop.~3.10]{Lo08}. 

\begin{prop} \mlabel{prop:lo081} 
Suppose that  $\sV_1 \subeq \sV_2$ are standard subspaces of $\cH$. 
If 
\begin{itemize}
\item[\rm(a)] $\Delta_{\sV_2}^{it} \sV_1 = \sV_1$ for $t \in \R$, or 
\item[\rm(b)] $\Delta_{\sV_1}^{it} \sV_2 = \sV_2$ for $t \in \R$, 
\end{itemize}
then $\sV_1 = \sV_2$. 
\end{prop}

\begin{prf} That (a) implies $\sV_1 = \sV_2$ follows from 
\cite[Prop.~3.10]{Lo08}. From (b) we obtain by dualization 
$\sV_2^\prime  \subeq \sV_1^\prime$ with 
$\Delta_{\sV_1^\prime}^{it} \sV_2^\prime = \sV_2'$ for $t \in \R$, so that 
we obtain $\sV_1^\prime = \sV_2^\prime $ with (a), hence $\sV_1 = \sV_2$ also holds in this case. 
\end{prf}

\subsection{Standard subspaces and the KMS condition}\mlabel{ss:KMS}
As mentioned above, the bijection in Theorem \ref{thm:bij} is closely related to the  
 characterization of standard subspaces and their modular objects in terms of a KMS condition 
(\cite{Lo08}, \cite{NO19}). 

\begin{defn} Let $V$ be a real vector space 
and $\Bil(V)$ be the space of real bilinear maps $V \times V \to \C$.
A function $\psi \: \R \to \Bil(V)$ is said to be {\it positive 
definite} if the kernel 
$\psi(t-s)(v,w)$ on the product set $\R \times V$ is positive definite. 

We say that a positive definite 
function $\psi \: \R \to \Bil(V)$ satisfies the {\it KMS condition} 
for $\beta > 0$ if $\psi$ 
extends to a function $\oline{\cS_\beta} \to \Bil(V)$ which is pointwise 
continuous and pointwise holomorphic on the interior $\cS_\beta$, and satisfies 
\begin{equation}
  \label{eq:kms-def}
 \psi(i \beta+t) = \oline{\psi(t)}\quad \mbox{ for } \quad t \in \R.
\end{equation}
\end{defn}

In a similar fashion as Lemma~\ref{lem:8.2}(iv) characterizes 
the conjugation $J_\sV$ of a standard subspace $\sV$ in terms 
of the $J$-positivity of $\sV$ and $\sV'$, the following proposition 
characterizes the corresponding modular group in terms of a KMS condition. 

\begin{prop} \mlabel{prop:9.5}
Let $\sV \subeq \cH$ be a standard subspace and 
$U \: \R \to \U(\cH)$ be a continuous unitary one-parameter group. 
Then $U(t) = \Delta_\sV^{-it/2\pi}$ holds for all $t \in \R$ if and only if the 
positive definite function 
\[ \phi \:  \R \to \Bil(\sV),\quad 
\phi(t)(\xi,\eta) := \la \xi, U(t) \eta \ra \] 
satisfies the KMS condition for $\beta = 2\pi$.
\end{prop}

\begin{prf} (see also \cite[Thm.~2.6]{NO19}) 
In \cite[Prop.~3.7]{Lo08}, this characterization is stated for the function 
$\la U(t)\xi, \eta\ra$, but this should be 
$\la \xi, U(t) \eta\ra$ if the scalar product is linear in the second argument.
\end{prf}

\subsection{Hardy space and graph realizations} 
\mlabel{subsec:hardy}

Let $\Delta >  0$ be a positive selfadjoint operator on $\cH$. 
Then $\cD(\Delta^{1/2})$ is a dense subspace of $\cH$, and the map 
\[ \Psi \: \cD(\Delta^{1/2}) \to \Gamma(\Delta^{1/2}), \quad 
\xi \mapsto (\xi, \Delta^{1/2}\xi) \] 
is a complex linear bijection onto the closed 
graph of the selfadjoint operator $\Delta^{1/2}$ 
in the Hilbert space $\cH \oplus \cH$. 
We thus obtain on $\cD(\Delta^{1/2})$ the structure of a 
complex Hilbert space for which $\Psi$ is unitary. 

The operator $\Delta$ defines a unitary one-parameter group 
$(\Delta^{it})_{t \in \R}$, and we consider the 
$\cH$-valued Hardy space 
\begin{align*}
&  H^2(\cS_\pi,\cH)^{\Delta} \\
&:=\Big\{ f \in \Hol(\cS_\pi,\cH) \: 
(\forall z \in \cS_\pi)(\forall t \in \R) 
\ f(z + t) = \Delta^{-it/2\pi} f(z), \quad 
\sup_{0 < y < \pi} \|f(iy)\| < \infty \Big\}
\end{align*} 
of equivariant bounded holomorphic maps $\cS_\pi \to \cH$. 
For $\Delta^{-it/2\pi} = e^{itH}$, i.e., $H = -\frac{1}{2\pi} \log \Delta$,  
and the spectral measure $P_H$ of $H$, we have 
\[ \|\Delta^{y/2\pi}\xi\|^2 =  \|e^{-yH}\xi\|^2 
= \int_\R e^{-2\lambda y}\, dP_H^\xi(\lambda), \] 
so that the Monotone Convergence Theorem implies 
for $f \in H^2(\cS_\pi,\cH)^{\Delta}$ and $\xi := f(\pi i/2)$ 
\[ \int_\R e^{\pm \lambda \pi}\, dP_H^\xi(\lambda) < \infty, 
\quad \mbox{ so that } \quad \xi \in \cD(\Delta^{\pm 1/4}).\]
 Thus \cite[Lemma~A.2.5]{NO18} 
implies that $f$ extends to a continuous function 
on $\oline{\cS_\pi}$, also denoted~$f$. 
It satisfies 
\[ \sup_{0 < y < \pi} \|f(iy)\| = \max(\|f(0)\|, \|f(\pi i)\|).\] 
In particular, the  map 
\[ \Phi \: H^2(S_\pi, \cH)^{\Delta} \to \cH \oplus \cH, \quad 
\Phi(f) := (f(0), f(\pi i))\] 
is defined. To identify the range of 
$\Phi$, we use \cite[Lemma~A.2.5]{NO18} to see that 
$\xi = f(0)$ for some $f \in H^2(S_\pi, \cH)^{\Delta_V}$ if and only if 
$\xi \in \cD(\Delta^{1/2})$. Then 
$f(\pi i) = \Delta^{1/2} \xi$, and we conclude that 
\[ \Phi\big(H^2(S_\pi, \cH)^{\Delta}\big) = \Gamma(\Delta^{1/2})\] 
(cf.\ \cite[Prop.~3.4]{LLQR18}). 
As $\Phi$ is injective with closed range, it is an isomorphism of 
Banach spaces but not necessarily isometric. 

\cite[Prop.~3.2]{LLQR18} also contains observations which are very similar to the following 
lemma. 

\begin{lem} If $J$ is a conjugation on $\cH$, 
then 
\[ \tilde J(\xi,\eta) := (J\eta, J\xi) \]
defines a conjugation on $\cH \oplus \cH$, 
and $\tilde J$ maps $\Gamma(\Delta^{1/2})$ into itself 
if and only if the modularity relation $J\Delta J = \Delta^{-1}$ holds. 
\end{lem}

\begin{prf} If the modularity relation holds, then 
we also have $\Delta^{-1/2} J= J \Delta^{1/2}$, so that 
$J \cD(\Delta^{1/2}) = \cD(\Delta^{-1/2}) = \cR(\Delta^{1/2})$, 
and therefore 
\[ \tilde J(\xi, \Delta^{1/2}\xi) 
= (J \Delta^{1/2}\xi, J \xi) 
= (\Delta^{-1/2}J \xi, J \xi)  \in \Gamma(\Delta^{1/2}) 
\quad \mbox{ for } \quad \xi \in \cD(\Delta^{1/2}).\] 
If, conversely, $\tilde J$ preserves $\Gamma(\Delta^{1/2})$, 
then 
\[ J \Delta^{1/2} \xi \in \cD(\Delta^{1/2}) 
\quad \mbox{ and } \quad 
\Delta^{1/2}  J \Delta^{1/2} \xi  = J \xi
\quad \mbox{ for } \quad \xi \in \cD(\Delta^{1/2}).\] 
This means that $J\xi \in \cD(\Delta^{-1/2})$ with 
$J \Delta^{1/2} \xi  = \Delta^{-1/2} J \xi$. 
As $J$ is an involution, $J\cD(\Delta^{1/2}) = \cD(\Delta^{-1/2})$ 
and $J \Delta^{1/2} J = \Delta^{-1/2}$. This implies 
$J\Delta J = \Delta^{-1}$. 
\end{prf}

If $J\Delta J  = \Delta^{-1}$, the preceding lemma shows that 
the closed subspace $\Gamma(\Delta^{1/2})$ of $\cH \oplus \cH$ is 
invariant under $\tilde J$. We also observe that the antilinear operator 
\[ T := J \Delta^{1/2} = \Delta^{-1/2} J\:  \cD(\Delta^{1/2}) \to 
\cD(\Delta^{1/2})\]
satisfies  
\[ \tilde J \Psi(\xi) 
= (J\Delta^{1/2}\xi, J \xi) 
= (T\xi, J \xi) = (T \xi , \Delta^{1/2}T \xi)
= \Psi(T\xi) \quad \mbox{ for } \quad 
\xi \in \cD(\Delta^{1/2}). \] 
For the standard subspace $\sV$ with $J_\sV = J$ and $\Delta_\sV = \Delta$, 
$T_\sV := T$ is the corresponding Tomita operator 
(Subsection~\ref{subsec:1.1}), and the relation 
\[  \tilde J \circ \Psi = \Psi \circ T_\sV \] 
implies that $\Psi(\sV) = \Gamma(\Delta^{1/2})^{\tilde J}$. In particular, 
$T_\sV$ is a conjugation for the complex Hilbert space structure on 
$\cD(\Delta^{1/2})\cong \Gamma(\Delta^{1/2})$, whose fixed point space is $\sV$. 

Next we observe that, as $J$ commutes with the unitary operators 
$\Delta^{it}$, $t \in \R$, 
\[  (\hat J f)(z) = J f(\pi i + \oline z)\] 
defines an isometric involution on the Hardy space $H^2(\cS_\pi, \cH)^\Delta$. 
For $f \in H^2(\cS_\pi,\cH)^{\Delta}$ and $\xi := f(0) \in \cD(\Delta^{1/2})$, we have 
\[ \tilde J \Phi(f) = (J f(\pi i), J f(0))  = \Phi(\hat J f),\] 
so that $\Phi$ intertwines the conjugations $\tilde J$ and $\hat J$. 
We conclude in particular that 
\begin{equation}
  \label{eq:psiphi}
\Phi^{-1} \Psi(\sV) = \{  f \in H^2(\cS_\pi,\cH)^{\Delta} \: 
\hat J(f) = f \}.
\end{equation}

\begin{lem}
For $f \in H^2(\cS_\pi,\cH)^{\Delta}$, the following conditions are equivalent: 
\begin{itemize}
\item[\rm(a)] $\hat J(f) = f$, i.e., $f(\pi i + \oline z) = J f(z)$ 
for $z \in \cS_\pi$. 
\item[\rm(b)] $f(z) \in \cH^J$ for $\Im z = \frac{\pi}{2}$. 
\item[\rm(c)] $f(0) \in \sV$. 
\item[\rm(d)] $f(\pi i) \in \sV'$. 
\item[\rm(e)] $f(\pi i) = J f(0)$. 
\end{itemize}
\end{lem}

\begin{prf}
The equivalence of (a) and (b) follows by uniqueness of analytic continuation 
from the line 
$\frac{\pi i}{2} + \R \subeq~\cS_\pi.$ 
The equivalence of (a) and (c) follows from 
$\Psi^{-1}\Phi(f) = f(0)$ and \eqref{eq:psiphi}. 
As $f(\pi i) = \Delta^{1/2} f(0)$ is contained in $\sV' = J\sV$ if and only if 
$J \Delta^{1/2} f(0) \in \sV$, which in turn is equivalent  to 
$f(0) \in \sV$, conditions (c) and (d) are also equivalent. 
The equivalence of (c) and (e) follows from Proposition~\ref{prop:standchar}. 
\end{prf}

The map $\ev_0 = \Psi^{-1}\Phi \: \Fix(\hat J) \to \sV$ is an isometry of 
real Hilbert spaces because $\hat J(f) = f$ implies $\|f(0)\| = \|f(\pi i)\|$. 
In this sense every standard subspace can be realized in a natural 
way as a ``real form'' of a Hardy space on the strip~$\cS_\pi$.

\section{Wedges in euclidean Jordan algebras} 
\mlabel{app:b}

We expect that the reader is familiar with the basic theory of simple euclidean Jordan algebras. We use \cite{FK94}  as a standard reference. 
For the basic definitions we refer to Section~\ref{sec:6}.
From now on $E$ is always a simple euclidean Jordan algebra with 
\begin{equation}
  \label{eq:dimrank}
 \dim(E) = n \quad \mbox{ and }\quad \rank(E) = r,
\end{equation}
and $c_1,\ldots ,c_r$ is a Jordan frame.
 We then obtain the \textit{Pierce decomposition}
\begin{equation}
  \label{eq:pierce}
E = \bigoplus_{j = 1}^r \R c_j \oplus \bigoplus_{i < j} E_{ij} \quad \mbox{with} 
\quad E_{ij} = \big\{ v \in E \:  c_i v = \shalf v, c_j v = \shalf v\big\}
\end{equation}
(\cite[\S IV.1]{FK94}). 
The set $E^\times$ of invertible elements of $E$ 
has $r+1$ connected components that can be described 
as follows. We fix a spectral decomposition 
$x = \sum_{j = 1}^r x_j \tilde c_j$ of an element $x \in E$, 
where $(\tilde c_1,\ldots, \tilde c_r)$ 
is a Jordan frame (\cite[Thm.~III.1.1]{FK94}). 
This means that, under the automorphism group $\Aut(E)$, 
the element $x$ is conjugate to $\sum_{j = 1}^r x_j c_j$. 
For $E = \Herm_r(\K)$, this corresponds to the conjugation 
of a hermitian matrix $x$ by $\U_r(\K)$ to a diagonal matrix, 
and for Minkowski space, it corresponds to conjugation 
of an element $x \in \R^{1,d-1}$ under $\OO_{d-1}(\R)$ to one 
with $x_2 = \cdots = x_d = 0$.

We define (cf.~\cite[p. 29]{FK94}): 
\begin{itemize}
\item the {\it index of $x$} by 
$\ind(x) := \sum_{j = 1}^r \sgn(x_j) \in \{r,r-2,\ldots, -r\}.$ 
\item the {\it determinant of $x$} by 
$\Delta(x) = \prod_{j=1}^r x_j$, and 
\item the {\it trace  of $x$} by 
$\tr(x) = \sum_{j=1}^r x_j$. 
\end{itemize}
Then the connected components of $E^\times$ are the subsets 
\[ E^\times_j := \{ x \in E^\times \: \ind(x) = j\}, \qquad 
j = r,r-2, \ldots, -r.\] 

For a multiplication operator 
\[ h := \sum_{j = 1}^r a_j L(c_j), \] 
the Pierce decomposition \eqref{eq:pierce} shows that the eigenvalues are 
$a_1, \ldots, a_r$ and $\frac{a_i + a_j}{2}$ for $i \not=j$. 
For $h \not=0$, it follows that the eigenvalues of $h$ are contained in $\{-1,0,1\}$ 
if and only if $a_j \in \{\pm 1\}$. 
Reordering the Jordan frame, we see that, up to applying an automorphism of 
$E$, any such element is conjugate to one of the form 
\begin{equation}
  \label{eq:hk}
 h := h_k := L(c_1 + \cdots + c_k) - L(c_{k+1} + \cdots + c_r) 
\quad \mbox{ for some } \quad k \in \{0,\ldots, r\}.
\end{equation}
Then 
\[ E_1(h) = \bigoplus_{j = 1}^k \R c_j \oplus \bigoplus_{i < j \leq k} E_{ij}, 
 \quad 
E_0(h) = \bigoplus_{i \leq k < j} E_{ij} 
\quad \mbox{ and }  \quad 
E_{-1}(h) = \bigoplus_{j = k+1}^r \R c_j \oplus \bigoplus_{k < i < j} E_{ij}, \] 
where $E_{\pm 1}(h)$ are Jordan subalgebras of $E$.  
Note that $E_1(-h) = E_{-1}(h)$. 
 
We now observe that the quadruple 
$(E,C,h,\tau := e^{\pi i h})$ satisfies the assumptions 
(A1-3) in Section~\ref{sec:3}. Here (A1) and (A2)  are obvious. 
To verify (A3), note that 
$e^{\R h} C = C$ follows from $e^{L(x)}C = C$ for every $x \in E$ 
(\cite[p.~48]{FK94}). Moreover, $\tau = e^{\pi i h} \in \Str(E)$ 
(\cite[Prop.~VIII.2.8]{FK94}) satisfies $\tau(e) = - e$, so that 
$\tau(C) = - C$. This proves (A3). 

In this context, the constructions of Section~\ref{sec:3} have 
a Jordan theoretic interpretation. 
The cones $C_+:=  C \cap E_1(h)$ is  the positive cone 
in the Jordan algebras $E_1(h)$ and $C_- = - C \cap E_{-1}(h)$ is 
the negative cone in $E_{-1}(h)$. The corresponding wedge is 
\begin{equation}
  \label{eq:w-def}
W := W(h) := C_+^0 \oplus  C_-^0 \oplus E_0(h).
\end{equation}
Note that $x \in W^c = E\setminus W$ 
if and only if $x_1 \not\in C_+^0$ or $x_{-1} \not\in C_-^0$. 
For the extremal situations $k = 0,r$, 
we obtain $W(h_r) = C^0$ and $W(h_0) = - C^0$. 

\begin{lem}
  \mlabel{lem:trace-hk} 
$\tr(h_k) = (2k-r)\frac{n}{r}$ 
for $n = \dim(E)$ and $r = \rank(E)$.
\end{lem}

\begin{prf}
From the Pierce decomposition \eqref{eq:pierce} it follows that 
\begin{equation}
  \label{eq:nform}
n = \dim E= r + \frac{r(r-1)}{2}d = r\Big(1 + (r-1)\frac{d}{2}\Big) 
\end{equation}
and 
\begin{align*}
  \tr(h_k) 
&= k - (r-k) + d\Big(\frac{k(k-1)}{2} -  \frac{(r-k)(r-k-1)}{2}\Big) \\
&= 2k-r + \frac{d}{2}(k^2 - k +  (r-k) - (r^2 - 2rk + k^2)) \\
&= 2k-r + \frac{d}{2}(r- 2k + 2rk - r^2) \\
&= 2k-r + \frac{d}{2}(2k-r)(r-1) 
= (2k-r)\Big(1 + (r-1)\frac{d}{2}\Big) 
= (2k-r)\frac{n}{r}.\qedhere
\end{align*}
\end{prf}

\begin{rem} We are interested in the parity of the numbers 
$\tr(h_k)$. First, we observe that $\frac{n}{r} \in \shalf \Z$ by 
\eqref{eq:nform}  
and that this is an integer if and only if $(r-1)d$ is even. 
This is equivalent to $d$ even or $r$ odd. 
In this case $\tr(h_k)$ is even if and only if $n$ is even. 
In the other case the parity of $\tr(h_k)$ depends on $k$ 
if $\frac{2n}{r}$ is odd. 

\end{rem}

\begin{lem} \mlabel{lem:j1} Let $p_1 \: E \to E_1(h), x \mapsto x_1$ 
denote the projection map. 
Then $p_1(C) \subeq C$ and $\rk p_1(x) \leq \rk x$ for $x \in C$. 
\end{lem}

\begin{prf} Let $m := \rk x$ and $x \in C$. The subset
$C_{\leq m} := \{ w \in C \: \rk w \leq m\}$ is a closed (non-convex) cone  
invariant under $e^{\R h}$. Therefore 
$x_1 = \lim_{t \to \infty} e^{-t} e^{th} x \in C_{\leq m}.$
\end{prf}

For an element $x \in E$, we write $x = x_+ - x_-$ with $x_\pm \in C$ 
and $x_+ x_- = 0$ for the canonical decomposition of $x$ into positive and negative part 
which can be obtained from the spectral decomposition (\cite[Thm.~III.1.1]{FK94}). 

The following proposition and its corollary constitute 
the main result of this appendix. They are the 
key tool for the finer analysis of the support properties 
of the Fourier transforms $\hat\mu_s$ of the Riesz measures~$\mu_s$.

\begin{prop}
  \mlabel{prop:j2} 
Let $v = v_+ - v_-$ be the canonical decomposition of $v\in E$ into positive and negative part.  
Then $\Aut(E)v \subeq W(h_k)^c$ if and only if 
\[ \rk v_+ < k \quad \mbox{  or  } \quad \rk v_- < r-k.\] 
\end{prop}

\begin{prf} For $v \in W(h_k)$ we have $p_1(v) \in C_+^0$, so that 
$p_1(v_\pm) \in C$ yields 
\[ p_1(v_+) = p_1(v) + p_1(v_-) \in C_+^0 + C_+ \subeq C_+^0 \] 
is invertible in $E_1(h)$. 
Lemma~\ref{lem:j1} thus implies that $\rk v_+ \geq \rk p_1(v_+) = k$. 
Therefore $\rk v_+ < k$ entails $v \in W(h_k)^c$. 
For $g \in \Aut(E)$, we have $\rk(gv_+) = \rk(v_+)$, so that 
$\rk(v_+) < k$ implies $\Aut(E)v \subeq W(h_k)^c$. 
Likewise $\rk v_- < r-k$ implies that $\Aut(E)v \subeq W(h_k)^c$. 

Suppose, conversely, that $\Aut(E)v \subeq W(h_k)^c$. 
Then there exists a $g \in \Aut(E)$ with 
\[ gv = \sum_{j = 1}^r \nu_j c_j 
\in \sum_{j = 1}^r \R c_j \subeq E_1(h) \oplus E_{-1}(h)
\quad \mbox{ and } \quad 
\nu_1 \geq \cdots \geq \nu_r.\] 
As $gv \not\in W(h_k)$, we have 
$(gv)_1 = \sum_{j \leq k} \nu_j c_j \not \in C_+^0$ or 
$(gv)_{-1} = \sum_{j > k} \nu_j c_j \not \in C_-^0$. 
In the first case $\nu_k \leq 0$, so that $\rk v_+ < k$, 
and in the second case 
$\nu_{k+1} \geq 0$, so that $\rk v_- < r-k$.   
\end{prf}

By negation we immediately obtain: 
\begin{cor}
  \mlabel{cor:j2} 
For $v = v_+ - v_-$ as in {\rm Proposition~\ref{prop:j2}}, the following are equivalent: 
\begin{itemize}
\item[\rm(i)] $\Aut(E)v \cap W(h_k) \not=\eset$ 
for $W(h_k)$ as in \eqref{eq:w-def}. 
\item[\rm(ii)] $\rk(v_+) = k$ and $\rk(v_-) = r-k$. 
\item[\rm(iii)] $v$ is invertible of index $\ind(v) = 2k-r$. 
\end{itemize}
In particular, $W(h_k) \subeq E^\times_{2k-r}$. 
\end{cor}

\begin{prf} The equivalence of (i) and (ii) follows from Proposition~\ref{prop:j2}. 
For the equivalence of (ii) and (iii), we note that 
$v = v_+ - v_-$ is invertible if and only if $\rk(v_+) +\rk(v_-) = r$. 
Then $\ind(v) = \rk(v_+) - \rk(v_-) = 2k-r$ is equivalent to 
$\rk(v_+) = k$ and $\rk(v_-) = r-k$. 
\end{prf}

\begin{ex} (a) For $k = r$ we have $E = E_1(h)$ and $W(h_r) = C^0$. 
Therefore $\Aut(E)v \subeq W(h_r)^c$ is equivalent to $\rk v_+ < r$, 
which is equivalent to $v \not\in C^0$. 

\nin (b) For $r = 2$ and $k = 1$ (Lorentz boost on Minkowski space), 
we obtain by $W(h_1)$ a wedge domain in the Minkowski space $E$. 
Then $\Aut(E)W(h_1) = E^\times_0$ is the open subset of spacelike vectors 
whose complement is the closed double cone $\bigcap_{g \in \Aut(E)} g W(h_1)^c = C \cup - C$ (cf.\ Example~\ref{ex:lorentz}). 
\end{ex}

\end{document}